\documentclass[12pt,a4paper]{amsart}
\usepackage[utf8]{luainputenc}
\usepackage{array}
\usepackage{float}
\usepackage{multirow}
\usepackage{amstext}
\usepackage{amsthm}
\usepackage{amssymb}
\usepackage{graphicx}
\usepackage[numbers]{natbib}
\usepackage[unicode=true,
 bookmarks=true,bookmarksnumbered=false,bookmarksopen=false,
 breaklinks=false,pdfborder={0 0 0},backref=false,colorlinks=false]
 {hyperref}
\hypersetup{pdftitle={Cyclotomic Aperiodic Substitution Tilings},
 pdfauthor={Stefan Pautze},
 pdfsubject={Abstract: The class of Cyclotomic Aperiodic Substitution Tilings (CAST) is introduced. Its vertices are supported on the 2n -th cyclotomic field. It covers a wide range of known aperiodic substitution tilings of the plane with finite rotations. Substitution matrices and minimal inflation multipliers of CASTs are discussed as well as practical use cases to identify specimen with individual dihedral symmetry D_{n}  or D_{2n} , i.e. the tiling contains an infinite number of patches of any size with dihedral symmetry D_{n}  or D_{2n}  only by iteration of substitution rules on a single tile.},
 pdfkeywords={cyclotomic aperiodic substitution tiling geometry quasiperiodic nonperiodic minimal inflation multiplier scaling factor matrix}}

\makeatletter

\pdfpageheight\paperheight
\pdfpagewidth\paperwidth

\providecommand{\tabularnewline}{\\}

\numberwithin{equation}{section}
\numberwithin{figure}{section}
  \theoremstyle{definition}
  \newtheorem{defn}{\protect\definitionname}[section]
  \theoremstyle{plain}
  \newtheorem{thm}{\protect\theoremname}[section]
  \theoremstyle{remark}
  \newtheorem{rem}{\protect\remarkname}[section]
  \theoremstyle{plain}
  \newtheorem{conjecture}{\protect\conjecturename}[section]

\usepackage{latexsym}
\usepackage{a4wide}
\pdfoutput=1

\RequirePackage{fix-cm}

\usepackage[T1]{fontenc}
\usepackage{array}
\usepackage{url}
\usepackage{multirow}
\usepackage{amsmath}
\usepackage{booktabs}
\usepackage{graphicx}

\newcommand{\T}{\ensuremath{{\mathcal T}}}

\renewcommand{\rho}{\varrho}
\renewcommand{\phi}{\varphi}

\theoremstyle{remark}

\renewcommand{\thefootnote}{}

\makeatother

  \providecommand{\conjecturename}{Conjecture}
  \providecommand{\definitionname}{Definition}
  \providecommand{\remarkname}{Remark}
\providecommand{\theoremname}{Theorem}

\begin{document}

\title{Cyclotomic Aperiodic Substitution Tilings}

\author{Stefan Pautze}

\address{Stefan Pautze (Visualien der Breitbandkatze), Am Mitterweg 1, 85309
Pörnbach, Germany}

\urladdr{\texttt{http://www.pautze.de}}

\email{\texttt{stefan@pautze.de}}
\begin{abstract}
The class of Cyclotomic Aperiodic Substitution Tilings (CAST) is introduced.
Its vertices are supported on the $2n$-th cyclotomic field. It covers
a wide range of known aperiodic substitution tilings of the plane
with finite rotations. Substitution matrices and minimal inflation
multipliers of CASTs are discussed as well as practical use cases
to identify specimen with individual dihedral symmetry $D_{n}$ or
$D_{2n}$, i.e. the tiling contains an infinite number of patches
of any size with dihedral symmetry $D_{n}$ or $D_{2n}$ only by iteration
of substitution rules on a single tile.
\end{abstract}

\maketitle

\section{Introduction}

Tilings have been subject of wide research. Many of their properties
are investigated and discussed in view of their application in physics
and chemistry, in detail the research of crystals and quasicrystals,
such as D. Shechtman et al.'s renowned Al-Mn-alloy with icosahedral
point group symmetry \citep{PhysRevLett.53.1951}. Tilings are also
of mathmatical interest on their own \citep{Grunbaum:1986:TP:19304,oro38933}.\\
Without any doubt they have great aesthetic qualities as well. Some
of the most impressing examples are the tilings in M. C. Escher’s
art works~\citep{escher1982m}, H. Voderberg’s spiral tiling \citep{Voderberg1936,Voderberg1937}
and the pentagonal tilings of A. Dürer, J. Kepler and R. Penrose \citep{Grunbaum:1986:TP:19304,Lueck2000263},
just to name a few.\\
In contrast to a scientist, a designer may have different requirements.
Nevertheless the result may have interesting mathematical properties.
Due to the lack of a general criteria for subjective matters of taste,
we consider the following properties as preferable:
\begin{itemize}
\item The tiling shall be aperiodic and repetitive (locally indistinguishable)
to have an interesting (psychedelic) appearance.
\item The tiling shall have a small inflation multiplier for reasons of
economy. Large inflation multipliers either require large areas to
be covered or many tiles of a small size to be used.
\item The tiling shall yield ``individual dihedral symmetry'' $D_{n}$
or $D_{2n}$ with $n\geq4$. I.e. it shall contain an infinite number
of patches of any size with dihedral symmetry only by iteration of
substitution rules on a single tile. \\
Similar to G. Maloney we demand symmetry of individual tilings and
not only symmetry of tiling spaces \citep{Maloney2014,DBLP:journals/dmtcs/Maloney15}. 
\end{itemize}
The most common methods to generate aperiodic tilings are:
\begin{itemize}
\item Matching rules, as introduced in the very first publication of an
aperiodic tiling, the Wang-tiling by R. Berger \citep{Berger1966}.
See \citep{Socolar1989,Socolar1990,Gaehler1993} for more details.
\item Cut-and-project scheme, first described by de Bruijn for the Penrose
tiling in \citep{NG1981}, later extended to a general method, see
\citep{lagarias1996} and \citep{Moody97} and references therein
for more details, notably the earlier general method described in
\citep{meyer1972}. 
\item Duals of multi grids as introduced by de Bruijn \citep{DEBRUIJNN.G.1986}
with equidistant spacings are equivalent to the cut-and-project scheme
\citep{0305-4470-19-2-020}.\\
Duals of multi grids with aperiodic spacings as introduced by Ingalls
\citep{Ingalls:sp0042,Ingalls1993177} are close related to Ammann
bars, see \citep{Grunbaum:1986:TP:19304,Socolar1989,LUeCK1993} and
\citep[Ch. 5]{scheffer1998} for details and examples.
\item The idea of substitution rules or substitutions in general is a rather
old concept, e.g. Koch’s snowflake \citep{HvK1904,HvK1906,Mandelbrot77a}
or Rep-Tilings \citep{Gardner1963}. However, it seems its first consequent
application to tile the whole Euclidean plane aperiodically appeared
with the Penrose tiling \citep{Penrose1974,Gardner1977,Penrose1979,Grunbaum:1986:TP:19304,oro38933}.
\end{itemize}
Among those methods substitution rules may be the easiest approach
to construct aperiodic tilings. Additionally they have some other
advantages:
\begin{itemize}
\item The inflation multipliers of tilings obtained by the cut-and-project
scheme are limited to PV-numbers. According to \citep{lagarias1996}
and \citep{Moody97} this was first noted by \citep{meyer1972}. This
limitation does not apply to substitution tilings.
\item Matching rules tend to be more complex than substitution rules. See
\citep{GoodmanStrauss1998} and \citep{goodman2003} for examples. 
\end{itemize}
In view of the preferable properties we will introduce the class of
Cyclotomic Aperiodic Substitution Tilings (CAST). Its vertices are
supported on the $2n$-th cyclotomic field. It covers a wide range
of known aperiodic substitution tilings of the plane with finite rotations.
Its properties, in detail substitution matrices, minimal inflation
multiplier and aperiodicity are discussed in Section~\ref{sec:Cyclotomic_Aperiodic_Substitution_Tilings}.
CASTs with minimal or at least small inflation multiplier are presented
in Sections~\ref{sec:CASTs_with_minimal_inflation_multiplier}~and~\ref{sec:CASTs_with_inflation_multiplier_=0000B5n,n-1/2_n_odd},
which also includes a generalization of the Lançon-Billard tiling.
Section~\ref{sec:Rhomic_CASTs_with_symmetric_edges_and_sunbstitution_rules}
focuses on several cases of rhombic CASTs and their minimal inflation
multiplier. In Section~\ref{sec:Gaps_to_Prototiles_Algorithm} the
“Gaps to Prototiles” algorithm is introduced, which allows to identify
large numbers of new CASTs. Finally, examples of Girih CASTs with
$n\in\left\{ 4,5,7\right\} $ are presented in Section~\ref{sec:CASTs_with_Extended_Girih_Prototiles}.
Except the generalized Lançon-Billard tiling all CASTs in this article
yield local dihedral symmetry $D_{n}$ or $D_{2n}$. 

For terms and definitions we stay close to \citep{oro38933} and \citep{HFonl}:
\begin{itemize}
\item A ``tile'' in $\mathbb{\mathbb{\mathbb{R}}}^{d}$ is defined as
a nonempty compact subset of $\mathbb{\mathbb{\mathbb{R}}}^{d}$ which
is the closure of its interior.
\item A ``tiling'' in $\mathbb{\mathbb{\mathbb{R}}}^{d}$ is a countable
set of tiles, which is a covering as well as a packing of $\mathbb{\mathbb{\mathbb{R}}}^{d}$.
The union of all tiles is $\mathbb{\mathbb{\mathbb{R}}}^{d}$. The
intersection of the interior of two different tiles is empty.
\item A ``patch'' is a finite subset of a tiling.
\item A tiling is called ``aperiodic'' if no translation maps the tiling
to itself.
\item ``Prototiles'' serve as building blocks for a tiling.
\item Within this article the term ``substitution'' means, that a tile
is expanded with a linear map - the ``inflation multiplier'' - and
dissected into copies of prototiles in original size - the ``substitution
rule''.
\item A ``supertile'' is the result of one or more substitutions, applied
to a single tile. Within this article we use the term for one substitutions
only. 
\item We use $\zeta_{n}^{k}$ to denote the $n$-th roots of unity so that
$\zeta_{n}^{k}=e^{\frac{2\mathtt{i}k\pi}{n}}$and its complex conjugate
$\overline{\zeta_{n}^{k}}=e^{-\frac{2\mathtt{i}k\pi}{n}}$.
\item $\mathbb{Q}\left(\zeta_{n}\right)$ denotes the $n$-th cyclotomic
field. Please note that $\mathbb{Q}\left(\zeta_{n}\right)=\mathbb{Q}\left(\zeta_{2n}\right)$
for $odd\;n$.
\item The maximal real subfield of $\mathbb{Q}\left(\zeta_{n}\right)$ is
$\mathbb{Q}\left(\zeta_{n}+\overline{\zeta_{n}}\right)$.
\item $\mathbb{Z}\left[\zeta_{n}\right]$ denotes the the ring of algebraic
integers in $\mathbb{Q}\left(\zeta_{n}\right)$.
\item $\mathbb{Z}\left[\zeta_{n}+\overline{\zeta_{n}}\right]$ denotes the
the ring of algebraic integers (which are real numbers) in $\mathbb{Q}\left(\zeta_{n}+\overline{\zeta_{n}}\right)$. 
\item We use $\mu_{n,k}$ to denote the $k$-th diagonal of a regular $n$-gon
with side length $\mu_{n,1}=\mu_{n,n-1}=1$.
\item $\mathbb{Z}\left[\mu_{n}\right]=\mathbb{Z}\left[\mu_{n,1},\mu_{n,2},\mu_{n,3}...\mu_{n,\left\lfloor n/2\right\rfloor }\right]$
denotes the ring of the diagonals of a regular $n$-gon.
\end{itemize}
\clearpage{}

\section{\label{sec:Cyclotomic_Aperiodic_Substitution_Tilings}Properties
of Cyclotomic Aperiodic Substitution Tilings}

We define Cyclotomic Aperiodic Substitution Tilings (CASTs) similar
to the concept of Cyclotomic Model Sets as described in \citep[Ch. 7.3]{oro38933}.
\begin{defn}
A (substitution) tiling $\mathcal{T}$ in the complex plane is cyclotomic
if the coordinates of all vertices are algebraic integers in $\mathbb{Z}\left[\zeta_{2n}\right]$,
i.e. an integer sum of the $2n$-th roots of unity. As a result all
vertices of all substituted prototiles and the inflation multiplier
are algebraic integers as well. \end{defn}
\begin{thm}
\label{thm: CAST1} For the case that all prototiles $P_{k}$ of a
CAST $\mathcal{T}$ have areas equal to $A_{k}=A(P_{k})=R\cdot\sin\left(\frac{k\pi}{n}\right)$,
$R\in\mathbb{R}^{>0}$ , $k,n\in\mathbb{\mathbb{N}}$ , $0<k<n$ we
can use a given inflation multiplier $\eta$ to calculate the substitution
matrix if $|\eta^{2}|$ can be written as $|\eta^{2}|={\displaystyle {\textstyle \sum_{k=1}^{\left\lfloor n/2\right\rfloor }}}c_{k}\mu_{n,k}$,
$c_{k}\in\mathbb{\mathbb{N}}_{0}$, $\max\left(c_{k}\right)>0)$,
$\mu_{n,k}=\frac{\sin\left(\frac{k\pi}{n}\right)}{\sin\left(\frac{\pi}{n}\right)}$
and the conditions $|\eta^{2}|\notin\mathbb{N}$\textup{ and $\min((\max(c_{k}),\;odd\;k),(\max(c_{k}),\;even\;k))\geq1,\;(even\;n>4)$}
are met. (For simplification, all prototiles with the same area are
combined.)
\end{thm}
The inflation multiplier $\eta$ of such a substitution tiling can
be written as a sum of $2n$-th roots of unity:

\begin{equation}
\eta=\sum_{k=0}^{n-1}a_{k}\zeta_{2n}^{k}\;\;\;\;\;\;(a_{k}\in\mathbb{Z},\;\max(|a_{k}|)>0)\label{eq:lambda-i}
\end{equation}

Please note that there are multiple ways to describe $\eta$. We assume
$a_{k}$ to be chosen so that the sum is irreducible, i.e. $\sum_{k=0}^{n-1}|a_{k}|$
is minimal. 

A substitution tiling with $l\geq2$, $l\in\mathbb{N}$ prototiles
and substitution rules is partially characterized by its substitution
matrix $M\in\mathbb{\mathbb{N}}_{0}^{l\times l}$ with an eigenvalue
$\lambda$ and an eigenvector $x_{A}$. 

\begin{equation}
\lambda\,x_{A}=M\,x_{A}
\end{equation}

The elements of the right eigenvector $x_{A}$ contain the areas of
the prototiles $A_{k}=A(P_{k})$. Since $M\in\mathbb{\mathbb{N}}_{0}^{l\times l}$,
the elements of $x_{A}$ generate a ring of algebraic integers which
are real numbers.

The elements of the left eigenvector $x_{f}$ represent the frequencies
of the prototiles $f_{k}=f(P_{k})$, so that:

\[
\lambda\,x_{f}^{T}=x_{f}^{T}\,M
\]
\begin{equation}
\lambda\,x_{f}=M^{T}\,x_{f}\label{eq:eigenvector-frequency}
\end{equation}

The eigenvalue $\lambda$ can be interpreted as inflation multiplier
regarding the areas during a substitution. In the complex plane we
can conclude:

\begin{equation}
\lambda=|\eta^{2}|=\eta\cdot\overline{\eta}\;\;\;\;\;\;(\lambda\in\mathbb{\mathbb{R}})\label{eq:lambda-a-square}
\end{equation}

With \eqref{eq:lambda-i} and \eqref{eq:lambda-a-square} the eigenvalue
$\lambda$ can also be written as a sum of $2n$-th roots of unity.
In other words, the elements of the eigenvector $x_{A}$ span a ring
of algebraic integers which are real numbers. This ring is isomorphic
to the ring of algebraic integers $\mathbb{Z}\left[\zeta_{2n}+\overline{\zeta_{2n}}\right]$
in $\mathbb{Q}\left(\zeta_{2n}+\overline{\zeta_{2n}}\right)$, which
is the maximal real subfield of the cyclotomic field $\mathbb{Q}\left(\zeta_{2n}\right)$.
\begin{equation}
\lambda=b_{0}+\sum_{k=1}^{\left\lfloor (n-1)/2\right\rfloor }b_{k}\left(\zeta_{2n}^{k}+\overline{\zeta_{2n}^{k}}\right)\;\;\;\;\;\;(b_{0},b_{k}\in\mathbb{Z})\label{eq:lambda-a-2}
\end{equation}

\begin{equation}
b_{0}=\sum_{k=0}^{n-1}a_{k}^{2}
\end{equation}

Because of the conditions regarding $a_{k}$ we ensure that no combination
of roots of unity in the right part of the Equation~\eqref{eq:lambda-a-2}
can sum up to a integer which is a real number. Since we need at least
two roots of unity to have an inflation multiplier $\eta>1$, it follows
that 

\begin{equation}
b_{0}\geq2.\label{eq:mindest-b0}
\end{equation}

The length of the $k$-th diagonal $\mu_{n,k}$ of a regular $n$-gon
with side length 1 can be written as absolute value of a sum of $2n$-th
roots of unity. Note that $\mu_{n,1}=1$ because it refers to a single
side of the regular $n$-gon.

\begin{equation}
\mu_{n,1}=\zeta_{2n}^{0}=1\label{eq:diagonal-to-sum-of-n-roots-of-unity}
\end{equation}

\[
\mu_{n,2}=\zeta_{2n}^{1}+\zeta_{2n}^{-1}=\zeta_{2n}^{1}+\overline{\zeta_{2n}^{1}}
\]

\[
\mu_{n,3}=\zeta_{2n}^{-2}+\zeta_{2n}^{0}+\zeta_{2n}^{2}=\zeta_{2n}^{0}+\zeta_{2n}^{2}+\overline{\zeta_{2n}^{2}}\;\;\;\;\;\;(n>4)
\]

\[
\mu_{n,4}=\zeta_{2n}^{3}+\zeta_{2n}^{1}+\zeta_{2n}^{-1}+\zeta_{2n}^{-3}=\zeta_{2n}^{1}+\overline{\zeta_{2n}^{1}}+\zeta_{2n}^{3}+\overline{\zeta_{2n}^{3}}\;\;\;\;\;\;(n>5)
\]

\begin{equation}
\mu_{n,k}=\sum_{i=0}^{k-1}\zeta_{2n}^{2i-k+1}\;\;\;\;\;\;(n>k\geq1)\label{eq:diagonal-to-sum-of-n-roots-of-unity-2}
\end{equation}

This also works vice versa for $\zeta_{2n}^{k}+\overline{\zeta_{2n}^{k}}$:

\begin{equation}
\zeta_{2n}^{k}+\overline{\zeta_{2n}^{k}}=\mu_{n,k+1}-\mu_{n,k-1}\;\;\;\;\;\;(\left\lfloor n/2\right\rfloor \geq k>1)\label{eq:n-roots-of-unity-to-sum-of-diagonal}
\end{equation}

As a result, the eigenvalue $\lambda$ can also be written as a sum
of diagonals $\mu_{n,k}$:

\begin{equation}
\lambda=\sum_{k=1}^{\left\lfloor n/2\right\rfloor }c_{k}\mu_{n,k}\;\;\;\;\;\;(c_{k}\in\mathbb{\mathbb{Z}};\;\max(\left|c_{k}\right|)>0)\label{eq:lambda-a-3}
\end{equation}

We recall the areas $A_{n,k}$ of isosceles triangles with vertex
angles equal $\frac{k\pi}{n}$ spanned by roots of unity:

\begin{equation}
A_{n,k}=R\;\sin\left(\frac{k\pi}{n}\right)\;\;\;\;\;\;\left(R=\frac{1}{2}\right)\label{eq:triangle-area}
\end{equation}

We furthermore recall that the length of the $k$-th diagonal $\mu_{n,k}$
of a regular $n$-gon with side length 1 can also be written as:

\begin{equation}
\mu_{n,k}=\frac{\sin\left(\frac{k\pi}{n}\right)}{\sin\left(\frac{\pi}{n}\right)}\label{eq:diagonal-of-ngon}
\end{equation}

We recall the Diagonal Product Formula (DPF) as described in \citep{Nischke1996}
and \citep{zbMATH01071804} with some small adaptions:

\[
\mu_{n,1}\mu_{n,k}=\mu_{n,k}
\]

\[
\mu_{n,2}\mu_{n,k}=\mu_{n,k-1}+\mu_{n,k+1}\;\;\;\;\;\;(1<k\leq\left\lfloor n/2\right\rfloor )
\]

\[
\mu_{n,3}\mu_{n,k}=\mu_{n,k-2}+\mu_{n,k}+\mu_{n,k+2}\;\;\;\;\;\;(2<k\leq\left\lfloor n/2\right\rfloor )
\]

\[
\mu_{n,4}\mu_{n,k}=\mu_{n,k-3}+\mu_{n,k-1}+\mu_{n,k+1}+\mu_{n,k+3}\;\;\;\;\;\;(3<k\leq\left\lfloor n/2\right\rfloor )
\]

Or, more generally:

\begin{equation}
\mu_{n,h}\mu_{n,k}=\sum_{i=1}^{h}{\displaystyle \mu_{n,k-h-1+2i}}\;\;\;\;\;\;(1\leq h\leq k\leq\left\lfloor n/2\right\rfloor )\label{eq:goldenfield}
\end{equation}

The other diagonals are defined by:

\begin{equation}
\mu_{n,n-k}=\mu_{n,k}\;\;\;\;\;\;(1\leq k\leq n-1)\label{eq:diagonal-equivalenz}
\end{equation}

and

\begin{equation}
\sin\left(\frac{\left(n-k\right)\pi}{n}\right)=\sin\left(\frac{k\pi}{n}\right).\label{eq:diagonal-equivalenz-2}
\end{equation}

As a result the diagonals of a $n$-gon span a ring of diagonals $\mathbb{Z}\left[\mu_{n}\right]$.
With Equations~\eqref{eq:diagonal-to-sum-of-n-roots-of-unity}, \eqref{eq:diagonal-to-sum-of-n-roots-of-unity-2}
and~\eqref{eq:n-roots-of-unity-to-sum-of-diagonal} can be shown
that:

\begin{equation}
\mathbb{Z}\left[\mu_{n}\right]=\mathbb{Z}\left[\zeta_{2n}+\overline{\zeta_{2n}}\right]
\end{equation}

\begin{equation}
\lambda\in\mathbb{Z}\left[\mu_{n}\right]\label{eq:lambda-a-3-1}
\end{equation}

Because of Equations~\eqref{eq:lambda-a-3}, \eqref{eq:triangle-area},
\eqref{eq:diagonal-of-ngon} and~\eqref{eq:goldenfield}, we choose
the substitution matrix as $M\in\mathbb{\mathbb{N}}_{0}^{\left(n-1\right)\times\left(n-1\right)}$
and the eigenvector as: 

\begin{equation}
x_{A}=\begin{pmatrix}\mu_{n,n-1}\\
\vdots\\
\mu_{n,2}\\
\mu_{n,1}
\end{pmatrix}=\frac{1}{A_{n,1}}\begin{pmatrix}A_{n,n-1}\\
\vdots\\
A_{n,2}\\
A_{n,1}
\end{pmatrix}\label{eq:eigenvektor}
\end{equation}

With Equation~\eqref{eq:goldenfield} and~\eqref{eq:eigenvektor},
we can find a matrix $M_{n,k},\;\left\lfloor n/2\right\rfloor \geq k\geq1$
with eigenvalue $\mu_{n,k}$ for the given eigenvector $x_{A}$:

\begin{equation}
M_{n,1}=E=\left(\begin{array}{ccccccc}
1 & 0 & \cdots & \cdots & \cdots & \cdots & 0\\
0 & 1 & 0 &  &  &  & \vdots\\
\vdots & 0 & 1 & \ddots &  &  & \vdots\\
\vdots &  & \ddots & \ddots & \ddots &  & \vdots\\
\vdots &  &  & \ddots & 1 & 0 & \vdots\\
\vdots &  &  &  & 0 & 1 & 0\\
0 & \cdots & \cdots & \cdots & \cdots & 0 & 1
\end{array}\right)\label{eq:M-allgemein}
\end{equation}

\[
M_{n,2}=\left(\begin{array}{ccccccc}
0 & 1 & 0 & \cdots & \cdots & \cdots & 0\\
1 & 0 & 1 & \text{0} &  &  & \vdots\\
0 & 1 & 0 & \ddots & \ddots &  & \vdots\\
\vdots & 0 & \ddots & \ddots & \ddots & 0 & \vdots\\
\vdots &  & \ddots & \ddots & 0 & 1 & 0\\
\vdots &  &  & \text{0} & 1 & 0 & 1\\
0 & \cdots & \cdots & \cdots & 0 & 1 & 0
\end{array}\right)
\]

\[
M_{n,3}=\left(\begin{array}{ccccccc}
0 & 0 & 1 & 0 & \cdots & \cdots & 0\\
0 & 1 & 0 & 1 & \ddots &  & \vdots\\
1 & 0 & 1 & \ddots & \ddots & \ddots & \vdots\\
0 & 1 & \ddots & \ddots & \ddots & 1 & 0\\
\vdots & \ddots & \ddots & \ddots & 1 & 0 & 1\\
\vdots &  & \ddots & 1 & 0 & 1 & 0\\
0 & \cdots & \cdots & 0 & 1 & 0 & 0
\end{array}\right)
\]

\[
\cdots
\]

To get the substitution matrix for a given eigenvalue $\lambda$ as
defined in Equation~\eqref{eq:lambda-a-3} we just need to sum up
the matrices $M_{n,k}$ with the coefficients $c_{k}$ :

\begin{equation}
\lambda\:x_{A}=\left(\sum_{k=1}^{\left\lfloor n/2\right\rfloor }c_{k}\mu_{n,k}\right)\:x_{A}=\left(\sum_{k=1}^{\left\lfloor n/2\right\rfloor }c_{k}M_{n,k}\right)\:x_{A}=M_{n}\:x_{A}
\end{equation}

\begin{equation}
M_{n}=\sum_{k=1}^{\left\lfloor n/2\right\rfloor }c_{k}M_{n,k}
\end{equation}

We recall that a substitution matrix $M_{n}$ must be primitive and
real positive. To ensure the latter property its eigenvalue $\lambda\in\mathbb{\mathbb{R}}$
must be a positive sum of elements of the eigenvector $x_{A}$. For
this reason, we have to modify Equation~\eqref{eq:lambda-a-3} accordingly:

\begin{equation}
\lambda=\sum_{k=1}^{\left\lfloor n/2\right\rfloor }c_{k}\mu_{n,k}\;\;\;\;\;\;(c_{k}\in\mathbb{\mathbb{N}}_{0};\;\max(c_{k})>0)\label{eq:lambda-a-1}
\end{equation}

With Equation~\eqref{eq:goldenfield} it can be shown that the $\lambda$
with positive coefficients $c_{k}$ in Equation~\eqref{eq:lambda-a-1}
span a commutative semiring $\mathbb{N}_{0}\left[\mu_{n}\right]$
which is a subset of $\mathbb{Z}\left[\mu_{n}\right]$:

\begin{equation}
\lambda\in\mathbb{\mathbb{\mathbb{N}}}_{0}\left[\mu_{n}\right]\subset\mathbb{\mathbb{Z}}\left[\mu_{n}\right]\label{eq:lambda-a-1-1}
\end{equation}

Since all $M_{n,k}$ are symmetric, this is also true for $M_{n}$.
In this case, left and right eigenvector of $M_{n}$ are equal so
that $x_{A}=x_{f}$ and $x_{A}$ represents the frequencies of prototiles
as well as their areas.

Because of Equations~\eqref{eq:diagonal-equivalenz} and~\eqref{eq:diagonal-equivalenz-2},
the matrices and the eigenvector can be reduced so that $M_{n};\:M_{n,k}\in\mathbb{\mathbb{N}}_{0}^{\left\lfloor n/2\right\rfloor \times\left\lfloor n/2\right\rfloor }$
and $x_{A}\in\mathbb{\mathbb{\mathbb{R}}}^{\left\lfloor n/2\right\rfloor }$.
In detail we omit the redundant elements of $x_{A}$ and rows of $M_{n}$
and replace the orphaned entries in $M_{n}$ with valid ones. This
also enforces $\left\lfloor n/2\right\rfloor \geq k\geq1$. The results
of the reduction are shown in Equations~\eqref{eq:M-odd} and~\eqref{eq:M-even}. 

For $x_{A}$ we can write: 
\begin{equation}
x_{A}=\begin{pmatrix}\mu_{n,\left\lfloor n/2\right\rfloor }\\
\vdots\\
\mu_{n,2}\\
\mu_{n,1}
\end{pmatrix}=\frac{1}{A_{n,1}}\begin{pmatrix}A_{n,\left\lfloor n/2\right\rfloor }\\
\vdots\\
A_{n,2}\\
A_{n,1}
\end{pmatrix}\label{eq:eigenvektor-1}
\end{equation}

For $odd\;n$ the substitution matrices $M_{n}$ can be described
by the following scheme:

\begin{equation}
M_{5}=c_{2}\left(\begin{array}{cc}
1 & 1\\
1 & 0
\end{array}\right)+c_{1}\left(\begin{array}{cc}
1 & 0\\
0 & 1
\end{array}\right)\label{eq:M-odd}
\end{equation}

\[
M_{7}=c_{3}\left(\begin{array}{ccc}
1 & 1 & 1\\
1 & 1 & 0\\
1 & 0 & 0
\end{array}\right)+c_{2}\left(\begin{array}{ccc}
1 & 1 & 0\\
1 & 0 & 1\\
0 & 1 & 0
\end{array}\right)+c_{1}\left(\begin{array}{ccc}
1 & 0 & 0\\
0 & 1 & 0\\
0 & 0 & 1
\end{array}\right)
\]

\[
M_{9}=c_{4}\left(\begin{array}{cccc}
1 & 1 & 1 & 1\\
1 & 1 & 1 & 0\\
1 & 1 & 0 & 0\\
1 & 0 & 0 & 0
\end{array}\right)+c_{3}\left(\begin{array}{cccc}
1 & 1 & 1 & 0\\
1 & 1 & 0 & 1\\
1 & 0 & 1 & 0\\
0 & 1 & 0 & 0
\end{array}\right)+c_{2}\left(\begin{array}{cccc}
1 & 1 & 0 & 0\\
1 & 0 & 1 & 0\\
0 & 1 & 0 & 1\\
0 & 0 & 1 & 0
\end{array}\right)+c_{1}\left(\begin{array}{cccc}
1 & 0 & 0 & 0\\
0 & 1 & 0 & 0\\
0 & 0 & 1 & 0\\
0 & 0 & 0 & 1
\end{array}\right)
\]

\[
M_{11}=c_{5}\left(\begin{array}{ccccc}
1 & 1 & 1 & 1 & 1\\
1 & 1 & 1 & 1 & 0\\
1 & 1 & 1 & 0 & 0\\
1 & 1 & 0 & 0 & 0\\
1 & 0 & 0 & 0 & 0
\end{array}\right)+c_{4}\left(\begin{array}{ccccc}
1 & 1 & 1 & 1 & 0\\
1 & 1 & 1 & 0 & 1\\
1 & 1 & 0 & 1 & 0\\
1 & 0 & 1 & 0 & 0\\
0 & 1 & 0 & 0 & 0
\end{array}\right)+c_{3}\left(\begin{array}{ccccc}
1 & 1 & 1 & 0 & 0\\
1 & 1 & 0 & 1 & 0\\
1 & 0 & 1 & 0 & 1\\
0 & 1 & 0 & 1 & 0\\
0 & 0 & 1 & 0 & 0
\end{array}\right)+
\]

\[
c_{2}\left(\begin{array}{ccccc}
1 & 1 & 0 & 0 & 0\\
1 & 0 & 1 & 0 & 0\\
0 & 1 & 0 & 1 & 0\\
0 & 0 & 1 & 0 & 1\\
0 & 0 & 0 & 1 & 0
\end{array}\right)+c_{1}\left(\begin{array}{ccccc}
1 & 0 & 0 & 0 & 0\\
0 & 1 & 0 & 0 & 0\\
0 & 0 & 1 & 0 & 0\\
0 & 0 & 0 & 1 & 0\\
0 & 0 & 0 & 0 & 1
\end{array}\right)
\]

The scheme can be continued for any $odd\;n$. The matrices are still
symmetric, so that

\begin{equation}
x_{f}=x_{A}\;\;\;\;\;\;(odd\;n)\label{eq:xf_n_odd}
\end{equation}
and $x_{A}$ still represents the frequencies of prototiles.

For $even\;n$, the substitution matrices $M_{n}$ be described by
the following scheme:

\begin{equation}
M_{4}=c_{2}\left(\begin{array}{cc}
0 & 2\\
1 & 0
\end{array}\right)+c_{1}\left(\begin{array}{cc}
1 & 0\\
0 & 1
\end{array}\right)\label{eq:M-even}
\end{equation}

\[
M_{6}=c_{3}\left(\begin{array}{ccc}
1 & 0 & 2\\
0 & 1 & 0\\
1 & 0 & 0
\end{array}\right)+c_{2}\left(\begin{array}{ccc}
0 & 2 & 0\\
1 & 0 & 1\\
0 & 1 & 0
\end{array}\right)+c_{1}\left(\begin{array}{ccc}
1 & 0 & 0\\
0 & 1 & 0\\
0 & 0 & 1
\end{array}\right)
\]

\[
M_{8}=c_{4}\left(\begin{array}{cccc}
0 & 2 & 0 & 2\\
1 & 0 & 2 & 0\\
0 & 2 & 0 & 0\\
1 & 0 & 0 & 0
\end{array}\right)+c_{3}\left(\begin{array}{cccc}
1 & 0 & 2 & 0\\
0 & 2 & 0 & 1\\
1 & 0 & 1 & 0\\
0 & 1 & 0 & 0
\end{array}\right)+c_{2}\left(\begin{array}{cccc}
0 & 2 & 0 & 0\\
1 & 0 & 1 & 0\\
0 & 1 & 0 & 1\\
0 & 0 & 1 & 0
\end{array}\right)+c_{1}\left(\begin{array}{cccc}
1 & 0 & 0 & 0\\
0 & 1 & 0 & 0\\
0 & 0 & 1 & 0\\
0 & 0 & 0 & 1
\end{array}\right)
\]

\[
M_{10}=c_{5}\left(\begin{array}{ccccc}
1 & 0 & 2 & 0 & 2\\
0 & 2 & 0 & 2 & 0\\
1 & 0 & 2 & 0 & 0\\
0 & 2 & 0 & 0 & 0\\
1 & 0 & 0 & 0 & 0
\end{array}\right)+c_{4}\left(\begin{array}{ccccc}
0 & 2 & 0 & 2 & 0\\
1 & 0 & 2 & 0 & 1\\
0 & 2 & 0 & 1 & 0\\
1 & 0 & 1 & 0 & 0\\
0 & 1 & 0 & 0 & 0
\end{array}\right)+c_{3}\left(\begin{array}{ccccc}
1 & 0 & 2 & 0 & 0\\
0 & 2 & 0 & 1 & 0\\
1 & 0 & 1 & 0 & 1\\
0 & 1 & 0 & 1 & 0\\
0 & 0 & 1 & 0 & 0
\end{array}\right)+
\]

\[
c_{2}\left(\begin{array}{ccccc}
0 & 2 & 0 & 0 & 0\\
1 & 0 & 1 & 0 & 0\\
0 & 1 & 0 & 1 & 0\\
0 & 0 & 1 & 0 & 1\\
0 & 0 & 0 & 1 & 0
\end{array}\right)+c_{1}\left(\begin{array}{ccccc}
1 & 0 & 0 & 0 & 0\\
0 & 1 & 0 & 0 & 0\\
0 & 0 & 1 & 0 & 0\\
0 & 0 & 0 & 1 & 0\\
0 & 0 & 0 & 0 & 1
\end{array}\right)
\]

The scheme can be continued for any $even\;n$, but the matrices are
not symmetric anymore. With Equations~\eqref{eq:eigenvector-frequency},
\eqref{eq:M-allgemein} and~\eqref{eq:M-even}, we can derive: 

\begin{equation}
x_{f}=\begin{pmatrix}\mu_{n,n-1}\\
2\mu_{n,n-2}\\
\vdots\\
2\mu_{n,2}\\
2\mu_{n,1}
\end{pmatrix}=\frac{1}{A_{n,1}}\begin{pmatrix}A_{n,n-1}\\
2A_{n,n-2}\\
\vdots\\
2A_{n,2}\\
2A_{n,1}
\end{pmatrix}\;\;\;\;\;\;(even\;n)\label{eq:xf_n_even}
\end{equation}

\begin{rem}
For a given matrix $M_{n}$ of a CAST $\mathcal{T}$ as defined above
the eigenvalue $\lambda$ can be calculated easily with Equation~\eqref{eq:lambda-a-1}.
From Equations~\eqref{eq:M-odd} and~\eqref{eq:M-even} follows
that: 

\begin{equation}
M_{n}=\left[\begin{array}{cccc}
\vdots &  & \vdots & \vdots\\
c_{\left\lfloor n/2\right\rfloor } & \cdots & c_{2} & c_{1}
\end{array}\right]
\end{equation}

In other words, the coefficients $c_{k}$ can be read directly from
the bottom line of matrix $M_{n}$.

We also can find a vector notation for Equation~\eqref{eq:lambda-a-1}: 

\begin{equation}
\lambda=c^{T}x_{A}\label{eq:lambda-a-1-vector}
\end{equation}

with

\begin{equation}
c=\left(\begin{array}{c}
c_{\left\lfloor n/2\right\rfloor }\\
\vdots\\
c_{2}\\
c_{1}
\end{array}\right)\;\;\;\;\;\;(c_{k}\in\mathbb{\mathbb{N}}_{0};\;\max(c_{k})>0)\label{eq:coefficient_vector}
\end{equation}

\end{rem}
For $n\in\left\{ 2,3\right\} $ the $M_{n}=c_{1}M_{n,1}\in\mathbb{\mathbb{N}};\:\lambda\in\mathbb{N}$
and $x_{A};\:x_{F}\in\mathbb{\mathbb{\mathbb{R}}}^{>0}$. As a result
$M_{n}$ is always primitive.

For the following discussion of cases with $n\geq4$ we define some
sets of eigenvalues (and its powers) based on Equation~\eqref{eq:lambda-a-1}:

\begin{equation}
S_{even}=\left\{ {\displaystyle \sum_{k=1}^{\left\lfloor n/2\right\rfloor }}c_{k}\mu_{n,k}\mid c_{k}\in\mathbb{\mathbb{N}}_{0};\;c_{k}=0\;(odd\;k);\;\max(c_{k})>0\right\} 
\end{equation}

\begin{equation}
S_{odd}=\left\{ {\displaystyle \sum_{k=1}^{\left\lfloor n/2\right\rfloor }}c_{k}\mu_{n,k}\mid c_{k}\in\mathbb{\mathbb{N}}_{0};\;c_{k}=0\;(even\;k);\;\max(c_{k})>0\right\} 
\end{equation}

\begin{equation}
S_{mixed}=\overline{S_{odd}\cup S_{even}}\setminus\left\{ 0\right\} 
\end{equation}

\begin{equation}
S_{full}=\left\{ {\displaystyle \sum_{k=1}^{\left\lfloor n/2\right\rfloor }}c_{k}\mu_{n,k}\mid c_{k}\in\mathbb{N};\;c_{k}>\text{0}\right\} 
\end{equation}

\begin{equation}
S_{full}\subset S_{mixed}
\end{equation}

\begin{equation}
S_{even}\cap S_{odd}=S_{odd}\cap S_{mixed}=S_{mixed}\cap S_{even}=\emptyset
\end{equation}

Equations~\eqref{eq:M-odd} and~\eqref{eq:M-even} show that substitution
matrix $M_{n}$ is strictly positive $\left(M_{n}\gg0\right)$ and
so also primitive if the corresponding eigenvalue $\lambda\in S_{full}$.
Substitution matrix $M_{n}$ is also primitive if some finite power
of eigenvalue $\lambda^{p}\in S_{full}\;(p\in\mathbb{N})$.

For $odd\;n\geq5$ Equations~\eqref{eq:goldenfield} and~\eqref{eq:diagonal-equivalenz}
show that some finite power of diagonals of a regular $n$-gon exist
so that: 

\begin{equation}
\exists p\left(\mu_{n,k}^{p}\in S_{full}\right)\;\;\;\;\;\;\left(k\geq2;\right)
\end{equation}

It is also possible to show that some finite power of an eigenvalue
$\lambda\in\mathbb{\mathbb{\mathbb{N}}}_{0}\left[\mu_{n}\right]\setminus\mathbb{N}_{\text{0}}$
exist so that:

\begin{equation}
\exists p\left(\lambda^{p}\in S_{full}\right)\;\;\;\;\;\;\left(\lambda\in\mathbb{\mathbb{\mathbb{N}}}_{0}\left[\mu_{n}\right]\setminus\mathbb{N}_{\text{0}}\right)
\end{equation}

As a result any substitution matrix $M_{n}$ in Equation~\eqref{eq:M-odd}
with an eigenvalue $\lambda\in\mathbb{\mathbb{\mathbb{N}}}_{0}\left[\mu_{n}\right]\setminus\mathbb{N}_{\text{0}}$
is primitive.

For $even\;n\geq4$ Equations~\eqref{eq:goldenfield} and~\eqref{eq:diagonal-equivalenz}
show that the powers of diagonals of a regular $n$-gon have the following
porperties:

\begin{equation}
\mu_{n,k}^{p}\in S_{odd}\;\;\;\;\;\;\left(odd\;k\right)
\end{equation}

\begin{equation}
\mu_{n,k}^{p}\in S_{odd}\;\;\;\;\;\;\left(even\;k;\;even\;p\right)
\end{equation}

\begin{equation}
\mu_{n,k}^{p}\in S_{even}\;\;\;\;\;\;\left(even\;k;\;odd\;p\right)
\end{equation}

\begin{equation}
\left(\mu_{n,k}+\mu_{n,h}\right)^{p}\in S_{mixed}\;\;\;\;\;\;\left(even\;k;\;odd\;h\right)
\end{equation}

\begin{equation}
\exists p\left(\left(\mu_{n,k}+\mu_{n,h}\right)^{p}\in S_{full}\right)\;\;\;\;\;\;\left(even\;k;\;odd\;h\right)
\end{equation}

With the same equations it is also possible to show that:

\begin{equation}
\lambda^{p}\in S_{odd}\;\;\;\;\;\;\left(\lambda\in S_{odd}\right)
\end{equation}

\begin{equation}
\lambda^{p}\in S_{odd}\;\;\;\;\;\;\left(\lambda\in S_{even};\;even\;p\right)
\end{equation}

\begin{equation}
\lambda^{p}\in S_{even}\;\;\;\;\;\;\left(\lambda\in S_{even};\;odd\;p\right)
\end{equation}

\begin{equation}
\lambda^{p}\in S_{mixed}\;\;\;\;\;\;\left(\lambda\in S_{mixed}\setminus\mathbb{N}_{\text{0}}\right)
\end{equation}

It is also possible to show that some finite power of an eigenvalue
$\lambda\in S_{mixed}\setminus\mathbb{N}_{\text{0}}$ exist so that:

\begin{equation}
\exists p\left(\lambda^{p}\in S_{full}\right)\;\;\;\;\;\;\left(\lambda\in S_{mixed}\setminus\mathbb{N}_{\text{0}}\right)
\end{equation}

As a result any substitution matrix $M_{n}$ in Equation~\eqref{eq:M-even}
with an eigenvalue $\lambda\in S_{mixed}\setminus\mathbb{\mathbb{N}_{\text{0}}}$
is primitive. That follows that the coefficients $c_{k}$ have to
be chosen so that:

\begin{equation}
\min\left(\left(\max\left(c_{k}\right),\;odd\;k\right),\left(\max\left(c_{k}\right),\;even\;k\right)\right)\geq1\;\;\;\;\;\;(even\;n\geq4)\label{Bedingung_f=0000FCr_Mn_primitiv_falls_n_gerade}
\end{equation}

\begin{rem}
Every substitution tiling defines a substitution matrix. But not for
every substitution matrix $M_{n}$ exists a substitution tiling. 
\end{rem}

\begin{thm}
\label{thm:CASTs-Aperiodic}CASTs $\mathcal{T}$ in Theorem~\ref{thm: CAST1}
with $n\geq4$ are aperiodic under the condition that at least one
substituion rule contains a pair of copies of a prototile whose rotational
orientations differs in $\frac{k\pi}{n};\:n>k>0;\:k\in\mathbb{\mathbb{N}}$.\end{thm}
\begin{proof}
We recall that any substitution matrix $M_{n}$ of CASTs $\mathcal{T}$
in Theorem~\ref{thm: CAST1} is primitive and that every CAST $\mathcal{T}$
yields finite rotations.

Under this conditions \citep[Theorem 2.3]{2008PMag...88.2033F} applies,
so that all orientations of a prototile appear in the same frequency. 

We furthermore recall areas and frequencies of prototiles of CASTs
$\mathcal{T}$ in Theorem~\ref{thm: CAST1} are given by right and
left eigenvector $x_{A}$ and $x_{f}$.

With Equations~\eqref{eq:xf_n_odd} and~\eqref{eq:xf_n_even} the
ratio between the frequencies of prototiles $P_{2}$ and $P_{1}$
is given by: 

\begin{equation}
\frac{f(P_{2})}{f(P_{1})}=\frac{\text{\ensuremath{\mu}}_{n,2}}{\text{2\ensuremath{\mu}}_{n,1}}=\frac{\sin\left(\frac{2\pi}{n}\right)}{2}\notin\mathbb{\mathbb{Q}}\;\;\;\;\;\;(n=4)
\end{equation}
\begin{equation}
\frac{f(P_{2})}{f(P_{1})}=\frac{\text{\ensuremath{\mu}}_{n,2}}{\text{\ensuremath{\mu}}_{n,1}}=\sin\left(\frac{2\pi}{n}\right)\notin\mathbb{\mathbb{Q}}\;\;\;\;\;\;(n\geq4)
\end{equation}

Since $\frac{f(P_{2})}{f(P_{1})}\notin\mathbb{\mathbb{Q}}$ and all
orientations of a prototile appear in the same frequency, no periodic
configuration of prototiles $P_{1}$ and $P_{2}$ is possible. 

As a result the CASTs in Theorem~\ref{thm: CAST1} with $n\geq4$
are aperiodic. \end{proof}
\begin{rem}
The inflation multiplier makes no statement whether a CAST meets the
condition in Theorem~\ref{thm:CASTs-Aperiodic} or not. As a consequence
possible solutions have to be checked case by case. However, all examples
within this article fulfill the condition in Theorem~\ref{thm:CASTs-Aperiodic}.
\end{rem}

\begin{rem}
All prototiles of a CASTs $\mathcal{T}$ in Theorem~\ref{thm: CAST1}
with $n\in\left\{ 2,3\right\} $ have identical areas. For these CASTs
aperiodicity has to be proven with other methods.
\end{rem}

\begin{thm}
\label{thm:Minimal_Inflation_Multiplier}The smallest possible inflation
multipliers for CASTs $\mathcal{T}$ in Theorem~\ref{thm: CAST1}
with $n\geq4$ are given by $\left|\eta_{min}\right|=\left|\zeta_{2n}^{1}+\overline{\zeta_{2n}^{1}}\right|=\mu_{n,2},\;odd\;n$
and $\left|\eta_{min}\right|=\left|1+\zeta_{2n}^{1}\right|=\sqrt{\mu_{n,2}+2},\;even\;n$.\end{thm}
\begin{proof}
With Equation~\eqref{eq:lambda-i} we can describe the minimal eigenvalues
for the smallest possible inflation multipliers: 
\begin{equation}
\lambda_{min}=\eta_{min}\cdot\overline{\eta_{min}}=\left|\eta_{min}\right|^{2}=\mu_{n,3}+1\;\;\;\;\;\;(odd\;n)\label{eq:lambda_min_odd_n}
\end{equation}
\begin{equation}
\lambda_{min}=\eta_{min}\cdot\overline{\eta_{min}}=\left|\eta_{min}\right|^{2}=\mu_{n,2}+2\;\;\;\;\;\;(even\;n)\label{eq:lambda_min_even_n}
\end{equation}
From Equation~\eqref{eq:lambda-i} we can also derive:
\begin{equation}
\left|\eta_{1}\right|>\left|\eta_{2}\right|\implies\lambda_{1}>\lambda_{2}
\end{equation}
In other words, to prove Theorem~\ref{thm:Minimal_Inflation_Multiplier},
it is sufficient to show that for a substitution matrix of a CAST
no smaller eigenvalue $\lambda$ exist. We recall that such an eigenvalue
$\lambda$ has to fulfill Equations~\eqref{eq:mindest-b0}, \eqref{eq:lambda-a-1}
and~\eqref{Bedingung_f=0000FCr_Mn_primitiv_falls_n_gerade}.

With the following inequalities 

\begin{equation}
k>\mu_{n,k}\;\;\;\;\;\;(\left\lfloor n/2\right\rfloor \geq k>1)
\end{equation}
\begin{equation}
\mu_{n,k+1}>\mu_{n,k}\;\;\;\;\;\;(\left\lfloor n/2\right\rfloor \geq k+1>k\geq1)
\end{equation}

\begin{equation}
\mu_{n,k}+\mu_{n,l}>\mu_{n,k+1}+\mu_{n,l-1}\;\;\;\;\;\;(k\geq l,\;\left\lfloor n/2\right\rfloor \geq k+1>k\geq1,\;\left\lfloor n/2\right\rfloor \geq l>l-1\geq1)
\end{equation}

we can identify all eigenvalues which are smaller than $\lambda_{min}$.
As noted in Equation~\eqref{eq:lambda-a-1} all $c_{k}\geq0$, therefore
every eigenvalue $\lambda$ must be a sum of diagonals $\mu_{n,k}$. 

For $odd\;n$ we identified the following eigenvalues: 

\begin{equation}
\lambda_{min}>\mu_{n,5}=1+\zeta_{2n}^{2}+\overline{\zeta_{2n}^{2}}+\zeta_{2n}^{4}+\overline{\zeta_{2n}^{4}}\label{eq:odd-cases-start}
\end{equation}
\begin{equation}
\lambda_{min}>\mu_{n,4}=\zeta_{2n}^{1}+\overline{\zeta_{2n}^{1}}+\zeta_{2n}^{3}+\overline{\zeta_{2n}^{3}}\;\;\;\;\;\;(n=11)
\end{equation}

\begin{equation}
\lambda_{min}>\mu_{n,3}=1+\zeta_{2n}^{2}+\overline{\zeta_{2n}^{2}}
\end{equation}

\begin{equation}
\lambda_{min}>\mu_{n,2}=\zeta_{2n}^{1}+\overline{\zeta_{2n}^{1}}
\end{equation}

\begin{equation}
\lambda_{min}>1+\mu_{n,2}=1+\zeta_{2n}^{1}+\overline{\zeta_{2n}^{1}}\;\;\;\;\;\;(n>5)
\end{equation}

\begin{equation}
\lambda_{min}>2
\end{equation}

\begin{equation}
\lambda_{min}>1\label{eq:odd-cases-end}
\end{equation}

For $even\;n$ we identified the following eigenvalues:

\begin{equation}
\lambda_{min}>\mu_{n,6}\;\;\;\;\;\;(n=12)\label{eq:lambda-min-n-even-check-1}
\end{equation}
\begin{equation}
\lambda_{min}>\mu_{n,5}\;\;\;\;\;\;(n\in\{10,12\})\label{eq:lambda-min-n-even-check-2}
\end{equation}
\begin{equation}
\lambda_{min}>\mu_{n,4},\;\mu_{n,3},\;\mu_{n,2},\;\mu_{n,3}+1,\;3,\;2,\;1\label{eq:lambda-min-n-even-check-3}
\end{equation}

\begin{equation}
\lambda_{min}>1+\mu_{n,2}=1+\zeta_{2n}^{1}+\overline{\zeta_{2n}^{1}}\label{eq:lambda-min-n-even-check-4}
\end{equation}

\begin{equation}
\lambda_{min}>1+\mu_{n,2}=1+\zeta_{2n}^{1}+\overline{\zeta_{2n}^{1}}+\zeta_{2n}^{3}+\overline{\zeta_{2n}^{3}}\;\;\;\;\;\;(n=8)\label{eq:lambda-min-n-even-check-5}
\end{equation}

None of the eigenvalues in Equations~\eqref{eq:odd-cases-start}~-~\eqref{eq:lambda-min-n-even-check-5}
fulfills all conditions above, which completes the proof.\end{proof}
\begin{thm}
\label{thm:CAST2}The definition of CASTs with $n\geq4$ can be extended
to all CASTs with other prototiles than in Theorem~\ref{thm: CAST1}.\end{thm}
\begin{proof}
Let a CAST $\mathcal{T}^{*}$ exist with $l\geq l_{min}$ prototiles
$P_{k}^{*},\;l\geq k\geq1$ with a real positive, primitive substitution
matrix $M_{n}^{*}\in\mathbb{\mathbb{N}}_{0}^{l\times l}$ and a real
positive eigenvector $x_{A}^{*}\in\mathbb{R}^{l}$ whose elements
represent the relative areas $\frac{A(P_{k}^{*})}{A(P_{1}^{*})}$
of the prototiles:

\begin{equation}
x_{A}^{*}=\frac{1}{A(P_{1}^{*})}\left(\begin{array}{c}
A(P_{l}^{*})\\
\vdots\\
A(P_{3}^{*})\\
A(P_{2}^{*})\\
A(P_{1}^{*})
\end{array}\right)=\left(\begin{array}{c}
\frac{A(P_{l}^{*})}{A(P_{1}^{*})}\\
\vdots\\
\frac{A(P_{3}^{*})}{A(P_{1}^{*})}\\
\frac{A(P_{2}^{*})}{A(P_{1}^{*})}\\
1
\end{array}\right)
\end{equation}

with

\begin{equation}
A(P_{l}^{*})\geq A(P_{l-1}^{*})\geq\cdots\geq A(P_{2}^{*})\geq A(P_{1}^{*})>0
\end{equation}

and

\begin{equation}
\frac{A(P_{l}^{*})}{A(P_{1}^{*})}\geq\frac{A(P_{l-1}^{*})}{A(P_{1}^{*})}\geq\cdots\geq\frac{A(P_{2}^{*})}{A(P_{1}^{*})}\geq\frac{A(P_{1}^{*})}{A(P_{1}^{*})}=1
\end{equation}

Since $\mathcal{T}^{*}$ is a CAST by definition and all coordinates
of all vertices are algebraic integers in $\mathbb{Z}\left[\zeta_{2n}\right]$
the areas of all pototiles are real numbers so that

\begin{equation}
A(P_{k}^{*})\in\mathbb{R\cap Z}\left[\zeta_{2n}\right]=\mathbb{Z}\left[\zeta_{2n}+\overline{\zeta_{2n}}\right]=\mathbb{\mathbb{Z}}\left[\mu_{n}\right]
\end{equation}
 and so also 
\begin{equation}
\frac{A(P_{k}^{*})}{A(P_{1}^{*})}\in\mathbb{\mathbb{Z}}\left[\mu_{n}\right].
\end{equation}

As a consequence a transformation matrix $T\in\mathbb{\mathbb{\mathbb{Z}}}^{l\times\left\lfloor n/2\right\rfloor }$
and its inverse $T^{-1}\in\mathbb{\mathbb{\mathbb{Q}}}^{\left\lfloor n/2\right\rfloor \times l}$
exist such that:

\begin{equation}
x_{A}^{*}=Tx_{A}
\end{equation}

\begin{equation}
M_{n}^{*}=TM_{n}T^{-1}
\end{equation}

\begin{equation}
\lambda=c^{T}x_{A}=c^{*T}x_{A}^{*}=c^{*T}Tx_{A}
\end{equation}

In other words, if there is a CAST $\mathcal{T}^{*}$ then also a
corresponding CAST $\mathcal{T}$ (as defined in Theorem~\ref{thm: CAST1})
with the same eigenvalue $\lambda$ does exist. That includes but
may not be limited to all cases where the CASTs $\mathcal{T}$ and
$\mathcal{T}^{*}$ are mutually locally derivable.
\end{proof}

A similar approach is possible for the frequencies of the corresponding
CASTs $\mathcal{T}$ and $\mathcal{T}^{*}$. As a result Theorem~\eqref{thm:CASTs-Aperiodic}
also applies to CAST $\mathcal{T}^{*}$.

The corresponding CASTs $\mathcal{T}$ and $\mathcal{T}^{*}$ have
the same eigenvalue, so Theorem~\eqref{thm:Minimal_Inflation_Multiplier}
also applies to CAST $\mathcal{T}^{*}$.

\begin{rem}
$l_{min}$ is given by the algebraic degree of $\lambda$ which depends
on $n$. In other words, $\lambda$ is a solution of an irreducible
polynomial with integer coefficients, and $\lambda$ is of at least
$l_{min}$-th degree. $l_{min}$ can be described by Euler's totient
function denoted by $\Phi$, in detail $l_{min}=\frac{\Phi\left(n\right)}{2}$,
$odd\;n$ and $l_{min}=\Phi\left(n\right)$, $even\;n$.
\end{rem}

\begin{rem}
The eigenvalue $\lambda$ as noted in Equation~\eqref{eq:lambda-a-1}
and~\eqref{eq:lambda-a-1-vector} is unique if $l_{min}=\left\lfloor n/2\right\rfloor $.
This is the case if $n$ is a prime. If $n$ is not a prime, $\lambda$
may have more than one corresponding vector $c$. E.g. for $n=9$
we can show that $\mu_{9,4}=\mu_{9,2}+\mu_{9,1}$. An eigenvalue $\lambda=\mu_{9,4}-\mu_{9,1}=\mu_{9,2}$
implies that the corresponding vector $c\in\left\{ \left(\begin{array}{c}
1\\
0\\
0\\
-1
\end{array}\right),\left(\begin{array}{c}
0\\
0\\
1\\
0
\end{array}\right)\right\} $. Only the latter one is real positive. As a consequence we will consider
$\lambda\in\mathbb{\mathbb{\mathbb{N}}}_{0}\left[\mu_{n}\right]$
as noted in Equation~\eqref{eq:lambda-a-1-1} if at least one real
positive $c$ exist.
\end{rem}
\clearpage{}

\section{\label{sec:CASTs_with_minimal_inflation_multiplier}CASTs with Minimal
Inflation Multiplier}

In the this section, we discuss CASTs with minimal inflation multiplier
as noted in Theorem~\ref{thm:Minimal_Inflation_Multiplier}.

\subsection{\label{sub:The-case-n-odd}The case odd n}

The substitution matrix $M_{n,min}$ for CASTs with minimal eigenvalue
$\lambda_{min}=\mu_{n,3}+1,\;odd\;n$ is given by the following scheme:

\begin{equation}
M_{5,min}=\left(\begin{array}{cc}
2 & 1\\
1 & 1
\end{array}\right)
\end{equation}

\[
M_{7,min}=\left(\begin{array}{ccc}
2 & 1 & 1\\
1 & 2 & 0\\
1 & 0 & 1
\end{array}\right)
\]

\[
M_{9,min}=\left(\begin{array}{cccc}
2 & 1 & 1 & 0\\
1 & 2 & 0 & 1\\
1 & 0 & 2 & 0\\
0 & 1 & 0 & 1
\end{array}\right)
\]

\[
M_{n,min}=\left(\begin{array}{cccccc}
2 & 1 & 1 & 0 & \cdots & 0\\
1 & 2 & 0 & \ddots & \ddots & \vdots\\
1 & 0 & \ddots & \ddots & \ddots & 0\\
0 & \ddots & \ddots & \ddots & 0 & 1\\
\vdots & \ddots & \ddots & 0 & 2 & 0\\
0 & \cdots & 0 & 1 & 0 & 1
\end{array}\right)\;\;\;\;\;\;(odd\;n)
\]

The case $n=5$ describes the Penrose tiling with rhombs or Robinson
triangles and individual dihedral symmetry $D_{5}$. The cases with
$n>5$ are more difficult and require additional prototiles. For $n=7$
we give an example of a CAST $\mathcal{T}^{*}$ with individual dihedral
symmetry $D_{7}$ as shown in Fig.~\ref{fig:CAST7min} and the following
properties. $A(P_{k})$ stands for the areas of a prototile $P_{k}$.

\begin{equation}
x_{A}^{*}=\begin{pmatrix}\mu_{n,3}\\
\mu_{n,2}\\
\mu_{n,3}+\mu_{n,1}\\
\mu_{n,1}
\end{pmatrix}=\frac{1}{A(P_{1})}\begin{pmatrix}A(P_{3})\\
A(P_{2})\\
A(P_{1}^{'})\\
A(P_{1})
\end{pmatrix}
\end{equation}

\begin{equation}
T=\left(\begin{array}{ccc}
1 & 0 & 0\\
0 & 1 & 0\\
1 & 0 & 1\\
0 & 0 & 1
\end{array}\right)
\end{equation}

\begin{equation}
M_{7}^{*}=\left(\begin{array}{cccc}
1 & 1 & 1 & 0\\
1 & 2 & 0 & 0\\
3 & 1 & 0 & 2\\
0 & 0 & 1 & 0
\end{array}\right)
\end{equation}

\begin{figure}
\begin{center}
\resizebox{0.8\textwidth}{!}{%

\includegraphics{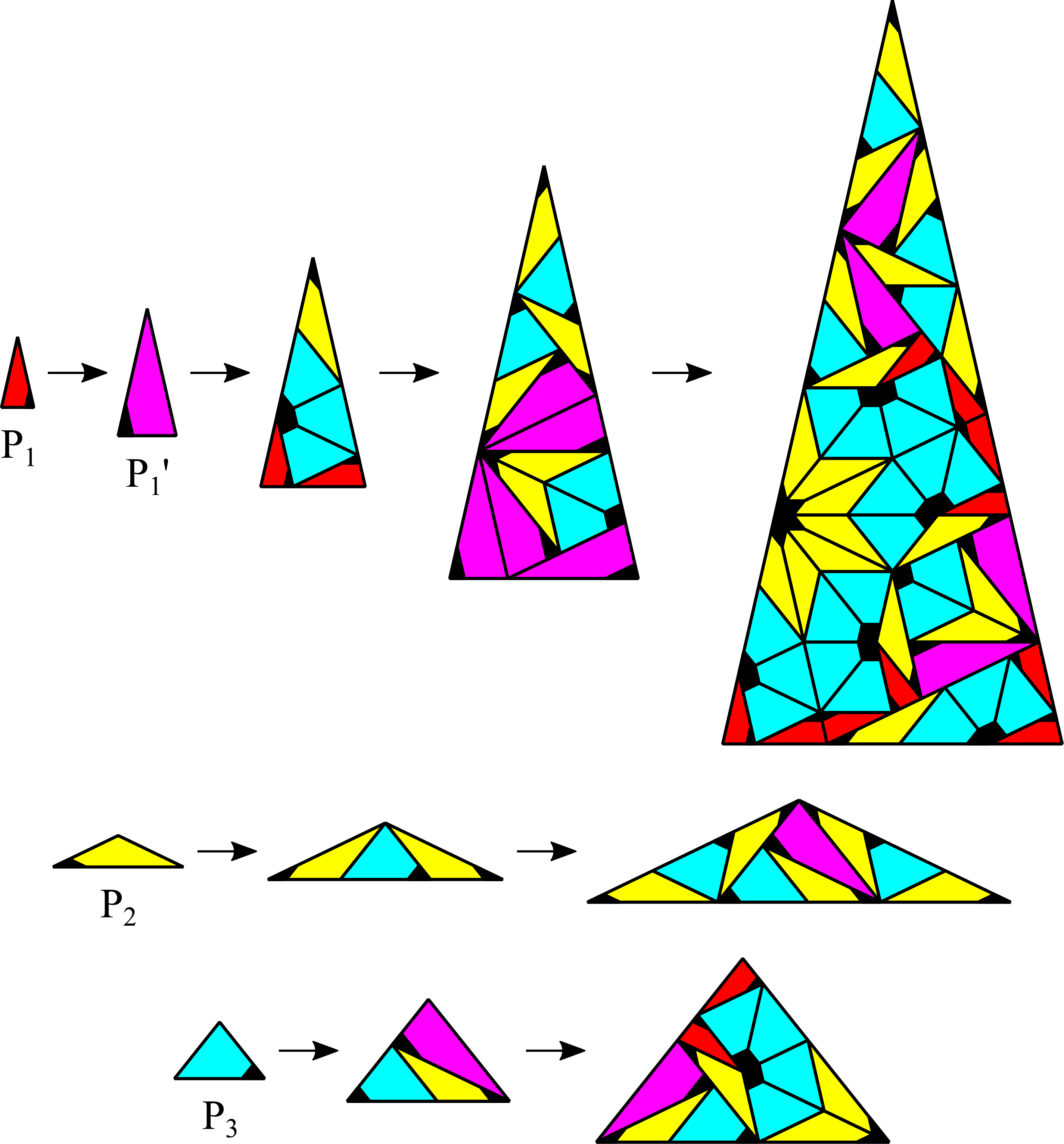}}
\end{center}\caption{\label{fig:CAST7min}CAST for the case $n=7$ with minimal inflation
multiplier. The black tips of the prototiles mark their respective
chirality.}
\end{figure}

A further example with $n=9$, minimal inflation multiplier but without
individual dihedral symmetry $D_{9}$ has been found but is not included
here.

Another CAST with $n=7$, individual dihedral symmetry $D_{7}$ and
the same minimal inflation multiplier is given in \citep[ Fig. 1, Sec. 3, 2nd matrix]{Nischke1996}
and shown in Fig.~\ref{fig:Nischke-Danzer-Mu2}.
\begin{conjecture}
\label{conj:CASTs-minimal-inflation-factor-n-odd}CASTs with minimal
inflation multiplier $\eta_{min}=\zeta_{2n}^{1}+\overline{\zeta_{2n}^{1}}=\mu_{n,2}$
and individual dihedral symmetry $D_{n}$ exist for every $odd\;n\geq5$.
\end{conjecture}
The status of Conjecture~\ref{conj:CASTs-minimal-inflation-factor-n-odd}
is subject to further research.

Formally all preferable conditions are met. However, for large $n$,
this type of CAST tends to be complex and the density of patches with
individual dihedral symmetry $D_{n}$ tends to be small.

\begin{figure}
\begin{center}
\resizebox{\textwidth}{!}{%

\includegraphics{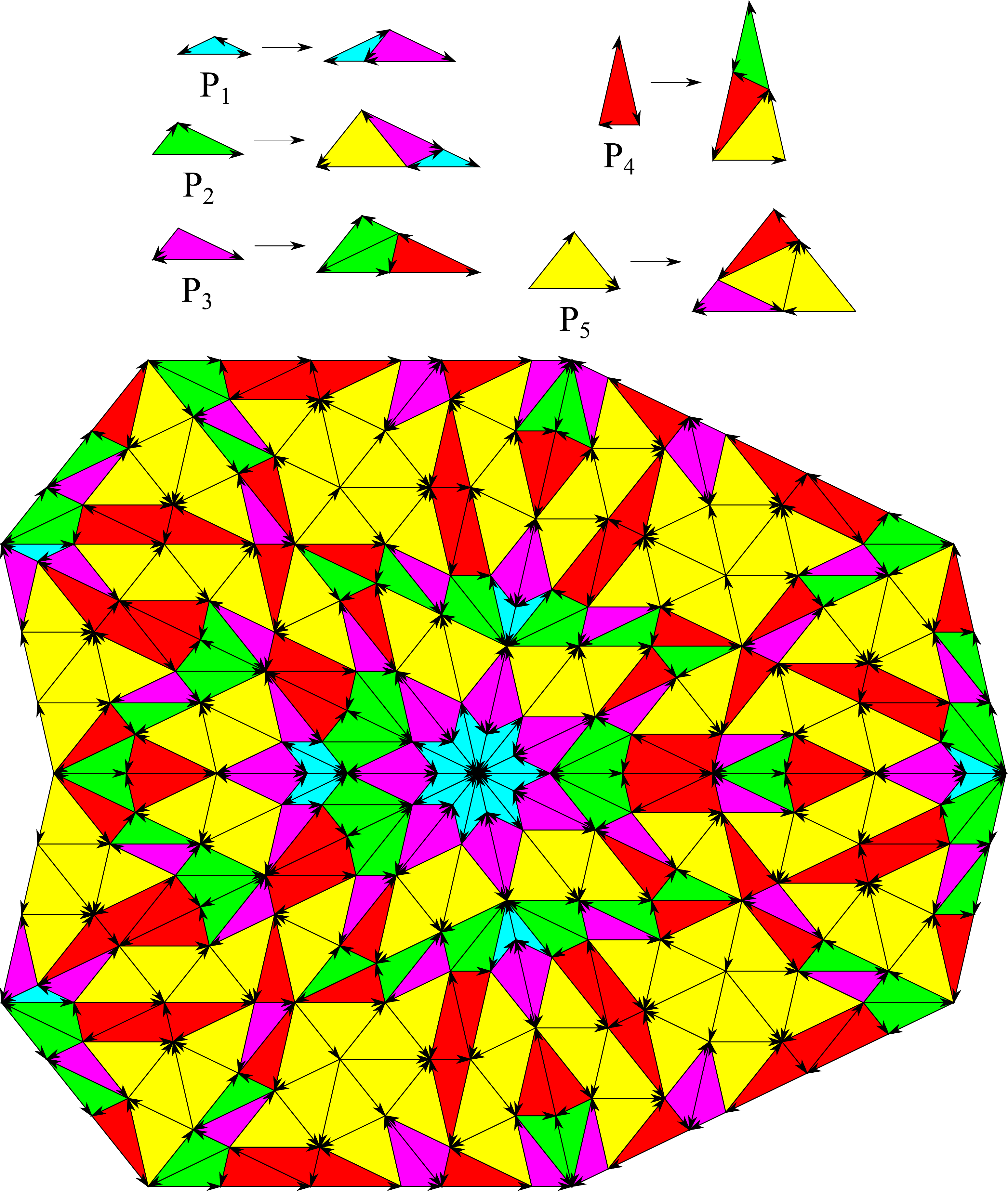}}
\end{center}\caption{\label{fig:Nischke-Danzer-Mu2} CAST for the case $n=7$ with minimal
inflation multiplier as described in \citep[Fig. 1 and Sec. 3, 2nd matrix]{Nischke1996}}
\end{figure}

\begin{rem}
The substitution rules of the CASTs in Section~\ref{sub:The-case-n-odd}
have no dihedral symmetry. The minimal inner angle of a prototile
is $\frac{\pi}{n}$. For this reason the CASTs in Section~\ref{sub:The-case-n-odd}
may have individual dihedral symmetry $D_{n}$ but not $D_{2n}$.
\end{rem}

\subsection{\label{sub:The-case-n-even}The case even n}

The substitution matrices $M_{n,min}$ for CASTs with minimal eigenvalue
$\lambda_{A,min}=\mu_{n,2}+2$ is given by the following scheme:

\begin{equation}
M_{4,min}=\left(\begin{array}{cc}
2 & 2\\
1 & 2
\end{array}\right)
\end{equation}

\[
M_{6,min}=\left(\begin{array}{ccc}
2 & 2 & 0\\
1 & 2 & 1\\
0 & 1 & 2
\end{array}\right)
\]

\[
M_{8,min}=\left(\begin{array}{cccc}
2 & 2 & 0 & 0\\
1 & 2 & 1 & 0\\
0 & 1 & 2 & 1\\
0 & 0 & 1 & 2
\end{array}\right)
\]

\[
M_{n,min}=\left(\begin{array}{cccccc}
2 & 2 & 0 & 0 & \cdots & 0\\
1 & 2 & 1 & \ddots & \ddots & \vdots\\
0 & 1 & \ddots & \ddots & \ddots & 0\\
0 & \ddots & \ddots & \ddots & 1 & 0\\
\vdots & \ddots & \ddots & 1 & 2 & 1\\
0 & \cdots & 0 & 0 & 1 & 2
\end{array}\right)\;\;\;\;\;\;(even\;n)
\]

The Lançon-Billard tiling (also known as binary tiling) \citep{LanconF.1988,C.Godreche1992,oro38933}
is a rhomb tiling with inflation multiplier $|\eta|=|1+\zeta_{2n}^{1}|$
for the case $n=5$. It is possible to generalize it for all $n\geq4$.
The approach is very similar to the approach in \citep{journals/dcg/Harriss05}.
The generalized substitution rules are shown in Fig.~\ref{fig:LB-Tiling-Generalized}. 

\begin{figure}
\begin{center}
\resizebox{\textwidth}{!}{%

\includegraphics{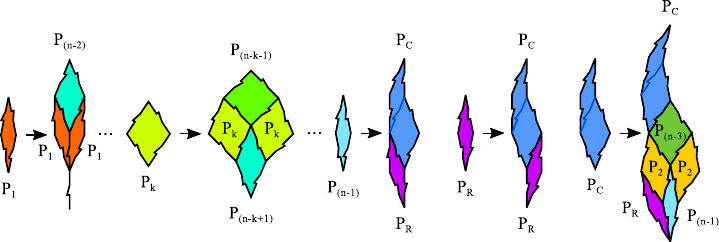}}
\end{center}\caption{\label{fig:LB-Tiling-Generalized}Generalized {\small{}Lançon-Billard
tiling}}
\end{figure}
The eigenvector $x_{B}$ and substitution matrix $M_{n}^{*}$ are
given by:

\begin{equation}
x_{A}^{*}=\begin{pmatrix}\mu_{n-1}+\mu_{2}\\
\mu_{n-1}\\
\mu_{n-1}\\
\vdots\\
\mu_{1}
\end{pmatrix}=\frac{1}{A(P_{1})}\begin{pmatrix}A(P_{C})\\
A(P_{R})\\
A(P_{n-1})\\
\vdots\\
A(P_{1})
\end{pmatrix}
\end{equation}

\begin{equation}
M_{4}^{*}=\left(\begin{array}{ccccc}
1 & 1 & 1 & 2 & 1\\
1 & 1 & 0 & 0 & 0\\
1 & 1 & 0 & 0 & 0\\
0 & 0 & 1 & 2 & 1\\
0 & 0 & 0 & 1 & 2
\end{array}\right)
\end{equation}

\[
M_{6}^{*}=\left(\begin{array}{ccccccc}
1 & 1 & 1 & 0 & 1 & 2 & 0\\
1 & 1 & 0 & 0 & 0 & 0 & 0\\
1 & 1 & 0 & 0 & 0 & 0 & 0\\
0 & 0 & 0 & 2 & 1 & 0 & 1\\
0 & 0 & 0 & 1 & 2 & 1 & 0\\
0 & 0 & 1 & 0 & 1 & 2 & 0\\
0 & 0 & 0 & 1 & 0 & 0 & 2
\end{array}\right)
\]

\[
M_{8}^{*}=\left(\begin{array}{ccccccccc}
1 & 1 & 1 & 0 & 1 & 0 & 0 & 2 & 0\\
1 & 1 & 0 & 0 & 0 & 0 & 0 & 0 & 0\\
1 & 1 & 0 & 0 & 0 & 0 & 0 & 0 & 0\\
0 & 0 & 0 & 2 & 0 & 0 & 1 & 0 & 1\\
0 & 0 & 0 & 0 & 2 & 1 & 0 & 1 & 0\\
0 & 0 & 0 & 0 & 1 & 2 & 1 & 0 & 0\\
0 & 0 & 0 & 1 & 0 & 1 & 2 & 0 & 0\\
0 & 0 & 1 & 0 & 1 & 0 & 0 & 2 & 0\\
0 & 0 & 0 & 1 & 0 & 0 & 0 & 0 & 2
\end{array}\right)
\]

\begin{rem}
\label{rem:LB_tiling_for_odd_n}For $odd\;n$, the scheme and the
substitution matrix $M_{n}^{*}$ can be separated into two independent
parts. We choose the part which relays on prototiles $P_{k},\;n>k\geq1,\;odd\;k$
only. The eigenvector is given by the areas of the rhombic prototiles
with side length 1 and area $A(P_{k})=\sin\left(\frac{k\pi}{n}\right)$.
With Equation~\eqref{eq:diagonal-equivalenz} and~\eqref{eq:diagonal-equivalenz-2}
we can write: 
\begin{equation}
\frac{1}{A(P_{1})}\begin{pmatrix}A(P_{2\left\lfloor n/4\right\rfloor +1})\\
\vdots\\
A(P_{n-4})\\
A(P_{3})\\
A(P_{n-2})\\
A(P_{1})
\end{pmatrix}=\frac{1}{A(P_{1})}\begin{pmatrix}A(P_{\left\lfloor n/2\right\rfloor })\\
\vdots\\
A(P_{4})\\
A(P_{3})\\
A(P_{2})\\
A(P_{1})
\end{pmatrix}=\begin{pmatrix}\mu_{n,\left\lfloor n/2\right\rfloor }\\
\vdots\\
\mu_{n,4}\\
\mu_{n,3}\\
\mu_{n,2}\\
\mu_{n,1}
\end{pmatrix}=x_{A}\;\;\;\;\;\;(odd\;n)\label{eq:xa_for_odd_n}
\end{equation}

As a result, we can use Equation~\eqref{eq:M-odd} to describe the
substitution matrix:

\begin{equation}
M_{n}=M_{n,2}+2E\;\;\;\;\;\;(odd\;n)
\end{equation}

\end{rem}
Regarding generalized Lançon-Billard tilings, we can confirm the results
sketched in \citep{2015arXiv150902053H}.

The generalized Lançon-Billard tiling does not contain any patches
with individual dihedral symmetry $D_{n}$ or $D_{2n}$. However,
for $n=4$ and $n=5$, other CASTs with individual symmetry $D_{8}$
and $D_{10}$ have been derived and submitted to \citep{HFonl}. An
example for $n=4$ is mentioned in Section~\ref{sec:Gaps_to_Prototiles_Algorithm}
and shown in Fig.~\ref{fig:Residual-CAST-4}.

\clearpage{}

\section{\label{sec:CASTs_with_inflation_multiplier_=0000B5n,n-1/2_n_odd}CASTs
with Inflation Multiplier Equal to the Longest Diagonal of a Regular
Odd n-Gon }

Another interesting approach to identify CASTs with preferred properties
is to choose a CAST as described in Theorem~\ref{thm: CAST1} and
a relative small inflation multiplier, in detail the longest diagonal
of a regular $n$-gon with $odd\;n$:

\begin{equation}
\eta=\mu_{n,\left\lfloor n/2\right\rfloor }\;\;\;\;\;\;(odd\;n)
\end{equation}

Because of Equations~\ref{eq:lambda-a-square}, \ref{eq:goldenfield}
and~\ref{eq:M-odd}, we can write:

\begin{equation}
M_{n}=M_{n,\left\lfloor n/2\right\rfloor }^{2}=\sum_{i=1}^{\left\lfloor n/2\right\rfloor }M_{n,i}=\left(\begin{array}{ccccc}
n & n-1 & \cdots & 2 & 1\\
n-1 & n-1 & \cdots & 2 & 1\\
\vdots & \vdots & \ddots & \vdots & \vdots\\
2 & 2 & \cdots & 2 & 1\\
1 & 1 & \cdots & 1 & 1
\end{array}\right)\label{eq:DHW-Matrix}
\end{equation}

For the case $n=5$, we have the Penrose tiling, as described in \citep[Ex. 6.1]{oro38933}
and \citep[Fig. 10.3.14]{Grunbaum:1986:TP:19304}. For the case $n=7$,
we have different tiling in \citep["Danzer's 7-fold variant"]{HFonl},
\citep[Fig. 11]{Nischke1996} and \citep{MathPages1}. More examples
for all cases $n\leq15$ can be found in \citep["Half Rhombs"]{THonl}.
The matrix in Equation~\eqref{eq:DHW-Matrix} is equivalent to those
mentioned in \citep["Half Rhombs"]{THonl}. However, it is much older.
To the knowledge of the author, it appeared first in \citep{Warrington1988a}.
\begin{conjecture}
\label{conj:Verallgemeinerte-Penrose-CAST}CASTs as described in Theorem~\ref{thm: CAST1}
with inflation multiplier $\eta=\mu_{n,\left\lfloor n/2\right\rfloor }$,
substitution matrix $M_{n}=\left(M_{n,\left\lfloor n/2\right\rfloor }\right){}^{2}$
and exactly $\left\lfloor n/2\right\rfloor $ prototiles and $\left\lfloor n/2\right\rfloor $
corresponding substitution rules exist for every $odd\;n\geq5$.
\end{conjecture}
The substitution rules have no dihedral symmetry, so different combinations
of the prototiles chirality within the substitution rules define different
CASTs with identical inflation multiplier. For the cases $7\leq n\leq11$,
solutions with individual dihedral symmetry $D_{n}$ have been found
by trial and error. For the case $n=7$, see Fig.~\ref{fig:CAST7maxdiag}
and~\ref{fig:CAST7maxdiag-2}. For the case $n=11$, see Fig.~\ref{fig:CAST11maxdiag}.
Because of the complexity of that case, just one vertex star within
prototile $P_{4}$ has been chosen to illustrate the individual dihedral
symmetry $D_{11}$.
\begin{conjecture}
\textup{\label{conj:Verallgemeinerte-Penrose-CAST-Dn}}For every \textup{$odd\;n\geq5$,}
the substitution rules of a CAST as described in Conjecture~\ref{conj:Verallgemeinerte-Penrose-CAST}
can be modified so that the CAST yields individual dihedral symmetry
$D_{n}$.
\end{conjecture}
Formally all preferable conditions are met. However, for CASTs with
large $n$ as described in Conjecture~\ref{conj:Verallgemeinerte-Penrose-CAST-Dn}
the density of patches with individual dihedral symmetry $D_{n}$
tend to be small.
\begin{rem}
CASTs with $odd\;n$ in Conjectures~\ref{conj:Verallgemeinerte-Penrose-CAST}
and~\ref{conj:Verallgemeinerte-Penrose-CAST-Dn} can easily be identified
by choosing the set of prototiles as set of isosceles triangles with
leg length $1$ and vertex angles $\frac{k\pi}{n}$. Such isosceles
triangles have areas as described in Equation~\eqref{eq:triangle-area}.
Because of Equation~\eqref{eq:diagonal-equivalenz-2} we can use
a similar approach as in Remark~\ref{rem:LB_tiling_for_odd_n}, in
detail we only use the triangles with vertex angle $\frac{k\pi}{n},\;odd\;k$
and an eigenvector $x_{A}$ as noted in Equation~\eqref{eq:xa_for_odd_n}.
As a result all inner angles of all prototiles are intgeres multiples
of $\frac{\pi}{n}$. The approach was used in \citep[Ex. 6.1]{oro38933},
\citep[Fig. 10.3.14]{Grunbaum:1986:TP:19304}, \citep["Danzer's 7-fold variant"]{HFonl},
\citep["Half Rhombs"]{THonl} and herein, see Fig.~\ref{fig:CAST7maxdiag},
\ref{fig:CAST11maxdiag} and~\ref{fig:CAST7maxdiag-2}.
\end{rem}

\begin{rem}
The substitution rules of the CASTs in this section have no dihedral
symmetry. The minimal inner angle of a prototile is $\frac{\pi}{n}$.
For this reason the CASTs in this section may have individual dihedral
symmetry $D_{n}$ but not $D_{2n}$.
\end{rem}
\begin{figure}
\begin{center}
\resizebox{0.8\textwidth}{!}{%

\includegraphics{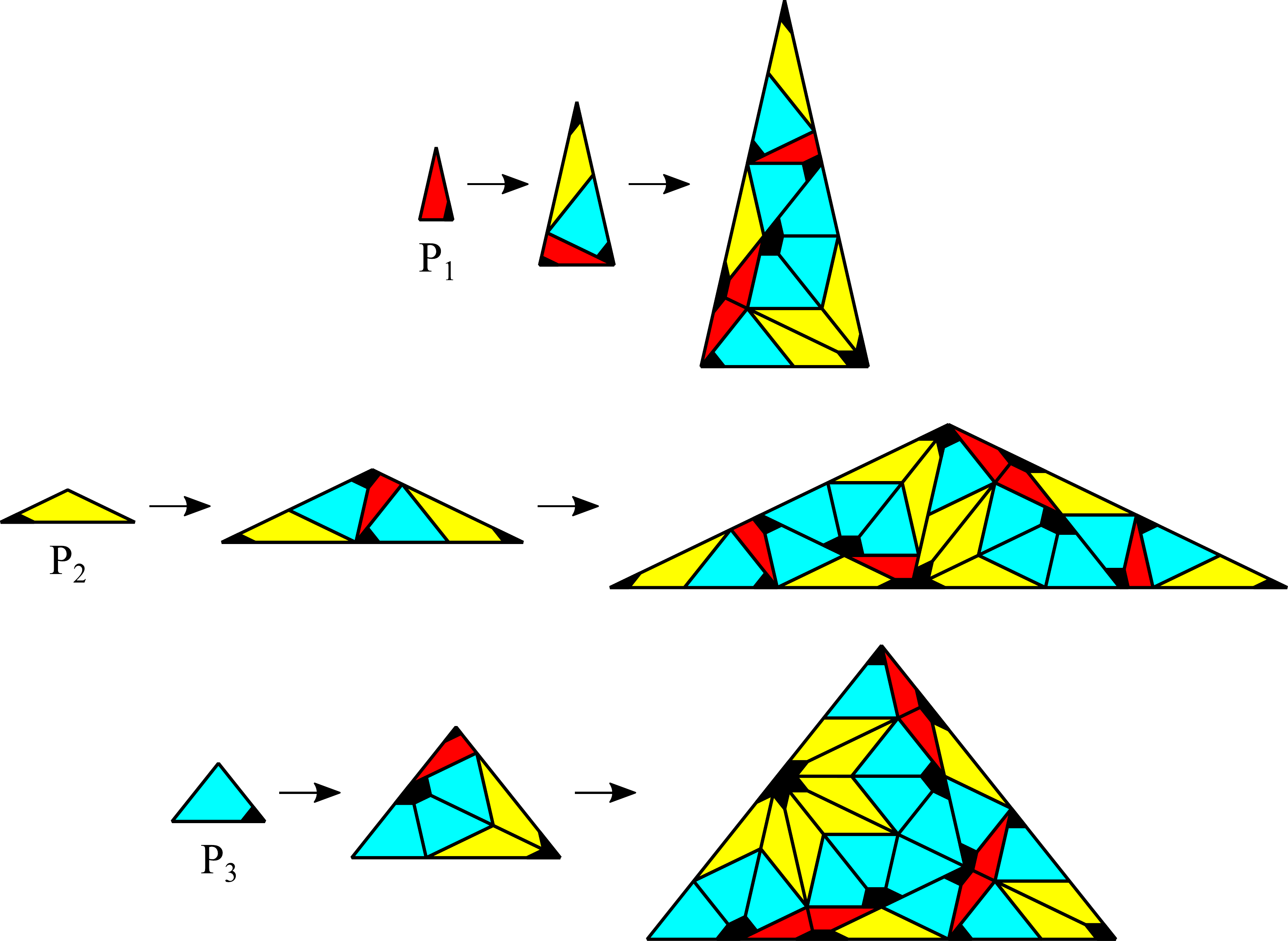}}
\end{center}\caption{\label{fig:CAST7maxdiag}CAST for the case $n=7$ with inflation multiplier
$\mu_{7,3}$}
\end{figure}

\begin{figure}
\begin{center}
\resizebox{0.9\textwidth}{!}{%

\includegraphics{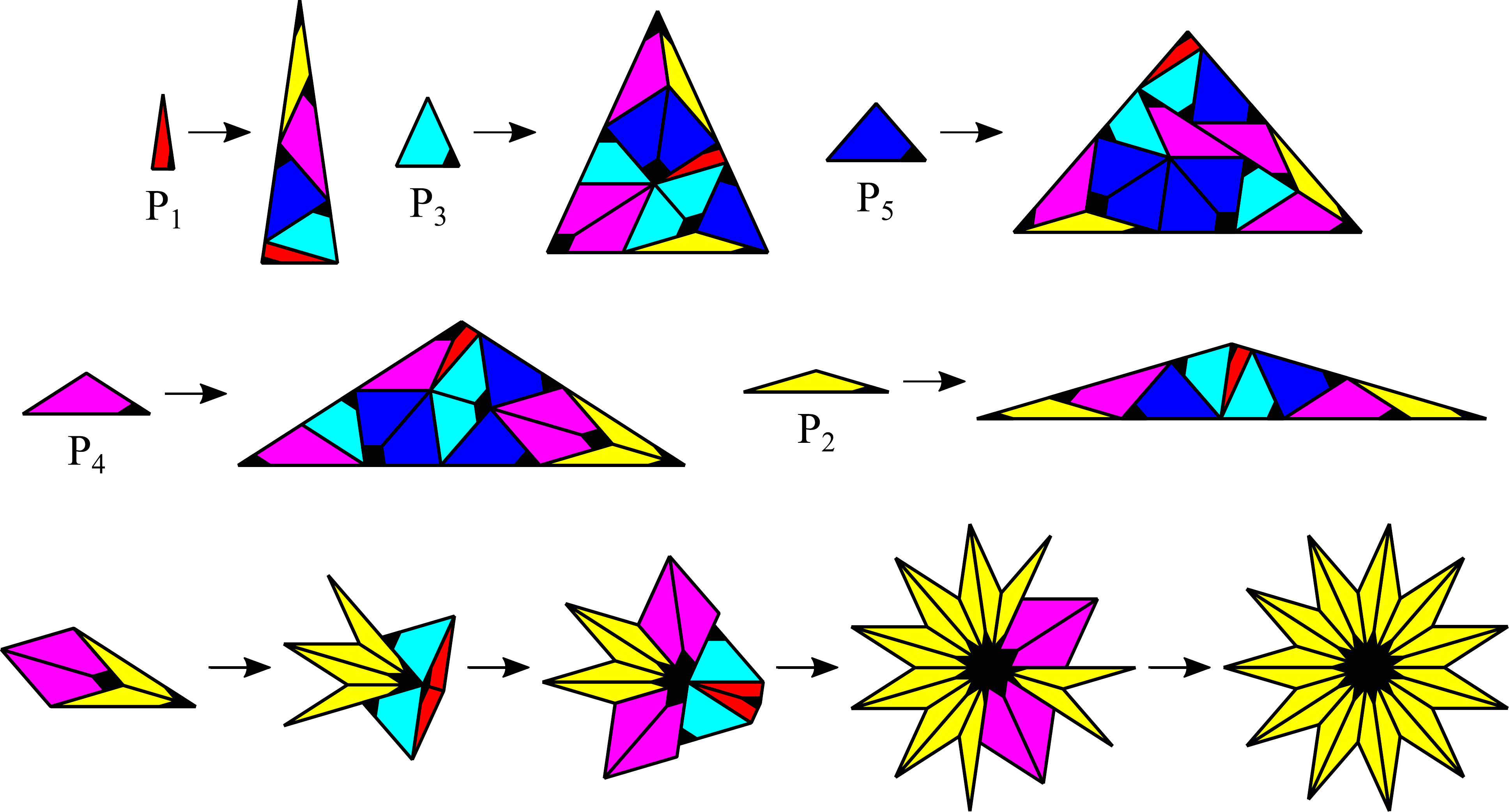}}
\end{center}\caption{\label{fig:CAST11maxdiag}CAST for the case $n=11$ with inflation
multiplier $\mu_{11,5}$\protect \\
One vertex star within prototile $P_{4}$ has been chosen to illustrate
the individual dihedral symmetry $D_{11}$.}
\end{figure}

\begin{figure}
\begin{center}
\resizebox{\textwidth}{!}{%

\includegraphics{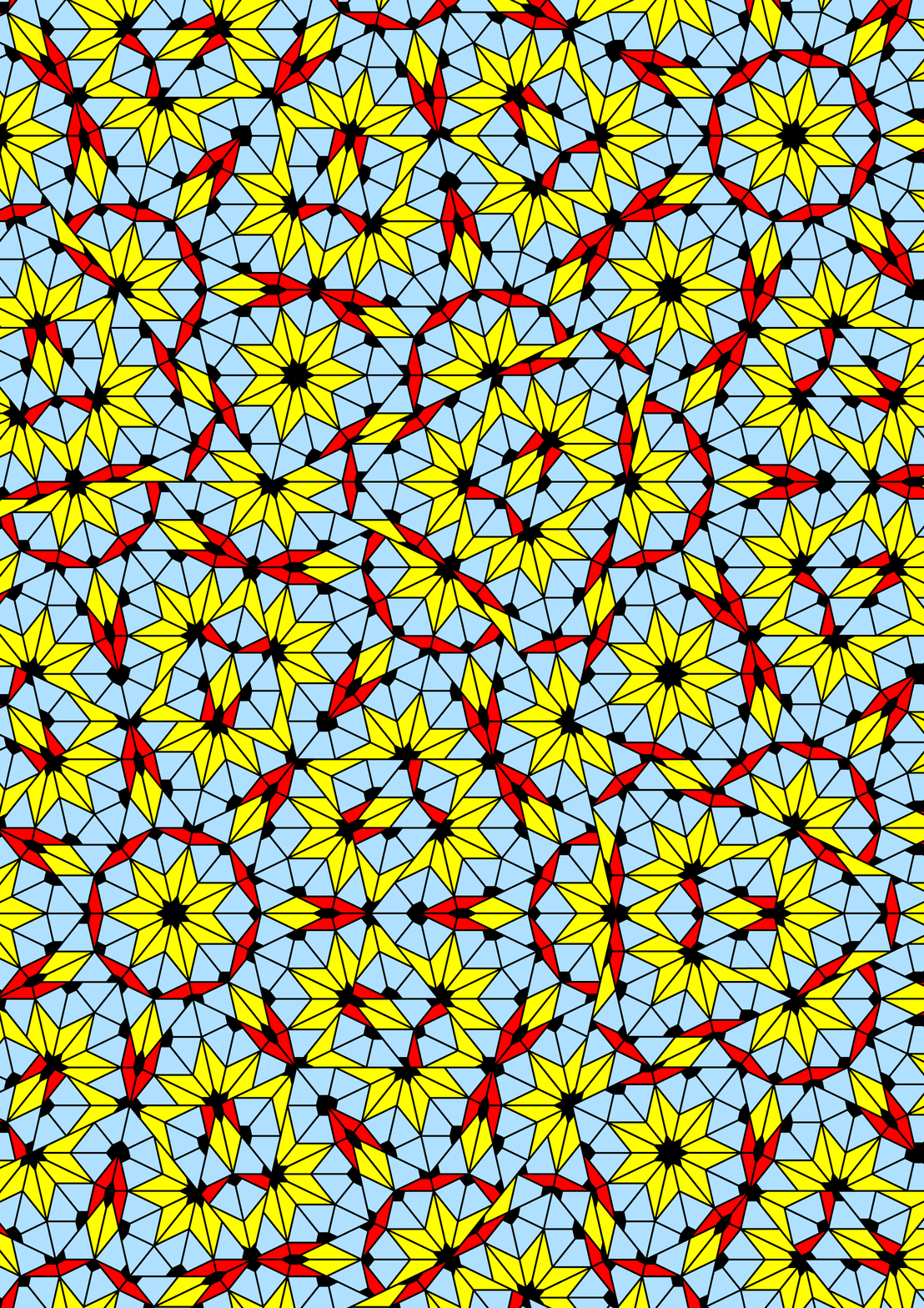}}
\end{center}\caption{\label{fig:CAST7maxdiag-2}CAST for the case $n=7$ with inflation
multiplier $\mu_{7,3}$}
\end{figure}

\clearpage{}

\section{\label{sec:Rhomic_CASTs_with_symmetric_edges_and_sunbstitution_rules}Rhombic
CASTs with Symmetric Edges and Substitution Rules\label{sub:Subsection-title-1-1}}

For many cases of CASTs with given $n$ and rhombic prototiles, the
minimal inflation multiplier can be described if some additional preconditions
are met. In the following section, we will focus on cases where all
edges of all substitution rules are congruent and the inner angles
of the rhomb prototiles are integer multiples of $\frac{\pi}{n}$.
\begin{defn}
\label{def:edge_defintion}“Edge” means here a segment of the boundary
of a supertile including the set of tiles which are crossed by it. 
\end{defn}
Furthermore we consider only cases where all rhombs of the edge are
bisected by the boundary of the supertile along one of their diagonals
as shown in Fig.~\ref{fig:edge-configuration}.

\begin{figure}[H]
\begin{center}
\resizebox{0.5\textwidth}{!}{%

\includegraphics{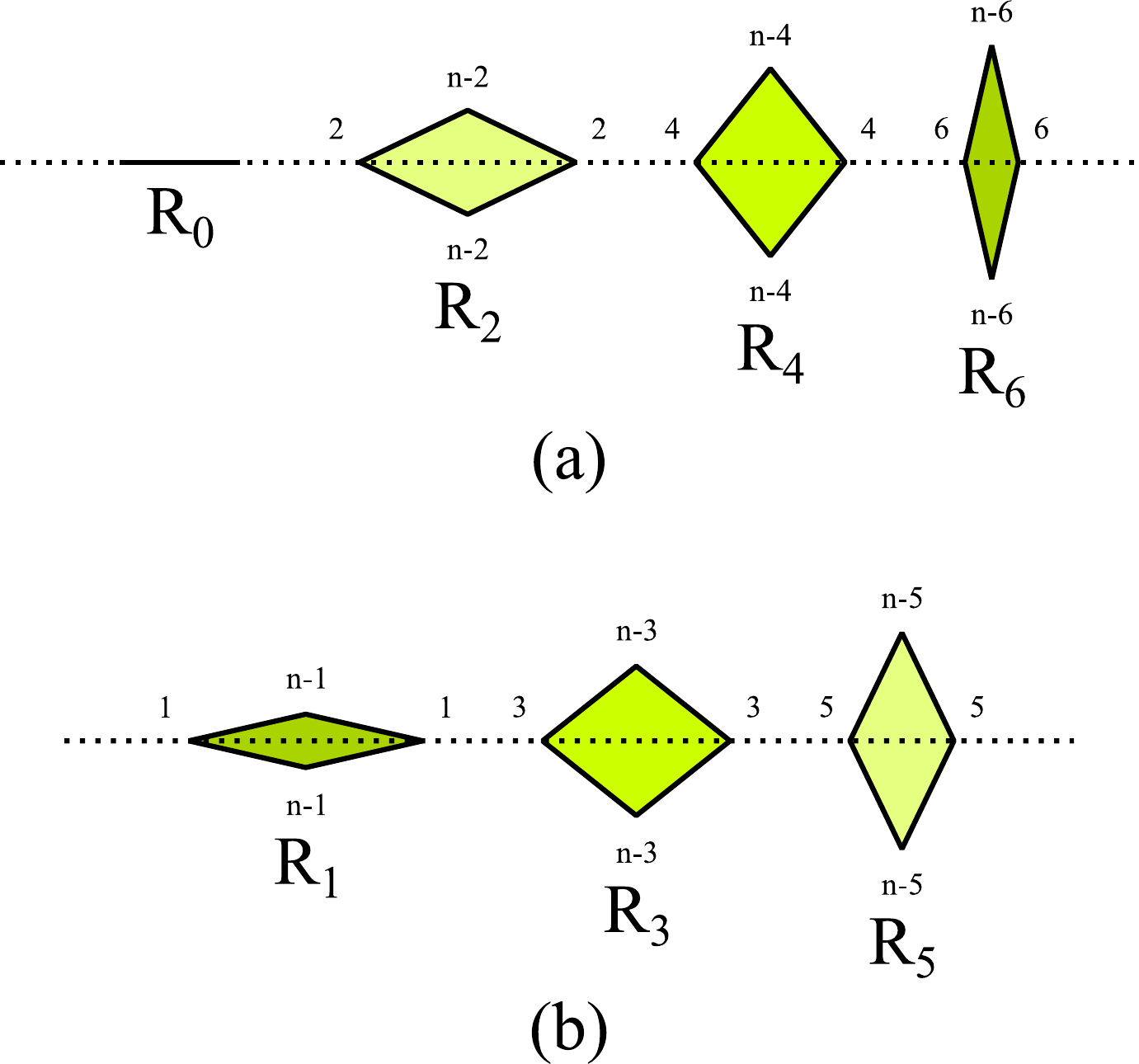}}
\end{center}\caption{\label{fig:edge-configuration}The ''Edges'' of a substitution rules
are defined as the boundaries of the supertile (dashed line) and the
rhombs bisected by it along one of their diagonals. The figure illustrates
how the rhombs can be placed accordingly for even edge configuration~(a)
and odd edge configuration~(b). The inner angles of the rhombs are
integer multiples of $\frac{\pi}{n}$ and are denoted by small numbers
near the tips. (Example $n=7$)}
\end{figure}

Despite these restrictions we still have several options left:
\begin{itemize}
\item There are two ways to place rhombs on the edge of substitution rules.
We recall that the inner angles of the rhombs are integer multiples
of $\frac{\pi}{n}$. We can place all rhombs on the edge so that the
inner angles either with even or odd multiples of $\frac{\pi}{n}$
are bisected by the boundary of the supertile. We will call these
two cases “even” and “odd edge configuration”, for details see Fig.~\ref{fig:edge-configuration}.
A “mixed” configuration is not allowed, because it would force the
existence of rhombs with inner angle equal to $\left(k+\frac{1}{2}\right)\frac{\pi}{n}$. 
\item We can choose the symmetry of the substitution rules and their edges.
Possible choices are dihedral symmetry $D_{1}$ and $D_{2}$. Edges
with dihedral symmetry $D_{1}$ can have the boundary of the supertile
or its perpendicular bisector as line of symmetry. The smallest nontrivial
solution for the latter case is the generalized Goodman-Strauss tiling
\citep{journals/dcg/Harriss05}. Since this example does not provide
individual dihedral symmetry $D_{n}$ or $D_{2n}$ in general, we
will focus on the other case.\\
Substitution rules of rhombs which appear on the edge of a substitution
rule are forced to have the appropriate dihedral symmetry $D_{1}$
as well. This is also true for substitution rules of prototiles which
lie on the diagonal i.e. a line of symmetry of a substitution rule.
The orientations of the edges have to be considered as well. These
three conditions may force the introduction of additional rhomb prototiles
and substitution rules. Additionally, the existence of edges with
orientations may require additional preconditions.\\
To avoid this problem, a general dihedral symmetry $D_{2}$ can be
chosen for the substitution rules and their edges.
\item Parity of the chosen $n$ may require different approaches in some
cases, similar to the example of the generalized Lançon-Billard tiling
in Section~\ref{sub:The-case-n-even} and Fig.~\ref{fig:LB-Tiling-Generalized}.
\end{itemize}
The options above can be combined without restrictions, so in total
we have eight cases:

\begin{table}[H]
\caption{\label{tab:Definition-of-cases}Definition of cases of rhombic CASTs}

\begin{center}
\resizebox{\textwidth}{!}{%

\begin{tabular}{|>{\centering}m{0.2\textwidth}|>{\centering}m{0.4\textwidth}|>{\centering}m{0.4\textwidth}|}
\hline 
 & Substitution rules and their edges\linebreak have at least dihedral
symmetry $D_{1}$ & Substitution rules and their edges\linebreak have dihedral symmetry
$D_{2}$\tabularnewline
\hline 
Even edge\linebreak configuration & Case 1a $(even\;n)$\linebreak Case 1b $(odd\;n)$ & Case 3a $(even\;n)$\linebreak Case 3b $(odd\;n)$\tabularnewline
\hline 
Odd edge\linebreak configuration & Case 2a $(even\;n)$\linebreak Case 2b $(odd\;n)$ & Case 4a $(even\;n)$\linebreak Case 4b $(odd\;n)$\tabularnewline
\hline 
\end{tabular}

}
\end{center}
\end{table}

Within this section, we will denote a rhomb with side length $d_{0}\equiv1$
and an inner angle $\frac{k\pi}{n}$, $n>k>0$ which is bisected by
the boundary of a supertile as $R_{k}$ and its diagonal which lies
on the boundary of a supertile as $d_{k}$. The line segments which
appear on the edges of substitution rules in the cases with even edge
configuration are denoted by $R_{0}$ and their length by $d_{\text{0}}$.
Under those conditions, we can write the inflation multiplier $\eta$
as sum of the $d_{k}$:
\begin{equation}
\eta=\sum_{i=0}^{\left\lfloor (n-1)/2\right\rfloor }\alpha_{2i}\;d_{2i}\;\;\;\;\;\;(\mathbf{\alpha}_{k}\in\mathbb{\mathbb{N}}_{0},\;cases\;1\;and\;3)\label{eq:rhomb_inflation_multiplier_by_rhomb_diagonals_1}
\end{equation}

\begin{equation}
\eta=\sum_{i=0}^{\left\lfloor n/2-1\right\rfloor }\alpha_{2i+1\;}d_{2i+1}\;\;\;\;\;\;(\alpha_{k}\in\mathbb{\mathbb{N}}_{0},\;cases\;2\;and\;4)\label{eq:rhomb_inflation_multiplier_by_rhomb_diagonals_2}
\end{equation}

The inflation multiplier $\eta$ can also be written as a sum of $2n$-th
roots of unity, where the roots occur in pairs $\zeta_{2n}^{i}+\overline{\zeta_{2n}^{i}}$
so that $\eta\in\mathbb{R}$:

\begin{equation}
\eta=\alpha_{0}+\sum_{i=0}^{\left\lfloor (n-1)/2\right\rfloor }\alpha_{2i}\;\left(\zeta_{2n}^{i}+\overline{\zeta_{2n}^{i}}\right)\;\;\;\;\;\;(\mathbf{\alpha}_{k}\in\mathbb{\mathbb{N}}_{0},\;\eta\in\mathbb{R},\;cases\;1\;and\;3)\label{eq:rhomb_inflation_multiplier_by_roots_of_unity_1}
\end{equation}

\begin{equation}
\eta=\sqrt{2+\zeta_{2n}^{1}+\overline{\zeta_{2n}^{1}}}\;\left(\alpha_{1}+\sum_{i=0}^{\left\lfloor n/2-1\right\rfloor }\alpha_{2i+1}\;\left(\zeta_{2n}^{i}+\overline{\zeta_{2n}^{i}}\right)\right)\;\;\;\;\;\;(\alpha_{k}\in\mathbb{\mathbb{N}}_{0},\;\eta\in\mathbb{R},\;cases\;2\;and\;4)\label{eq:rhomb_inflation_multiplier_by_roots_of_unity_2}
\end{equation}

We borrow a remark regarding worms from \citep{Frettloeh2013}: ``By
definition, a parallelogram has two pairs of parallel edges. In a
parallelogram tiling therefore there are natural lines of tiles linked
each sharing a parallel edge with the next. These lines are called
worms and were used by Conway in studying the Penrose tilings \citep{Grunbaum:1986:TP:19304}.
We will follow Conway to call these lines worms.'' 

Every rhomb $R_{k}$ with an inner angle $\frac{k\pi}{n}$, $n>k>0$
which is bisected by the boundary of a supertile work as entry-exit-node
of two worms. We also can say that such a rhomb on the edge ``reflects''
a worm back into the substitution rule. The line segments $R_{0}$
work as entry-exit-node of one worm only.

The Kannan-Soroker-Kenyon (KSK) criterion as defined in \citep{KannanS92,Kenyon93}
decides whether a polygon can be tiled by parallelograms. We use a
simplified phrasing for the criterion: All line segments $R_{0}$
and all inner line segments of rhombs $R_{k}$, $n>k>0$ on the edge
of the substitution rule serve as nodes. All corresponding nodes of
the substitution rules are connected by lines. The KSK criterion is
fulfilled if for every intersection of lines the inner angles between
the corresponding nodes are larger than $0$ and smaller than $\pi$.

We start with observations of the substitution rule for rhomb prototile
$R_{1}$ or $R_{n-1}$. In detail we will focus on those both edges
which enclose the inner angle equal $\frac{\pi}{n}$ and which we
will denote as ``corresponding edges''. We will denote the area
where the both edges meet as the ``tip'' of the substitution rule.
\begin{enumerate}
\item The rhombs on the edges must not overlap. For this reason, for the
tip of the substitution rule, only three configurations are possible
as shown in Fig.~\ref{fig:tip-of-rhomb-configuration}. Obviously,
a tip as shown in Fig.~\ref{fig:tip-of-rhomb-configuration}~(c)
is compliant to the cases 2 and 4 with odd edge configuration and
Fig.~\ref{fig:tip-of-rhomb-configuration}~(b) to cases 1 and 3
with even edge configuration. Fig.~\ref{fig:tip-of-rhomb-configuration}~(a)
requires the even edge configuration as well. Since all edges are
congruent, it must be the start and the end of the same edge, which
meet on that vertex. Since start and end of the edge are different,
it can not have dihedral symmetry $D_{2}$. For this reason, the tip
in Fig.~\ref{fig:tip-of-rhomb-configuration}~(a) is not compliant
to case 3.\\
\begin{figure}[H]
\begin{center}
\resizebox{0.5\textwidth}{!}{%

\includegraphics{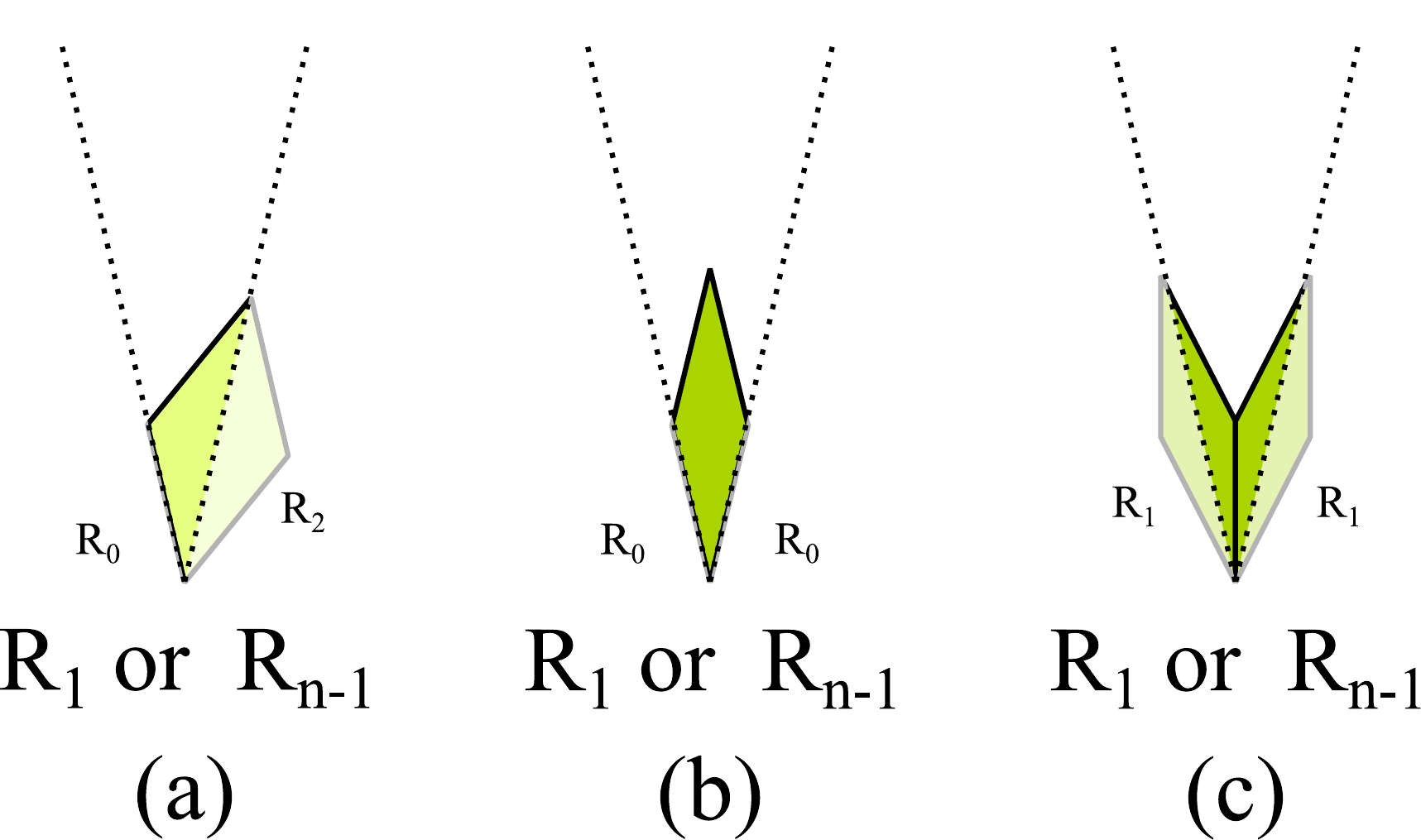}}
\end{center}\caption{\label{fig:tip-of-rhomb-configuration}Substitution rule for rhomb
$R_{1}$ or $R_{n-1}$: Possible configurations of the tip . (Example
$n=7$)}
\end{figure}

\item Any rhomb $R_{k}$ with $n-3\geq k\geq3$ on one edge implies the
existence of a rhomb $R_{k-2}$ on the corresponding edge. In turn,
rhomb $R_{k}$ on one edge is implied by a rhomb $R_{k+2}$ on the
corresponding edge or a rhomb $R_{k}$ on the opposite edge. An example
is shown in Fig.~\ref{fig:Rhomb-Rk-and-relatives}.\\
\begin{figure}[H]
\begin{center}
\resizebox{0.5\textwidth}{!}{%

\includegraphics{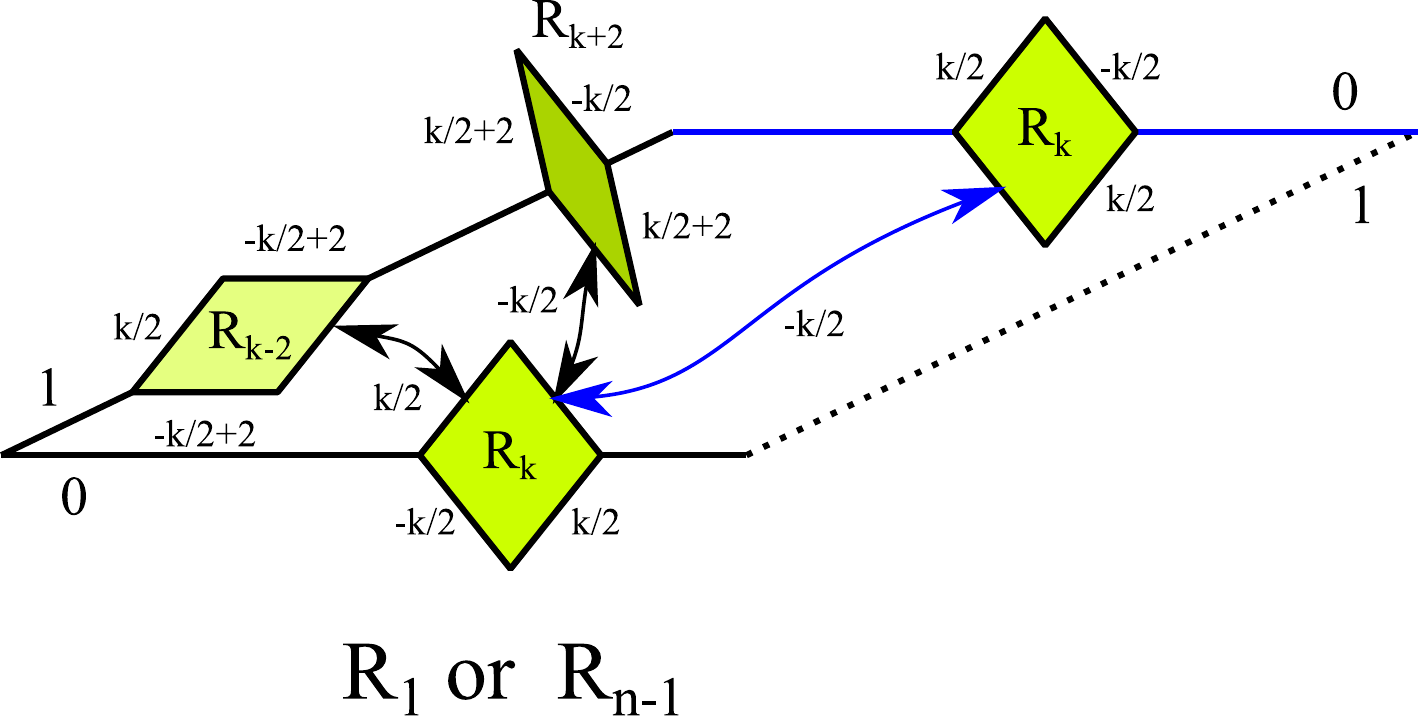}}
\end{center}\caption{\label{fig:Rhomb-Rk-and-relatives}Substitution rule for rhomb $R_{1}$
or $R_{n-1}$: Rhomb $R_{k}$ on the edge and its relatives at the
corresponding edge (black) and the opposite edge (blue). (Example
$n=7$, $k=4$)}
\end{figure}

\item Any rhomb $R_{2}$ on one edge implies the existence of a line segment
$R_{0}$ on the corresponding edge. In turn, rhomb $R_{2}$ on one
edge is implied by a rhomb $R_{4}$ on the corresponding edge or a
rhomb $R_{2}$ on the opposite edge.
\item Any rhomb $R_{1}$ on one edge implies the existence of a rhomb $R_{1}$
on the corresponding edge. In turn, rhomb $R_{1}$ on one edge is
implied by a rhomb $R_{3}$ on the corresponding edge or a rhomb $R_{1}$
on the opposite edge.
\item Any line segment $R_{0}$ on one edge is implied by a rhomb $R_{2}$
on the corresponding edge or a line segment $R_{0}$ on the opposite
edge.
\item Any rhomb $R_{n-2}$ on one edge implies the existence of a rhomb
$R_{n-4}$ on the corresponding edge. In turn, rhomb $R_{n-2}$ on
one edge is implied by a rhomb $R_{n-2}$ on the opposite edge. (Rhomb
$R_{n}$ does not exist, because the inner angle would be zero.)
\item Any rhomb $R_{n-1}$ on one edge implies the existence of a rhomb
$R_{n-3}$ on the corresponding edge. In turn, rhomb $R_{n-1}$ on
one edge is implied by a rhomb $R_{n-1}$ on the opposite edge. (Rhomb
$R_{n+1}$ does not exist, it would have an inner angle greater than
$\pi$ or smaller the $0$.)
\item If for a rhomb $R_{k},\;n>k>0$ on one edge two related elements (rhomb
or line segment) $R_{|k-2|}$ and $R_{k+2}$ exist on the corresponding
edge, $R_{|k-2|}$ is closer to the tip than $R_{k+2}$
\end{enumerate}
Since all edges are congruent we can derive the following inequalities
for $\mathbf{\alpha}_{k}$:

\begin{equation}
\alpha_{0}\geq\alpha_{2}\geq\alpha_{4}\geq...\alpha_{2\left\lfloor (n-1)/2\right\rfloor }\geq0\;\;\;\;\;\;(\mathbf{\alpha}_{k}\in\mathbb{\mathbb{N}}_{0},\;cases\;1\;and\;3)\label{eq:rhomb_inflation_multiplier_by_rhombs_1}
\end{equation}

\begin{equation}
\alpha_{1}\geq\alpha_{3}\geq\alpha_{5}\geq...\alpha_{2\left\lfloor n/2\right\rfloor -1}\geq0\;\;\;\;\;\;(\mathbf{\alpha}_{k}\in\mathbb{\mathbb{N}}_{0},\;cases\;2\;and\;4)\label{eq:rhomb_inflation_multiplier_by_rhombs_2}
\end{equation}

Because of Equations~\eqref{eq:diagonal-to-sum-of-n-roots-of-unity}
and~\eqref{eq:diagonal-to-sum-of-n-roots-of-unity-2}, the inflation
multiplier can also be written as a sum of diagonals $\mu_{n,k}$:

\begin{equation}
\eta=\sum_{k=1}^{\left\lfloor n/2\right\rfloor }\beta_{k}\mu_{n,k}\;\;\;\;\;\;(\mathbf{\beta}_{k}\in\mathbb{\mathbb{N}}_{0},\;cases\;1\;and\;3)\label{eq:rhomb_inflation_multiplier_by_diagonals_1}
\end{equation}

\begin{equation}
\eta=\sqrt{\mu_{n,2}+2}\sum_{k=1}^{\left\lfloor n/2\right\rfloor }\beta_{k}\mu_{n,k}\;\;\;\;\;\;(\beta_{k}\in\mathbb{\mathbb{N}}_{0},\;cases\;2\;and\;4)\label{eq:rhomb_inflation_multiplier_by_diagonals_2}
\end{equation}

Equations~\eqref{eq:rhomb_inflation_multiplier_by_rhomb_diagonals_1}~-~\eqref{eq:rhomb_inflation_multiplier_by_roots_of_unity_2}
give no hint on which $\alpha_{k}=0$, so that a rhombic CAST still
can exist. For this reason, we have to extend our observations to
the substitution rules of rhomb $R_{\left\lfloor n/2\right\rfloor }$.

For case 1b and 3b:
\begin{itemize}
\item Any line segment $R_{0}$ on the edge implies the existence of a rhomb
$R_{n-1}$ on the correspondent edge or a line segment $R_{0}$ on
the opposite edge. As shown in Fig.~\ref{fig:Rn-3_must_exist}~(a)
and~(b), the existence of rhomb $R_{n-1}$ on the edge is not required
to meet the KSK criterion.
\item Any rhomb $R_{2}$ on the edge implies the existence of a rhomb $R_{n-3}$
on the correspondent edge or a rhomb $R_{2}$ on the opposite edge.
As shown in Fig.~\ref{fig:Rn-3_must_exist}~(c) and~(d), the KSK
criterion is only met if at least one $R_{n-3}$ exists on the edge.
\end{itemize}
\begin{equation}
\alpha_{n-3}>\alpha_{n-1}\geq0\;\;\;\;\;\;(cases\;1b\;and\;3b)\label{eq:case-1b-alpha-1}
\end{equation}

For case 1a and 3a:
\begin{itemize}
\item The line segment $R_{0}$ on the edge implies the existence of a line
segment $R_{0}$ on the opposite edge only (rhomb $R_{n}$ does not
exist).
\item Any rhomb $R_{2}$ on the edge implies the existence of a rhomb $R_{n-2}$
on the correspondent edge or a rhomb $R_{2}$ on the opposite edge.
So the KSK criterion is only met if at least one $R_{n-2}$ exists
on the edge.
\end{itemize}
\begin{equation}
\alpha_{n-2}\geq1\;\;\;\;\;\;(cases\;1a\;and\;3a)
\end{equation}

For case 2b and 4b:
\begin{itemize}
\item Any rhomb $R_{1}$ on the edge implies the existence of a rhomb $R_{n-2}$
on the correspondent edge or a rhomb $R_{1}$ on the opposite edge.
So the KSK criterion is only met if at least one $R_{n-2}$ exists
on the edge.
\end{itemize}
\begin{equation}
\alpha_{n-2}\geq1\;\;\;\;\;\;(cases\;2b\;and\;4b)
\end{equation}

For case 2a and 4a:
\begin{itemize}
\item Any rhomb $R_{1}$ on the edge implies the existence of a rhomb $R_{n-1}$
on the correspondent edge or a rhomb $R_{1}$ on the opposite edge.
So the KSK criterion is only met if at least one $R_{n-1}$ exists
on the edge.
\end{itemize}
\begin{equation}
\alpha_{n-1}\geq1\;\;\;\;\;\;(cases\;2a\;and\;4a)
\end{equation}

\begin{figure}
\begin{center}
\resizebox{0.8\textwidth}{!}{%

\includegraphics{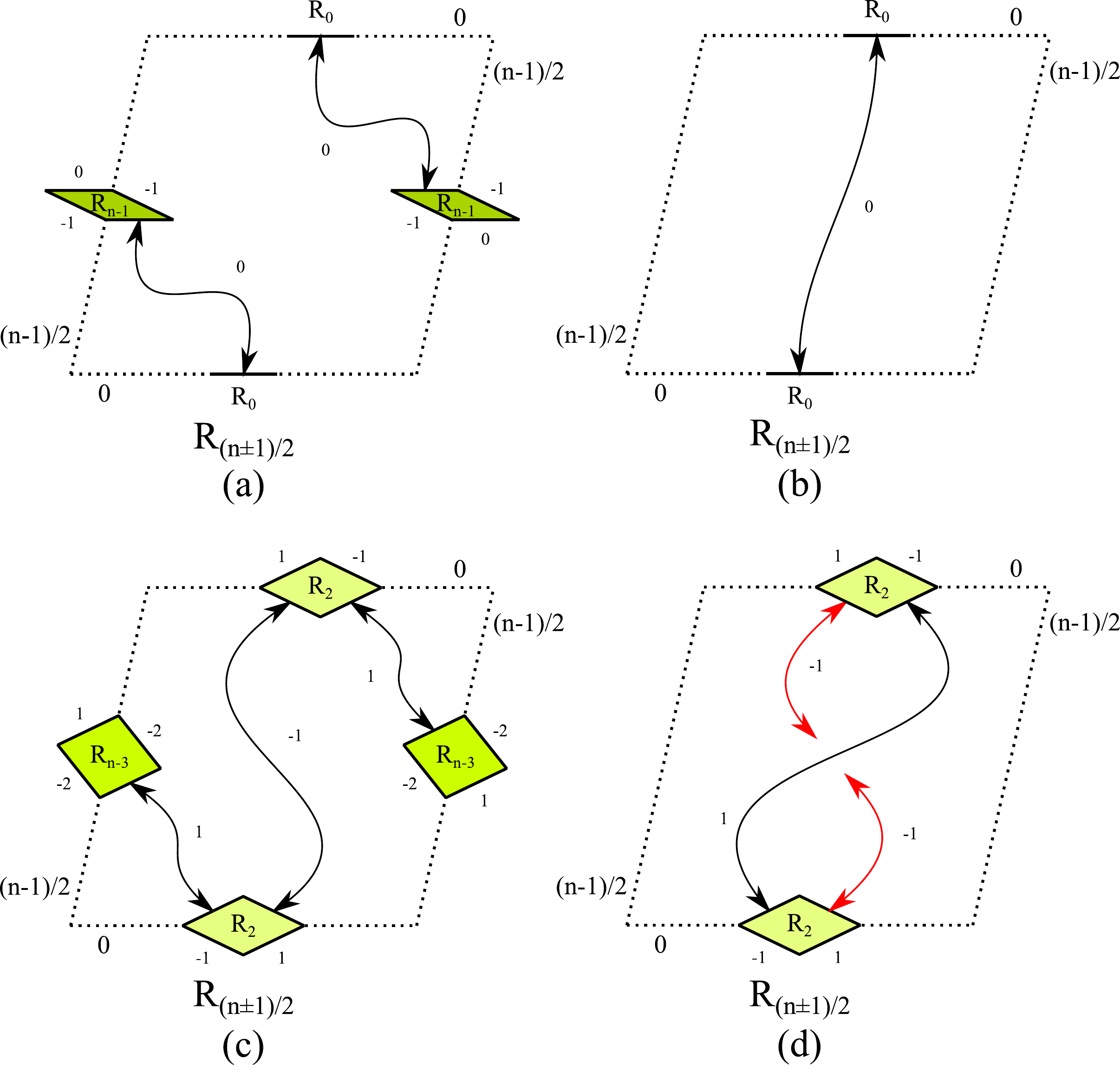}}
\end{center}\caption{\label{fig:Rn-3_must_exist}Substitution rule for rhomb $R_{(n\pm1)/2}$:
The KSK criterion requires rhomb $R_{n-3}$ but not rhomb $R_{n-1}$
on the edge. (Case 1b, example $n=7$)}
\end{figure}

In the next part we will determine the minimal inflation multiplier
$\eta_{min}$ of a rhombic CAST. We will discuss only case 1 in detail
to show the general concept.

We continue our observations of the substitution rule $R_{1}$ or
$R_{n-1}$. There are four possible combinations of the corresponding
edge orientations as shown in Fig.~\ref{fig:tip-of-rhomb-configuration-edge}.
\\
\begin{figure}[H]
\begin{center}
\resizebox{0.7\textwidth}{!}{%

\includegraphics{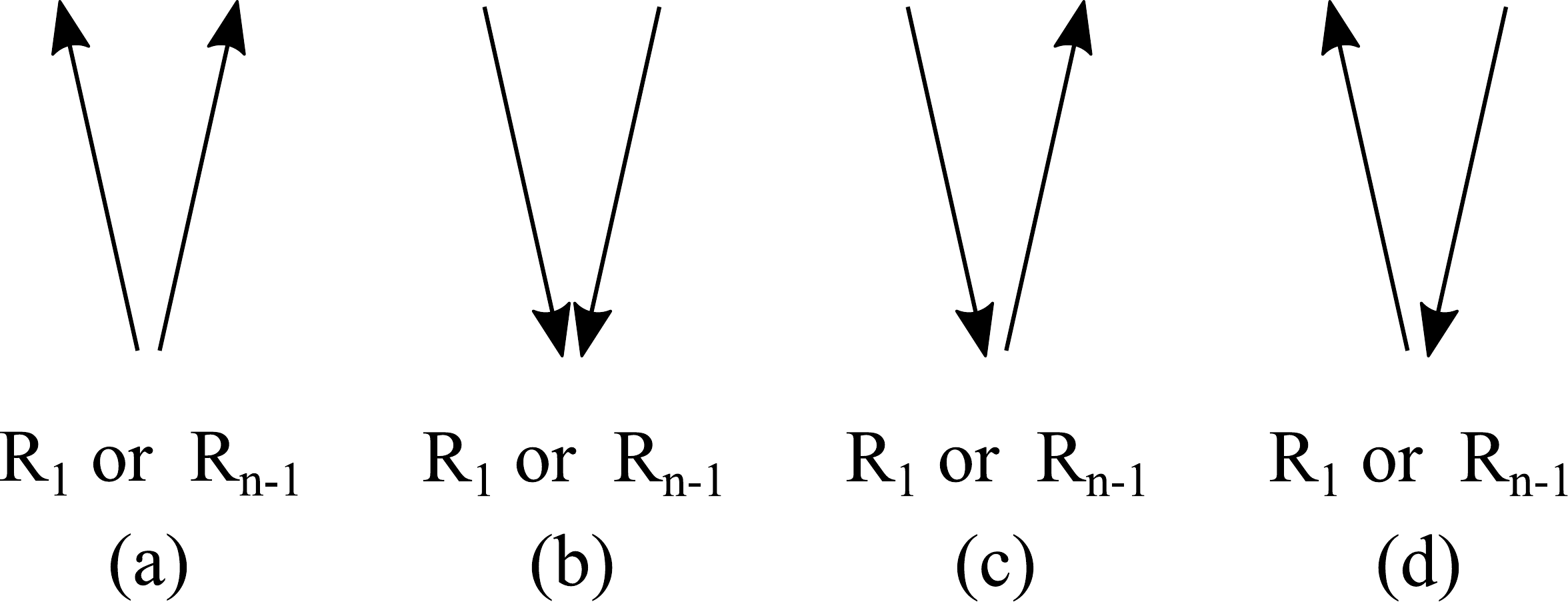}}
\end{center}\caption{\label{fig:tip-of-rhomb-configuration-edge}Substitution rule for
rhomb $R_{1}$ or $R_{n-1}$: Possible configurations of the edges
orientations. (Example $n=7$)}
\end{figure}
We recall Equations~\eqref{eq:rhomb_inflation_multiplier_by_rhombs_1}
and~\eqref{eq:case-1b-alpha-1}. To get a small inflation multiplier,
the $\alpha_{k}$ with $even\;k$ have to be chosen as small as possible.
We use the KSK criterion to check whether substitution rules can exist
for different orientations of the corresponding edges and equal amounts
of rhombs $R_{k}$ and $R_{k+2}$ so that $\alpha_{k}=\alpha_{k+2}$.
The checks are illustrated in Fig.~\ref{fig:tip-of-rhomb-configuration-edge-orientation}.
Actually, the figure just shows the case $\alpha_{k}=\alpha_{k+2}=1$,
while the result is true for $\alpha_{k}=\alpha_{k+2}>1$ as well.

The results for $\alpha_{0}=\alpha_{2}$ in Fig.~\ref{fig:tip-of-rhomb-configuration-edge-orientation}~(a),~(b),~(c)
show that the KSK criterion is fulfilled for more than one combination.

The results for $\alpha_{k}=\alpha_{k+2};\;k>0$ in Fig.~\ref{fig:tip-of-rhomb-configuration-edge-orientation}~(d),~(e),~(f)
show that the KSK criterion is fulfilled for one orientation only.

\begin{figure}
\begin{center}
\resizebox{\textwidth}{!}{%

\includegraphics{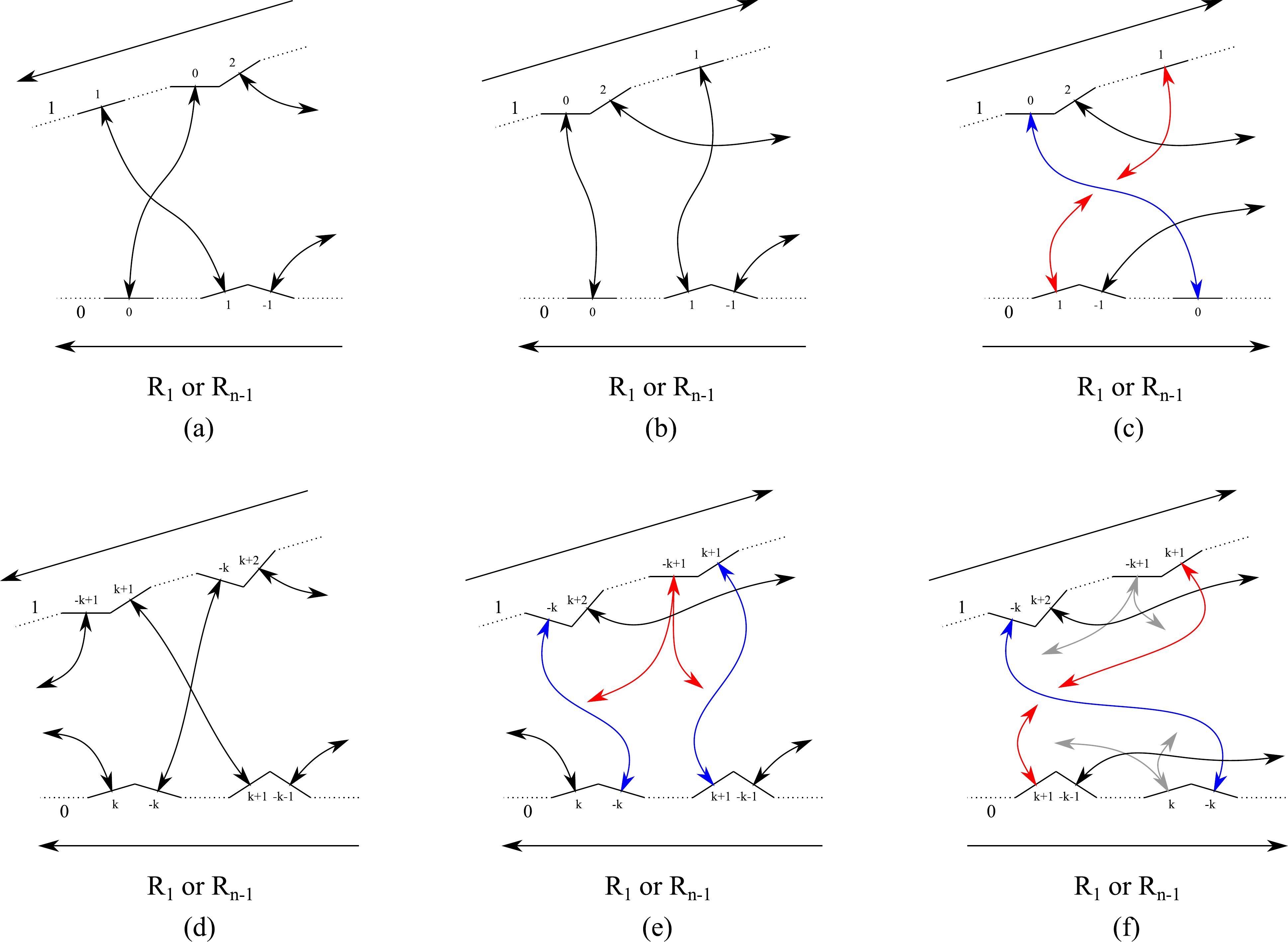}}
\end{center}\caption{\label{fig:tip-of-rhomb-configuration-edge-orientation}Substitution
rule for rhomb $R_{1}$ or $R_{n-1}$: KSK criterion for corresponding
edges with different orientations with $\alpha_{0}=\alpha_{2}$ and
$\alpha_{2k}=\alpha_{2k+2};\;k>1$.}
\end{figure}

At this point we have to introduce an additional condition. For cases
1 and 2, we consider only rhombic CASTs whose prototiles $R_{1}$
or $R_{n-1}$ yield at least two different orientations of the corresponding
edges.
\begin{rem}
The author is only aware of one example of a rhombic CAST whose prototiles
$R_{1}$ or $R_{n-1}$ show only one orientation of the corresponding
edges, namely the Ammann-Beenker tiling with $n=4$. For $n>4$, the
existence of such rhombic CASTs seem to be unlikely. However, a proof
to rule out this possibility is not available yet.
\end{rem}
With the results and conditions above we can conclude:

\begin{equation}
\alpha_{0}\geq\alpha_{2}\;\;\;\;\;\;(case\;1)\label{eq:case-1b-alpha-2}
\end{equation}

\begin{equation}
\alpha_{2k}>\alpha_{2k+2};\;k>0\;\;\;\;\;\;(case\;1)\label{eq:case-1b-alpha-3}
\end{equation}

With Equations~\eqref{eq:diagonal-to-sum-of-n-roots-of-unity-2},
\eqref{eq:rhomb_inflation_multiplier_by_roots_of_unity_1}, \eqref{eq:rhomb_inflation_multiplier_by_rhombs_1}
and~\eqref{eq:rhomb_inflation_multiplier_by_diagonals_1}, we can
describe the minimal inflation multiplier:

\begin{equation}
\eta{}_{min}=\sum_{i=1}^{\left\lfloor n/2-1\right\rfloor }\left(\mu_{n,i}+\mu_{n,i+1}\right)\;\;\;\;\;\;(case\;1)
\end{equation}

For case 1b with $odd\;n$, we can use the same concept of orientations
as shown in \citep[Fig. 2]{Nischke1996} which leads to rhombs with
orientations as shown in Fig.~\ref{fig:Case1b-rhomb-edge-configuration}.
As a result, the orientations of the edges are globally defined. In
detail, the edges which enclose an inner angle $\frac{\pi}{n}$ have
orientations as shown in Fig.~\ref{fig:tip-of-rhomb-configuration-edge}~(c)
and~(d) and a tip as shown in Fig.~\ref{fig:tip-of-rhomb-configuration}~(a).

For case 1a with $even\;n$ this simplification can not be used. The
edges which enclose an inner angle $\frac{\pi}{n}$ have not only
orientations as shown in Fig.~\ref{fig:tip-of-rhomb-configuration-edge}~(c)
and~(d). Depending on how exactly the orientations are defined one
of the orientations in Fig.~\ref{fig:tip-of-rhomb-configuration-edge}~(a)
or~(b) may appear. 

\begin{figure}[H]
\begin{center}
\resizebox{\textwidth}{!}{%

\includegraphics[bb=0bp 0bp 747bp 380bp]{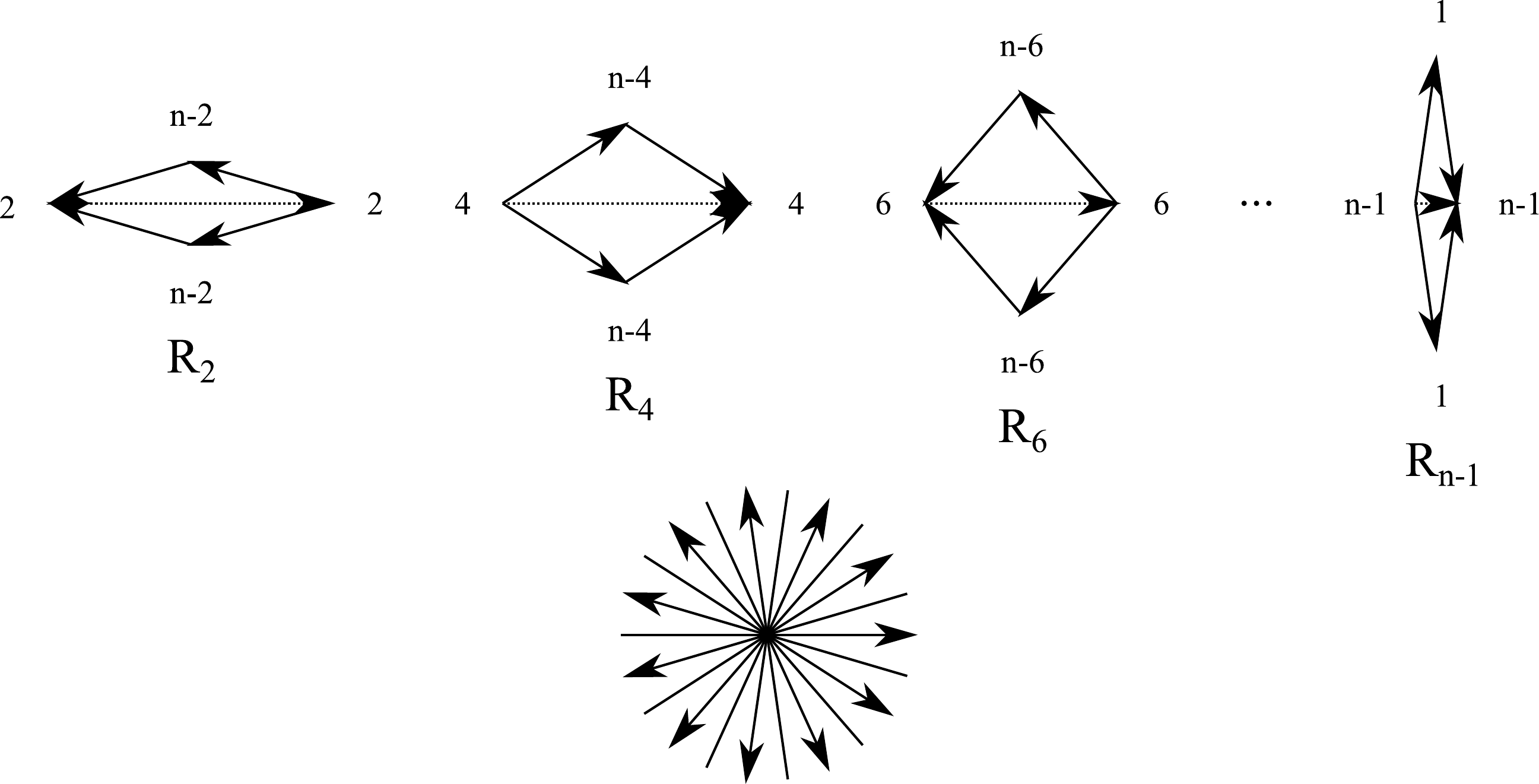}}
\end{center}\caption{\label{fig:Case1b-rhomb-edge-configuration}Orientations of edges
(and orientations of lines of symmetry) for rhomb prototiles in case
1b. (Example $n=11$)}
\end{figure}
Similar approaches are possible for cases~2, 3 and~4. For cases~3
and~4 the symmetry conditions of the substitution rules and its edges
enforce $\alpha_{k}\;even,\;k>1$, just $\alpha_{0}$ is not required
to be even.

With the minimal inflation multiplier $\eta{}_{min}$ as shown in
Table~\ref{tab:Rhomb-CAST-edgeconf-min-infl-mult-case1}~-~\ref{tab:Rhomb-CAST-edgeconf-min-infl-mult-case4},
we can start the search of rhombic CASTs by trial and error. The minimal
rhomb edge sequences therein were also obtained in this way, so additional
solutions with varied sequences obtained by permutation may exist.
Examples for all eight cases are shown in Fig.~\ref{fig:Rhomb-CAST-Case1}~-~\ref{fig:Rhomb-CAST-Case4}.
\begin{conjecture}
\label{conj:Rhomb-Casts-for-all-cases}For all cases as listed in
Table~\ref{tab:Definition-of-cases}, there exist rhombic CASTs for
any $n$ with edge configuration and inflation multiplier $\eta_{min}$
as shown in Tables~\ref{tab:Rhomb-CAST-edgeconf-min-infl-mult-case1}~-~\ref{tab:Rhomb-CAST-edgeconf-min-infl-mult-case4}
which yield individual dihedral symmetry $D_{n}$ or $D_{2n}$.
\end{conjecture}
A general proof of Conjecture~\ref{conj:Rhomb-Casts-for-all-cases}
is subject to further research. However, a proof for case 4b already
exists. The results in \citep{2015arXiv151201402K,Kari2016} for $odd\;n$
match very well with our results in case 4b. The case $even\;n$ in
the same publication is almost identical to our case 3a, up to a difference
of one line segment $R_{0}$ in the center of the edge of the substitution
rules.
\begin{rem}
Case 1a with $n=4$ is equivalent to the Ammann-Beenker tiling \citep{Beenker,Socolar1989,journals/dcg/AmmannGS92,HFonl}. 
\end{rem}

\begin{rem}
Socolar's 7-fold tiling \citep[credited to J. Socolar]{HFonl} can
be derived from case 1b with $n=7$ with an modified minimal rhomb
edge sequence.
\end{rem}

\begin{rem}
Two substitution steps of the Penrose rhomb tiling are equivalent
to one substitution step in case 1b with $n=5$. 
\end{rem}
\begin{table}[H]
\caption{\label{tab:Rhomb-CAST-edgeconf-min-infl-mult-case1}Rhombic CAST substitution
rule edge configuration for case 1. }

\begin{center}
\resizebox{\textwidth}{!}{%

\begin{tabular}{|c|c|l|}
\hline 
$n$ & minimal rhomb edge sequence & minimal inflation multiplier $\eta_{min}$\tabularnewline
\hline 
\hline 
$4,5$ & $\;\;\;\;\;\;\;\;\;\;\;\;\;\;\;\;\;\;\;\;\;\;\;\;\;\overline{0-2}$ & $\mu_{n,2}+1$\tabularnewline
\hline 
$6,7$ & $\;\;\;\;\;\;\:\overline{0-2}-4-\overline{0-2}$ & $\mu_{n,3}+2\mu_{n,2}+1$\tabularnewline
\hline 
$8,9$ & $\;\;\;\;\;\;\;\;\;\;\;\;\;\;\;\;\;\;\;\;\;\;\;\;\;\;\;\;\;\;\;\overline{0-2}-4-\overline{0-2}-6-4-\overline{0-2}$ & $\mu_{n,4}+2\mu_{n,3}+2\mu_{n,2}+1$\tabularnewline
\hline 
$10,11$ & $\overline{0-2}-4-6-8-\overline{0-2}-4-\overline{0-2}-6-4-\overline{0-2}$ & $\mu_{n,5}+2\mu_{n,4}+2\mu_{n,3}+2\mu_{n,2}+1$\tabularnewline
\hline 
$...$ & $...$ & $...$\tabularnewline
\hline 
\end{tabular}

}
\end{center}
\end{table}

\begin{table}[H]
\caption{\label{tab:Rhomb-CAST-edgeconf-min-infl-mult-case2}Rhombic CAST substitution
rule edge configuration for case 2}

\begin{center}
\resizebox{\textwidth}{!}{%

\begin{tabular}{|c|c|l|}
\hline 
$n$ & minimal rhomb edge sequence & minimal inflation multiplier $\eta_{min}$\tabularnewline
\hline 
\hline 
$4,5$ & $1-\,\,\,\,\,\,\,\,\,\,\,\,\,\,\,\,\,\,\,\,\,\,\,\,\,\,\,\,\,\,\,\,\,\,\,\,\,\,\,\,\,\,\,\,\,\,\,\,\,\,\,\,\,\,\:\:\:\:\:\:3\:\:\:\:\:\:\,\,\,\,\,\,\,\,\,\,\,\,\,\,\,\,\,\,\,\,\,\,\,\,\,\,\,\,\,\,\,\,\,\,\,\,\,\,\,\,\,\,\,\,\,\,\,\,\,\,\,\,\,\,-1$ & $\sqrt{\mu_{n,2}+2}\;\left(\mu_{n,2}+2\right)$\tabularnewline
\hline 
$6,7$ & $1-3-\,\,\,\,\,\,\,\,\,\,\,\,\,\,\,\,\,\,\,\,\,\,\,\,\,\,\,\,\,\,\,\,\,\,\,\,\,\,\,\,\,\:\:\:\:1-5\:\:\:\:\,\,\,\,\,\,\,\,\,\,\,\,\,\,\,\,\,\,\,\,\,\,\,\,\,\,\,\,\,\,\,\,\,\,\,\,\,\,\,\,\,-3-1$ & $\sqrt{\mu_{n,2}+2}\;\left(\mu_{n,3}+2\mu_{n,2}+2\right)$\tabularnewline
\hline 
$8,9$ & $1-3-5-\,\,\,\,\,\,\,\,\,\,1-\,\,\,\,\,\,\,\,\,\,\,\,\,\:\:3-7\:\:\,\,\,\,\,\,\,\,\,\,\,\,\,-1\,\,\,\,\,\,\,\,\,\,-5-3-1$ & $\sqrt{\mu_{n,2}+2}\;\left(\mu_{n,4}+2\mu_{n,3}+2\mu_{n,2}+2\right)$\tabularnewline
\hline 
$10,11$ & $1-3-5-7-1-3-1-5-9-3-1-7-5-3-1$ & $\sqrt{\mu_{n,2}+2}\;\left(\mu_{n,5}+2\mu_{n,4}+2\mu_{n,3}+2\mu_{n,2}+2\right)$\tabularnewline
\hline 
$...$ & $...$ & $...$\tabularnewline
\hline 
\end{tabular}

}
\end{center}
\end{table}

\begin{table}[H]
\caption{Rhombic CAST substitution rule edge configuration for case 3}

\begin{center}
\resizebox{\textwidth}{!}{%

\begin{tabular}{|c|c|l|}
\hline 
$n$ & minimal rhomb edge sequence & minimal inflation multiplier $\eta_{min}$\tabularnewline
\hline 
\hline 
$4,5$ & $0-2-0-2-0$ & $2\mu_{n,2}+3$\tabularnewline
\hline 
$6,7$ & $0-2-4-0-2-0-2-0-4-2-0$ & $2\mu_{n,3}+4\mu_{n,2}+3$\tabularnewline
\hline 
$8,9$ & $0-2-4-6-0-2-4-0-2-0-2-0-4-2-0-6-4-2-0$ & $2\mu_{n,4}+4\mu_{n,3}+4\mu_{n,2}+3$\tabularnewline
\hline 
$...$ & $...$ & $...$\tabularnewline
\hline 
\end{tabular}

}
\end{center}
\end{table}
\begin{table}[H]
\caption{\label{tab:Rhomb-CAST-edgeconf-min-infl-mult-case4}Rhombic CAST substitution
rule edge configuration for case 4}

\begin{center}
\resizebox{\textwidth}{!}{%

\begin{tabular}{|c|c|l|}
\hline 
$n$ & minimal rhomb edge sequence & minimal inflation multiplier $\eta_{min}$\tabularnewline
\hline 
\hline 
$4,5$ & $1-3-1-1-3-1$ & $\sqrt{\mu_{n,2}+2}\;\left(2\mu_{n,2}+4\right)$\tabularnewline
\hline 
$6,7$ & $1-3-5-1-3-1-1-3-1-5-3-1$ & $\sqrt{\mu_{n,2}+2}\;\left(2\mu_{n,3}+4\mu_{n,2}+4\right)$\tabularnewline
\hline 
$8,9$ & $1-3-5-7-1-3-5-1-3-1-1-3-1-5-3-1-7-5-3-1$ & $\sqrt{\mu_{n,2}+2}\;\left(2\mu_{n,4}+4\mu_{n,3}+4\mu_{n,2}+4\right)$\tabularnewline
\hline 
$...$ & $...$ & $...$\tabularnewline
\hline 
\end{tabular}

}
\end{center}
\end{table}
\clearpage{}
\begin{figure}
\begin{center}
\resizebox{\textwidth}{!}{%

\includegraphics{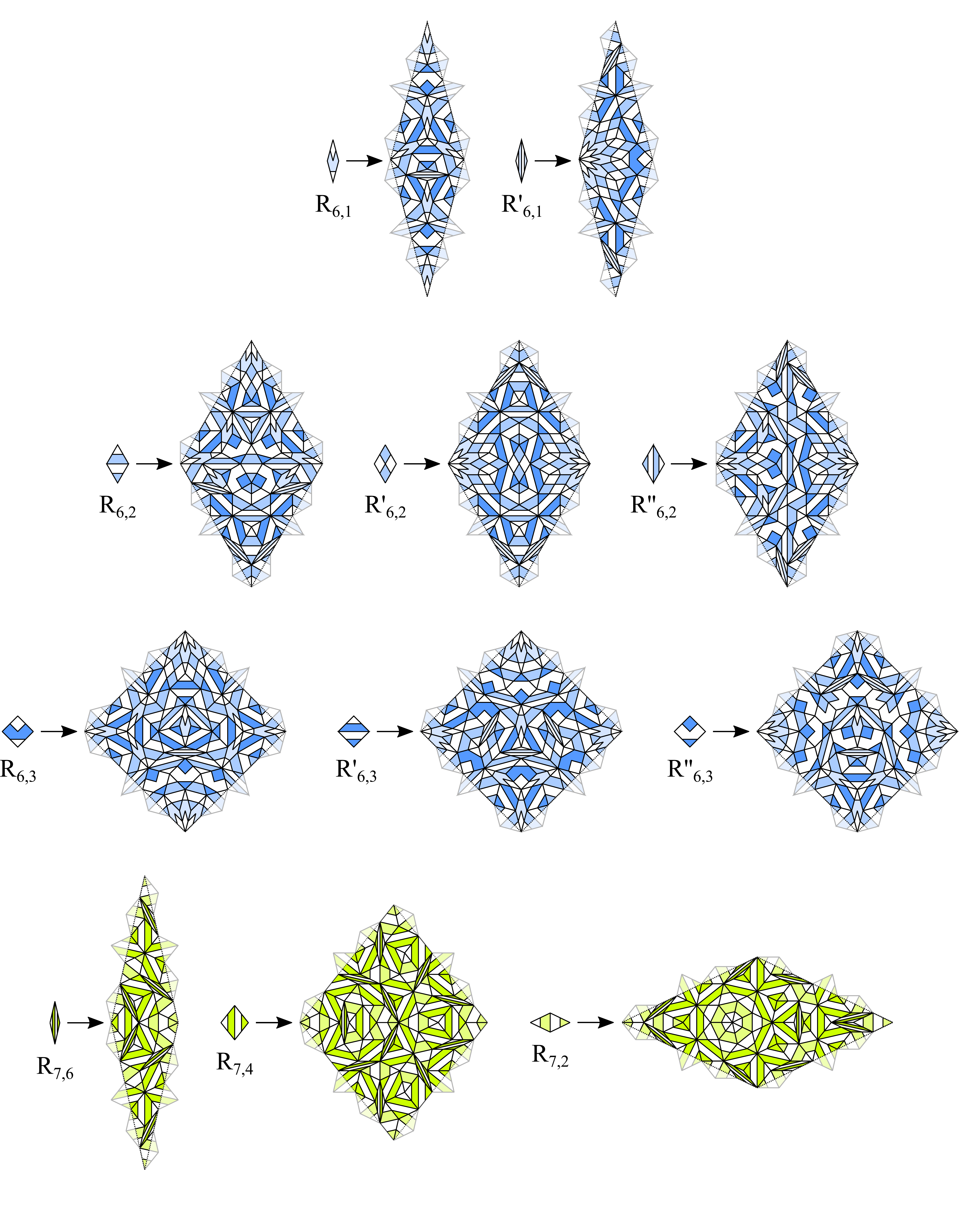}}
\end{center}\caption{\label{fig:Rhomb-CAST-Case1}Rhombic CAST examples for case 1a ($n=6$)
and case 1b ($n=7$)}
\end{figure}
\begin{figure}
\begin{center}
\resizebox{\textwidth}{!}{%

\includegraphics{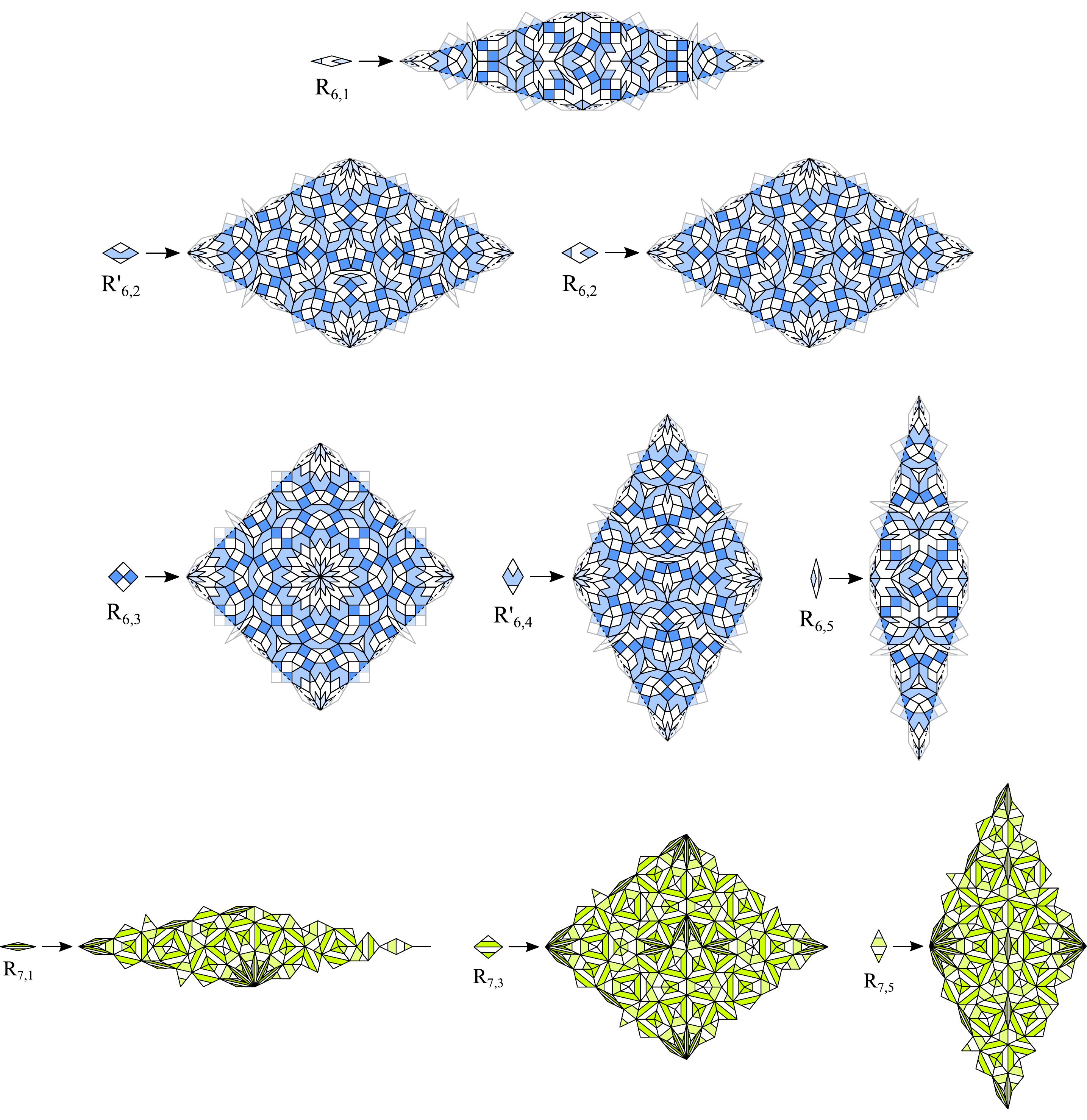}}
\end{center}\caption{\label{fig:Rhomb-CAST-Case2}Rhombic CAST examples for case 2a ($n=6$)
and case 2b ($n=7$). The shown example for case 2b was slightly modified
to reduce the number of prototiles to $\left\lfloor n/2\right\rfloor $
as in case 1b. In detail, the edges of the rhomb prototiles have orientation
as shown in Fig.~\ref{fig:Case1b-rhomb-edge-configuration}.}
\end{figure}
\begin{figure}
\begin{center}
\resizebox{\textwidth}{!}{%

\includegraphics{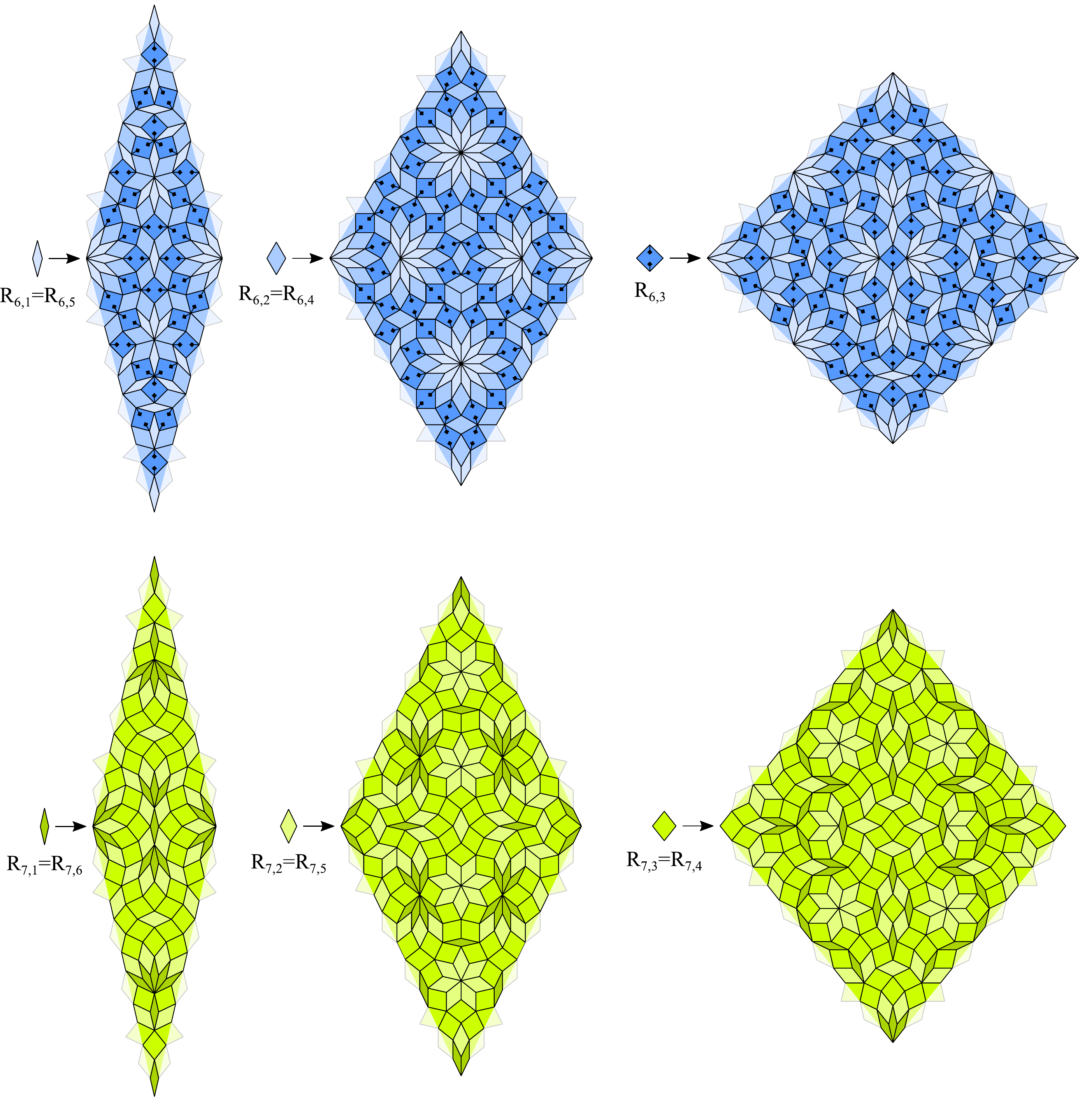}}
\end{center}\caption{\label{fig:Rhomb-CAST-Case3}Rhombic CAST examples for case 3a ($n=6$)
and case 3b ($n=7$)}
\end{figure}
\begin{figure}
\begin{center}
\resizebox{\textwidth}{!}{%

\includegraphics{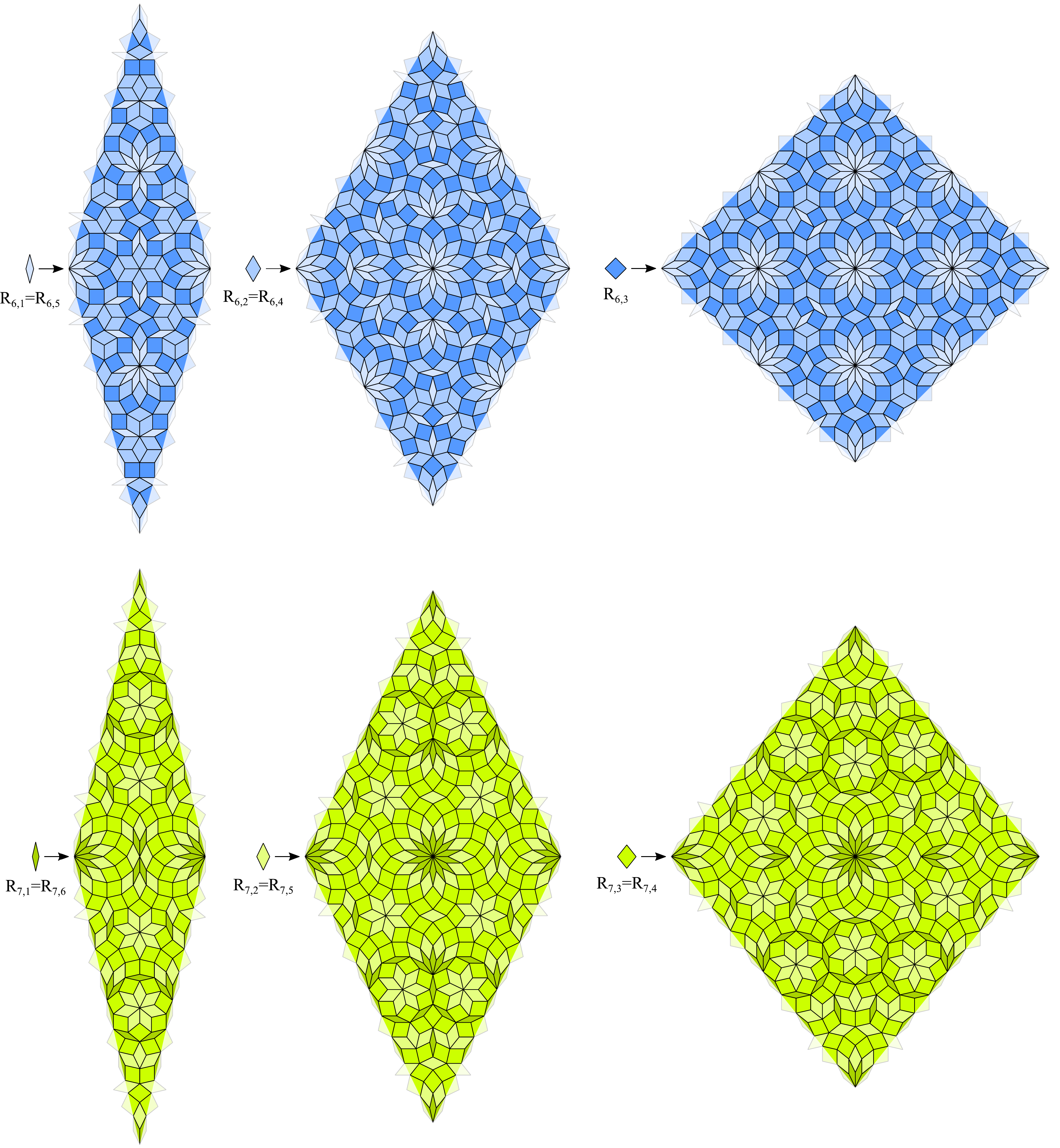}}
\end{center}\caption{\label{fig:Rhomb-CAST-Case4}Rhombic CAST examples for case 4a ($n=6$)
and case 4b ($n=7$)}
\end{figure}

\clearpage{}

\section{\label{sec:Gaps_to_Prototiles_Algorithm}Gaps to Prototiles Algorithm}

In this section, we sketch a “Gaps to Prototiles'' algorithm to identify
CASTs for a given $n$ and a selected edge of a substitution rule
(as in Definition~\ref{def:edge_defintion}).

Conditions:
\begin{itemize}
\item All prototiles have inner angles equal $\frac{k\pi}{n}$.
\item All edges of all substitution rules are congruent and have dihedral
symmetry $D_{2}$.
\item As discussed in Section~\ref{sec:Rhomic_CASTs_with_symmetric_edges_and_sunbstitution_rules}
the tiles on the edge have to be placed, so that the inner angles
either with even or odd multiples of $\frac{\pi}{n}$ are bisected
by the boundary of the supertile.
\item The tiles on the edge are bisected by one or two lines of symmetry
of the edge. This implies dihedral symmetry $D_{1}$ or $D_{2}$ of
the corresponding substitution rules. 
\item The inflation multiplier $\eta$ must fulfill the conditions in Theorem~\ref{thm: CAST1}.
\item The inflation multiplier $\eta$ is defined by the sequence of tiles
which are part of the edge.
\end{itemize}
Algorithm:
\begin{enumerate}
\item We start with the prototiles which appear on the edge of the substitution
rule.
\item \label{enu:step2}We start the construction of the substitution rules
by placing the prototiles on the edge.
\item If the edge prototiles overlap the algorithm has failed. In this case,
we may adjust the sequence of rhombs or other equilateral polygons
on the edge and start another attempt.
\item We try to “fill up” the substitution rules with existing prototiles
under consideration of the appropriate dihedral symmetry $D_{1}$
or $D_{2}$. If gaps remain, they are defined as new prototiles and
we go back to step (\ref{enu:step2}). Please note, if a gap lies
on one or two lines of symmetry, the substitution rule of the new
prototile must also have the appropriate dihedral symmetry $D_{1}$
or $D_{2}$.
\item If no gaps remain the algorithm was successful.
\end{enumerate}
In some cases the “Gaps to Prototiles'' algorithm delivers results
with preferable properties as shown in Fig.~\ref{fig:Residual-CAST-7-1},
\ref{fig:Residual-CAST-7-2} and~\ref{fig:Residual-CAST-4}. However,
some results contain a very large number of prototiles with different
sizes and complex shapes as in Fig.~\ref{fig:Residual-CAST-7-2}.
It is not known yet whether the algorithm always delivers solutions
with a finite number of prototiles. 

\begin{table}[H]
\caption{Inflation multipliers of CASTs identified by the “Gaps to Prototiles''
algorithm}

\begin{center}

\begin{tabular}{|c|c|c|}
\hline 
n & inflation multiplier & \tabularnewline
\hline 
\hline 
7 & $2\mu_{7,3}+2\mu_{7,2}+1$ & Fig.~\ref{fig:Residual-CAST-7-1}\tabularnewline
\hline 
7 & $\mu_{7,2}+2$ & Fig.~\ref{fig:Residual-CAST-7-2}\tabularnewline
\hline 
4 & $\sqrt{\mu_{4,2}+2}$ & Fig.~\ref{fig:Residual-CAST-4}\tabularnewline
\hline 
\end{tabular}

\end{center}
\end{table}

\begin{figure}
\begin{center}
\resizebox{\textwidth}{!}{%

\includegraphics{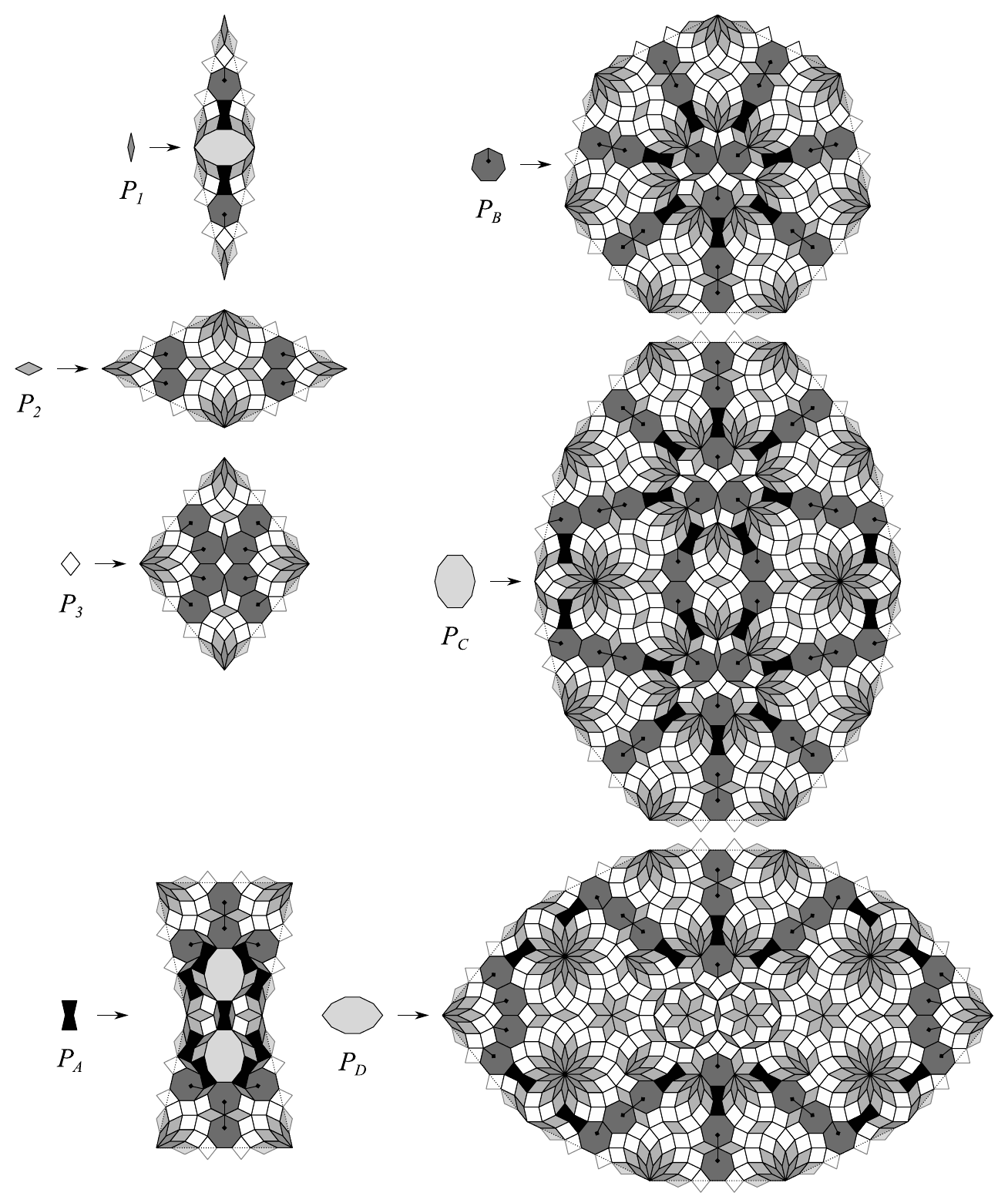}}
\end{center}\caption{\label{fig:Residual-CAST-7-1}CAST for the case $n=7$}
\end{figure}

\begin{figure}
\begin{center}
\resizebox{.9\textwidth}{!}{%

\includegraphics{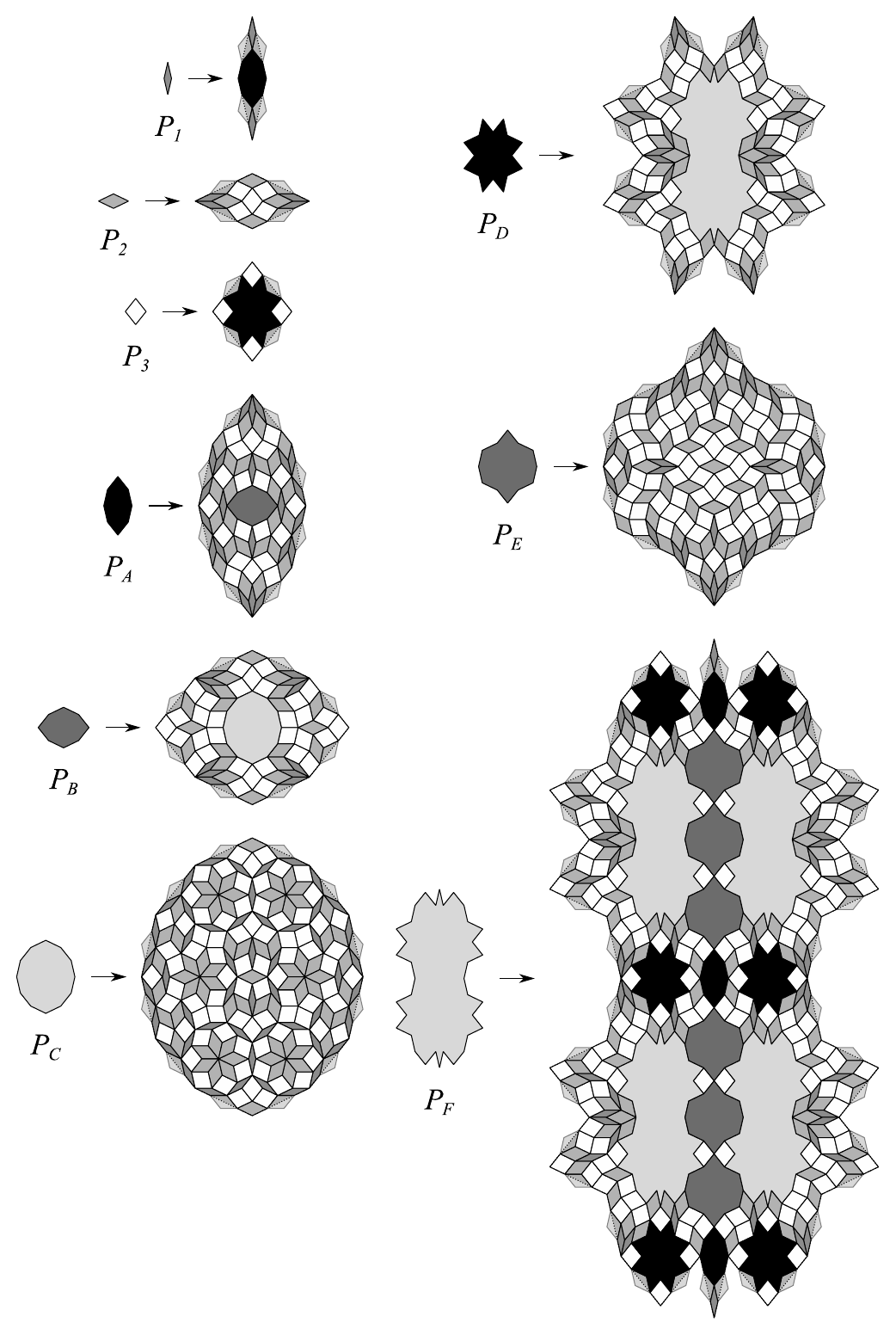}}
\end{center}\caption{\label{fig:Residual-CAST-7-2}CAST for the case $n=7$, derived from
the Goodman-Strauss tiling in \citep{journals/dcg/Harriss05}}
\end{figure}
\begin{figure}
\begin{center}
\resizebox{\textwidth}{!}{%

\includegraphics{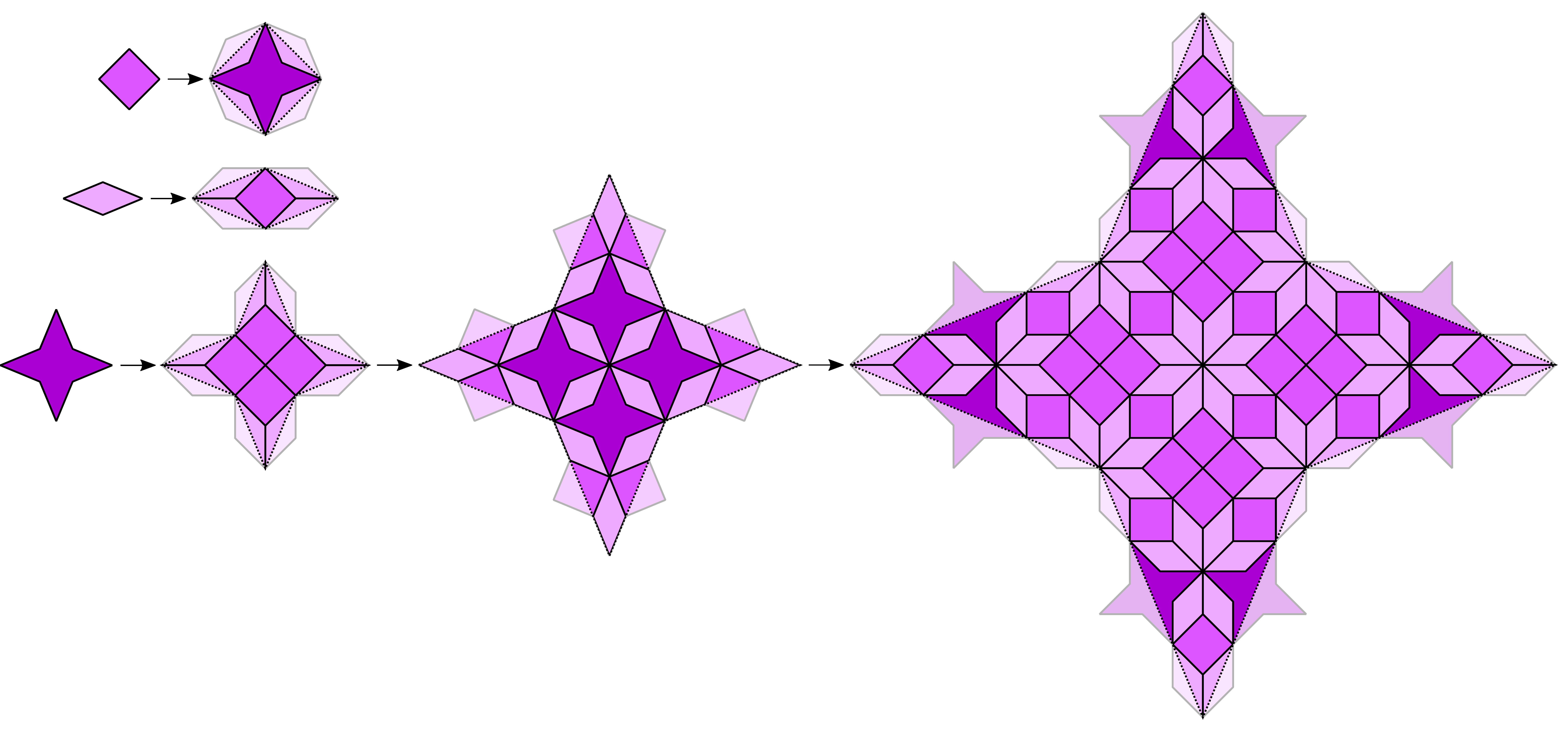}}
\end{center}\caption{\label{fig:Residual-CAST-4}CAST for the case $n=4$, derived from
the generalized Lançon-Billard tiling in Fig.~\ref{fig:LB-Tiling-Generalized}}
\end{figure}
\clearpage{}

\section{\label{sec:CASTs_with_Extended_Girih_Prototiles}Extended Girih CASTs}

\textquotedbl{}Girih\textquotedbl{} is the Persian word for \textquotedbl{}knot\textquotedbl{}
and stands for complex interlaced strap works of lines, which are
a typical feature of Islamic architecture and design. A common definition
is given in \citep{AllenT2004}: \textquotedbl{}Geometric (often star-and-polygon)
designs composed upon or generated from arrays of points from which
construction lines radiate and at which they intersect.\textquotedbl{}
The oldest known examples of star pattern date back to the 8th century
AD \citep{LeeAJ1987}. Girih designs are known in many styles and
symmetries, see \citep{Bourgoin1973} for examples. A variant of Girih
design relies on Girih tiles and tilings as shown in the reproduction
of the Topkapi Scroll in \citep{necipoglu1995topkapi}. The decorations
on the tiles consist of lines which run from tile to tile when they
are joined together. So the borders between joined tiles seem to disappear.

According to \citep{necipoglu1995topkapi} the shapes of all girih
tiles are equilateral polygons with the same side length and inner
angles $\frac{k\pi}{5}$, $k\in\{2,3,4,6\}$ including:
\begin{itemize}
\item Regular decagon with inner angles $\frac{4\pi}{5}$
\item Regular pentagon with inner angles $\frac{3\pi}{5}$
\item Rhomb with inner angles $\frac{2\pi}{5}$ and $\frac{3\pi}{5}$
\item Convex hexagon with inner angles $\frac{2\pi}{5}$,$\frac{4\pi}{5}$,$\frac{4\pi}{5}$,$\frac{2\pi}{5}$,$\frac{4\pi}{5}$,$\frac{4\pi}{5}$
\item Convex hexagon with inner angles $\frac{3\pi}{5}$,$\frac{3\pi}{5}$,$\frac{4\pi}{5}$,$\frac{3\pi}{5}$,$\frac{3\pi}{5}$,$\frac{4\pi}{5}$
\item Nonconvex hexagon with inner angles $\frac{2\pi}{5}$, $\frac{2\pi}{5}$,
$\frac{6\pi}{5}$, $\frac{2\pi}{5}$, $\frac{2\pi}{5}$, $\frac{6\pi}{5}$
\end{itemize}
Because of these properties, Girih tilings are cyclotomic tilings
as well.

A Girih (cyclotomic) aperiodic substitution tiling was derived from
a mosaic at the Darb-i Imam Shrine, Isfahan, Iran which dates back
from 1453. It relies on the regular decagon and two hexagons and has
individual dihedral symmetry $D_{10}$. It was published in \citep{Lu2007,Lu2007a}.
However, the complete set of substitution rules can be found in \citep{bridges2008:297,Cromwell2008}.
Additional examples of Girih CASTs have been discovered and submitted
to \citep{HFonl}. More examples have been published in \citep{LUCK1990832,Lueck1994139,Cromwell2008}.

That rises the question if Girih CASTs with other symmetries and relative
small inflation multiplier exist. For this reasons we have to define
the extended Girih CASTs. The following definition turned out to be
useful for a given $n\geq4$. 
\begin{itemize}
\item All prototiles of an extended Girih CAST are equilateral polygons
with the same side length.
\item The inner angles of all prototiles are $\frac{k\pi}{n}$, $k\in\left\{ 2,3...\left(n-1\right),\left(n+1\right),\left(n+2\right)...\left(n-2\right)\right\} $.
\item One of the prototiles may be a regular $n$-gon with inner angles
$\frac{\left(n-2\right)\pi}{n}$. 
\item One of the prototiles may be a regular $2n$-gon with inner angles
$\frac{\left(n-1\right)\pi}{n}$. 
\end{itemize}
Please note that prototiles with inner angle $\frac{\pi}{n}$ are
forbidden due to aesthetic reasons.

For the extended Girih CASTs in this section we choose the following
properties:
\begin{itemize}
\item All edges of the substitution rules are congruent and have dihedral
symmetry $D_{2}$.
\item All substitution rules except those for regular $n$-gons with $odd\;n$
have dihedral symmetry $D_{2}$.
\item The substitution rule of the regular $n$-gon with $odd\;n$ has dihedral
symmetry $D_{1}$.
\item The substitution rule of the regular $2n$-gon has dihedral symmetry
$D_{2n}$ for $odd\;n$ and $D_{n}$ for $even\;n$.
\end{itemize}
Examples for extended Girih CASTs in this section with $n\in\left\{ 4,5,7\right\} $
are shown in Fig.~\ref{fig:Girih4}~-~\ref{fig:Girih7c}.

The Girih CASTs in Fig.~\ref{fig:Girih4}, \ref{fig:Girih5} and
\ref{fig:Girih7a}~-~\ref{fig:Girih7c} have been obtained by a
trial and error method under the following conditions:
\begin{itemize}
\item In every corner of every substitution rule a regular $2n$-gon is
placed.
\item Edge and inflation multiplier have been derived from a periodic pattern
of regular $2n$-gons and their inter space counterparts.\\
(For $n\geq8$ this approach might fail. As an alternative it is possible
to reuse inflation multipliers from rhombic CASTs in Section~\ref{sec:Rhomic_CASTs_with_symmetric_edges_and_sunbstitution_rules},
case 1 and 2. See Table~\ref{tab:Rhomb-CAST-edgeconf-min-infl-mult-case1}
and \ref{tab:Rhomb-CAST-edgeconf-min-infl-mult-case2} for details.)
\end{itemize}
The Girih CAST in Fig.~\ref{fig:Girih5-Topkapi} has been derived
from \citep[Fig. 14, 15]{Cromwell2008}. In detail the nonconvex hexagons
were replaced and the star shaped gap in the center of the substitution
rule of the pentagon prototile \citep[Fig. 15 (b)]{Cromwell2008}
was eliminated. This was possible by changing the symmetry of the
substitution rule from dihedral symmetry $D_{5}$ to $D_{1}$. The
tilings in \citep[Fig. 14, 15]{Cromwell2008} were derived by an analysis
of patterns shown in the Topkapi Scroll, in detail \citep[Panels 28, 31, 32, 34]{necipoglu1995topkapi}.
\begin{rem}
Please note that the case $even\;n$ requires special care to make
sure that the regular $2n$-gons with dihedral symmetry $D_{n}$ always
match. It seems that the existence of a substitution rule of the regular
$2n$-gons with dihedral symmetry $D_{2n}$ requires the existence
of prototiles with inner angle $\frac{\pi}{n}$ which are forbidden
due to our preconditions. As a result, additional substitution rules
might be necessary for prototiles with the same shape but different
orientations.
\end{rem}

\begin{rem}
The decorations at the prototiles are related but not necessarily
equivalent to Ammann bars. 
\end{rem}
\begin{table}
\caption{Inflation multipliers of extended Girih CASTs}

\begin{center}

\begin{tabular}{|c|c|c|}
\hline 
n & inflation multiplier & \tabularnewline
\hline 
\hline 
4 & $\mu_{4,2}+1$ & Fig.~\ref{fig:Girih4}\tabularnewline
\hline 
5 & $\sqrt{\mu_{5,2}+2}(\mu_{5,2}+1)$ & Fig.~\ref{fig:Girih5}\tabularnewline
\hline 
5 & $2(\mu_{5,2}+1)$ & Fig.~\ref{fig:Girih5-Topkapi}\tabularnewline
\hline 
7 & $\mu_{7,2}+2\mu_{7,1}+2$ & Fig.~\ref{fig:Girih7a},~\ref{fig:Girih7b},~\ref{fig:Girih7c}\tabularnewline
\hline 
\end{tabular}

\end{center}
\end{table}

\begin{figure}
\begin{center}
\resizebox{0.7\textwidth}{!}{%

\includegraphics{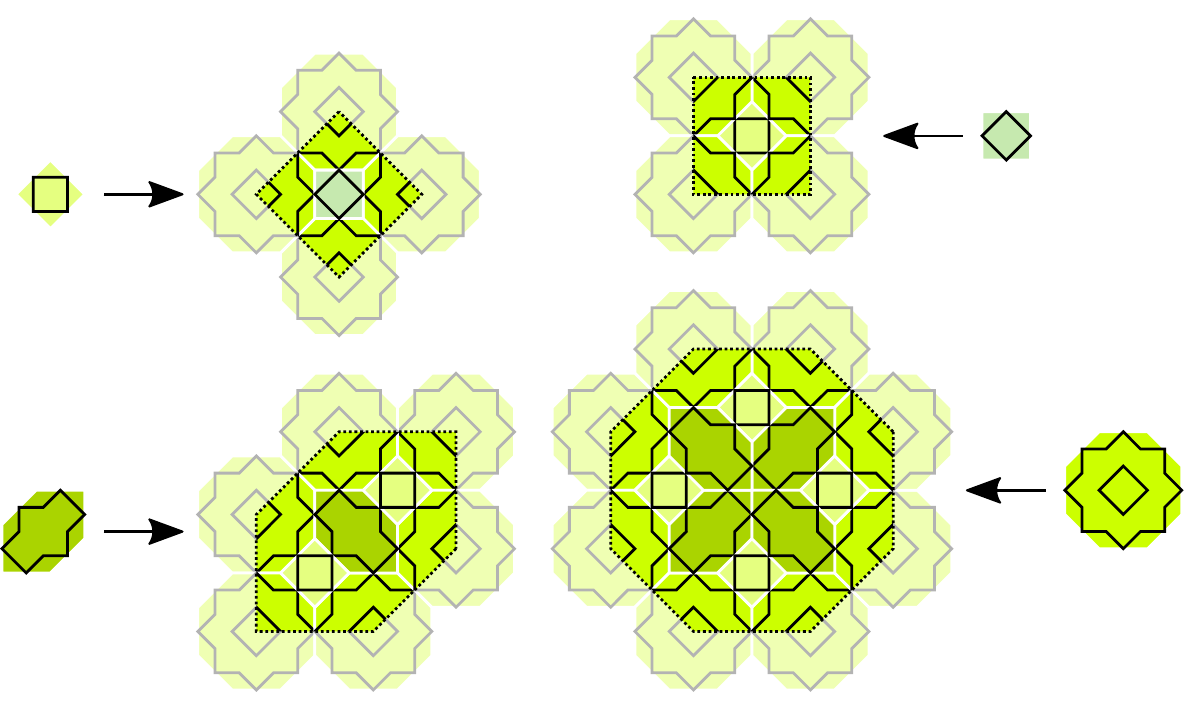}}
\end{center}\caption{\label{fig:Girih4}Extended Girih CAST for the case $n=4$}
\end{figure}
\begin{figure}[H]
\begin{center}
\resizebox{0.9\textwidth}{!}{%

\includegraphics{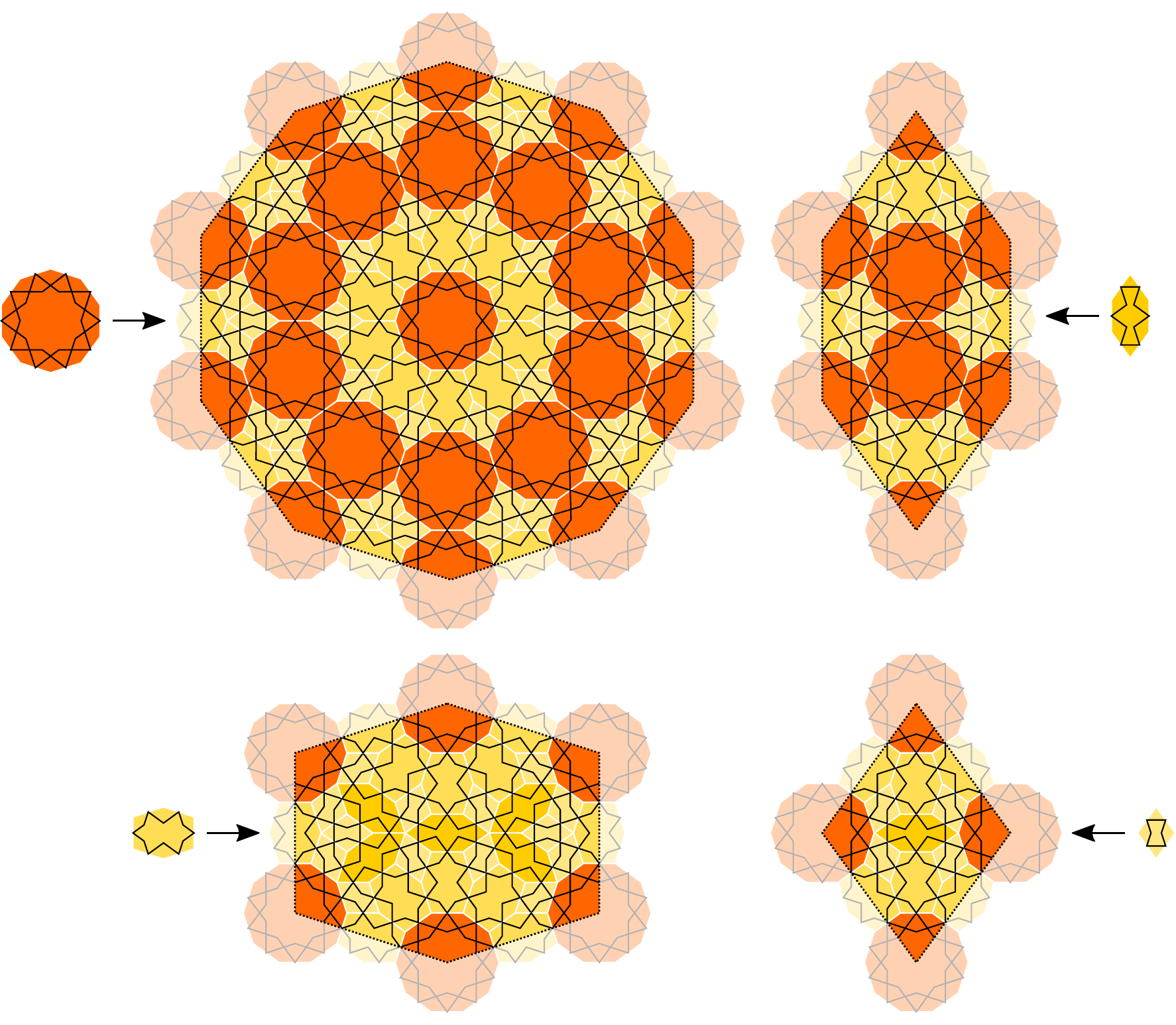}}
\end{center}\caption{\label{fig:Girih5}Girih CAST for the case $n=5$}
\end{figure}
\begin{figure}[H]
\begin{center}
\resizebox{1\textwidth}{!}{%

\includegraphics{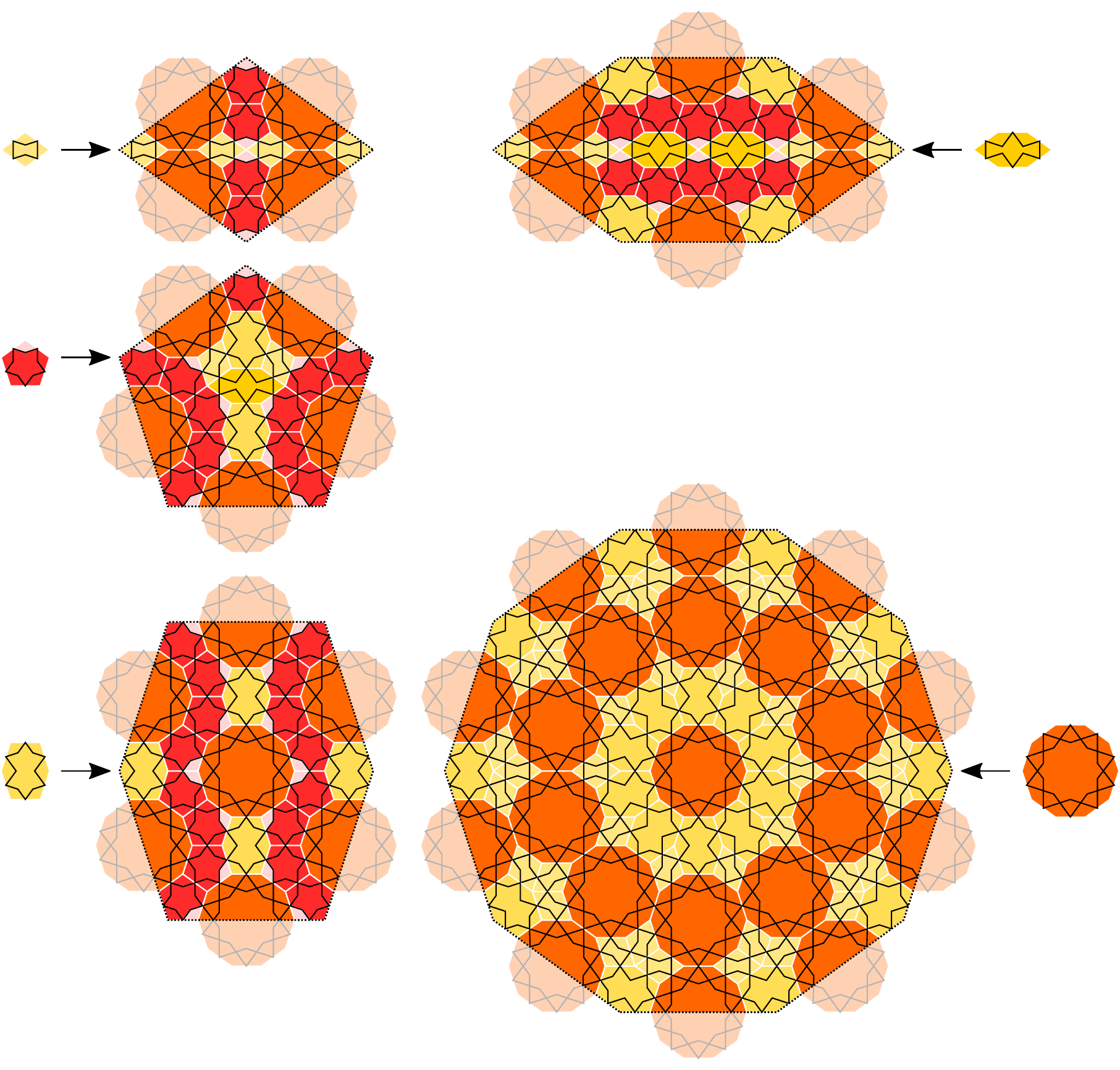}}
\end{center}\caption{\label{fig:Girih5-Topkapi}Another Girih CAST for the case $n=5$,
derived from \citep[Fig. 14, 15]{Cromwell2008} and patterns shown
in the Topkapi Scroll, in detail \citep[Panels 28, 31, 32, 34]{necipoglu1995topkapi}. }
\end{figure}
\begin{figure}[H]
\begin{center}
\resizebox{\textwidth}{!}{%

\includegraphics{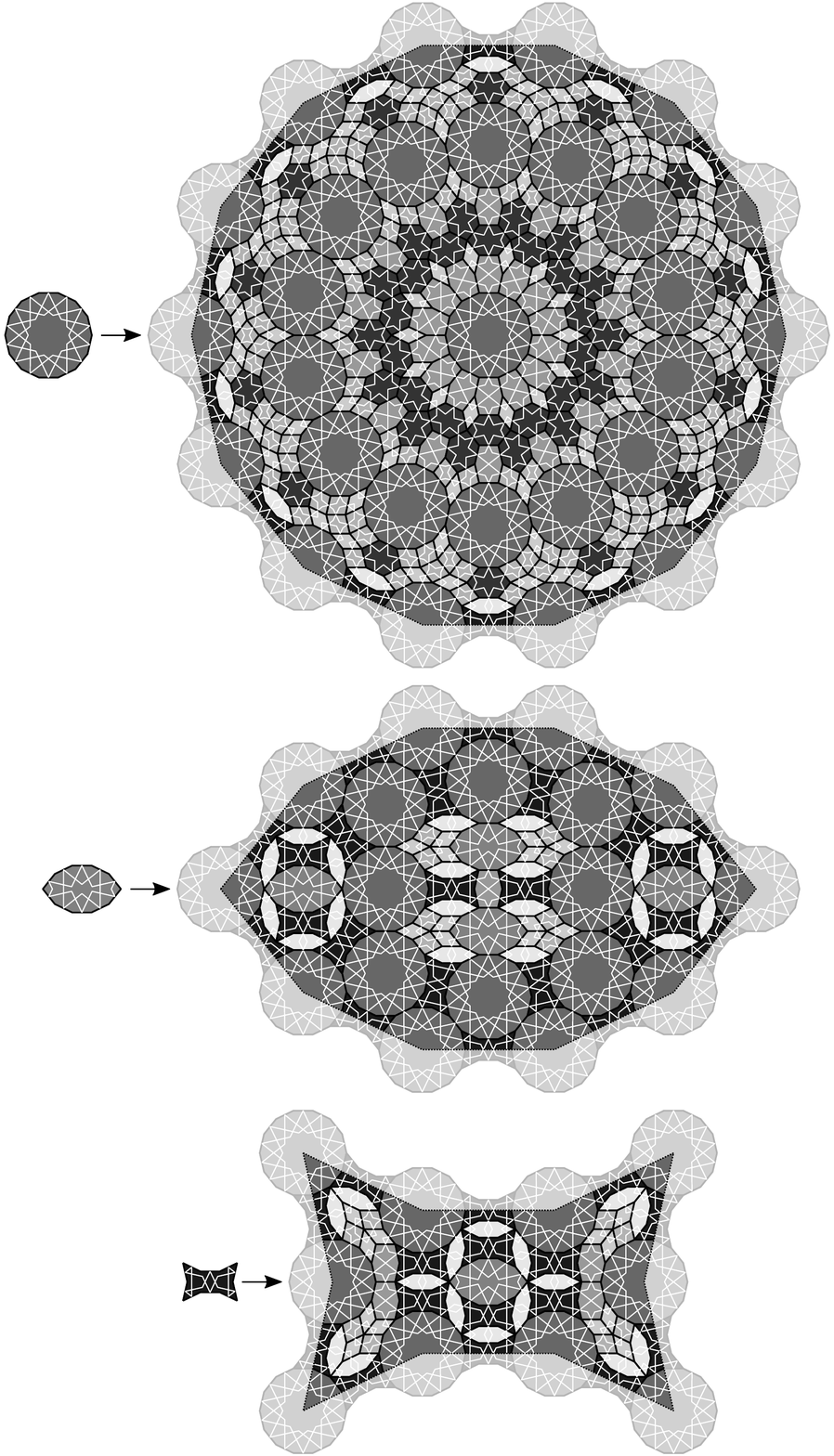}}
\end{center}\caption{\label{fig:Girih7a}Extended Girih CAST for the case $n=7$, substitution
rules part 1}
\end{figure}
\begin{figure}[H]
\begin{center}
\resizebox{\textwidth}{!}{%

\includegraphics{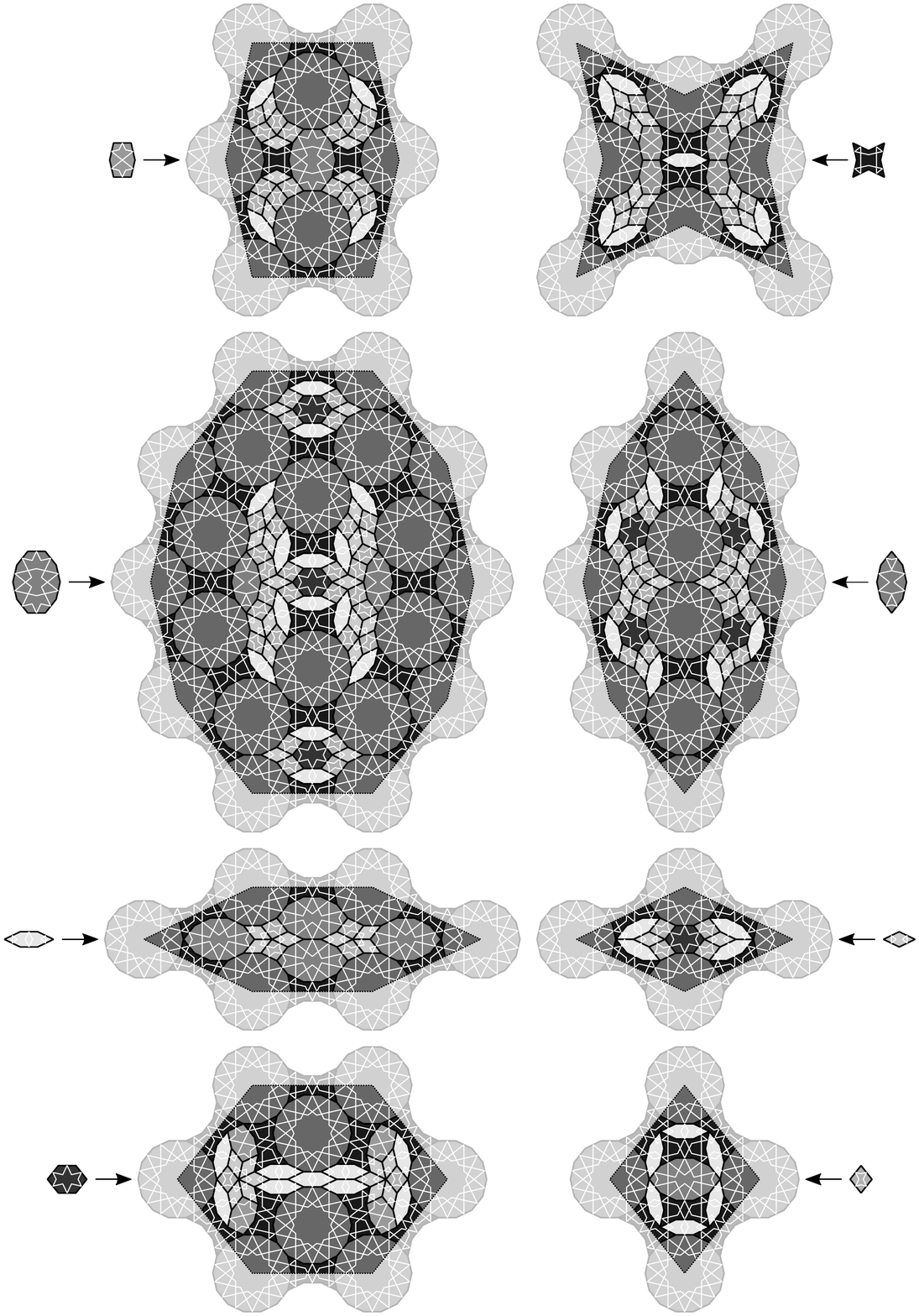}}
\end{center}\caption{\label{fig:Girih7b}Extended Girih CAST for the case $n=7$, substitution
rules part 2}
\end{figure}
\begin{figure}[H]
\begin{center}
\resizebox{\textwidth}{!}{%

\includegraphics{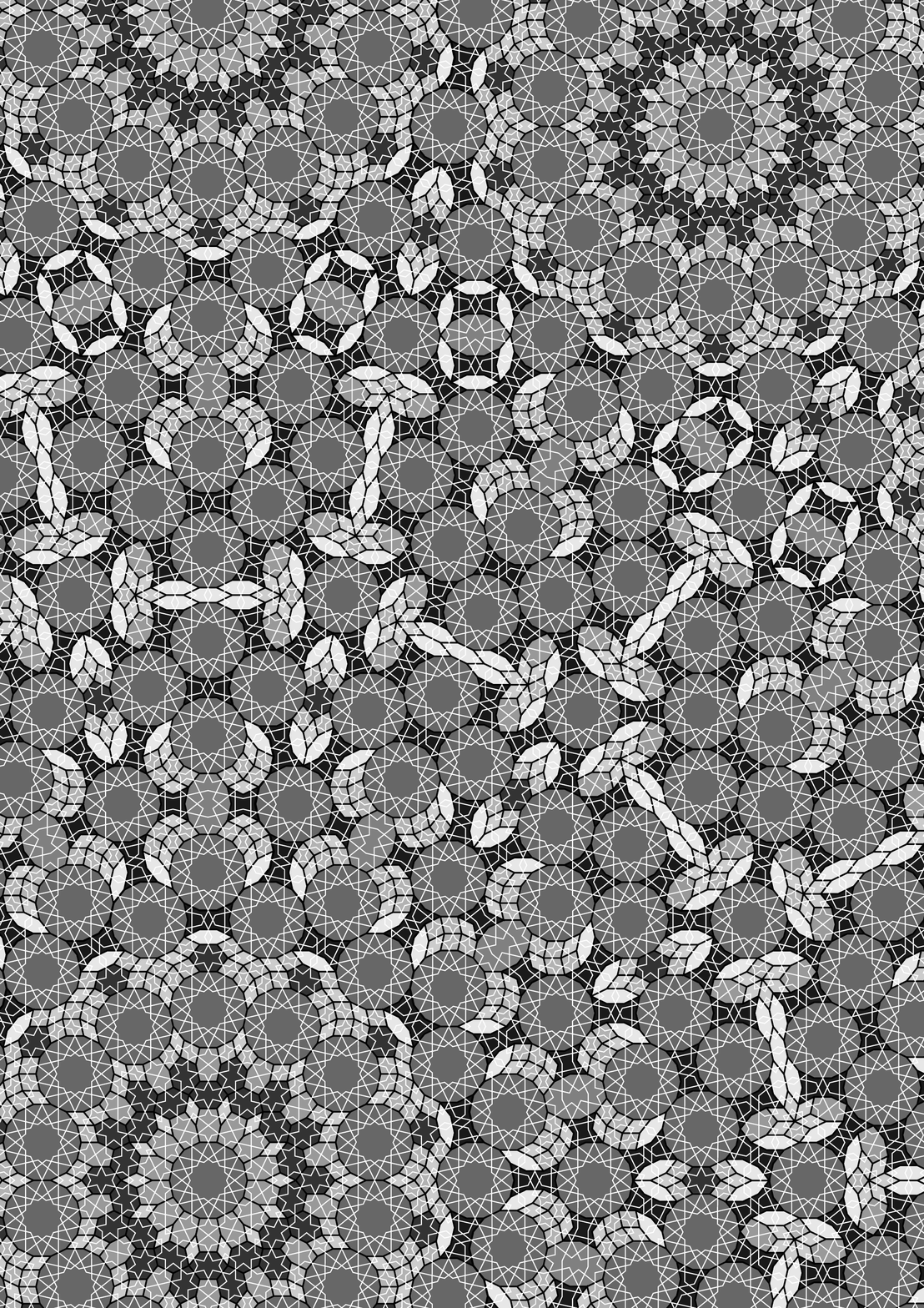}}
\end{center}\caption{\label{fig:Girih7c}Extended Girih CAST for the case $n=7$}
\end{figure}

\section{Summary and Outlook}

Although the motivation behind this article was mainly aesthetic,
some significant results have been achieved. Cyclotomic Aperiodic
Substitution Tilings (CASTs) cover a large number of new and well
known aperiodic substitution tilings as shown in Table~\ref{tab:CASTs}.
The properties of CASTs, in detail their substitution matrices and
their minimal inflation multipliers, can be used as practical starting
point to identify previously unknown solutions. For many cases, such
solutions yield individual dihedral symmetry $D_{n}$ or $D_{2n}$.

The different approaches to identify CASTs have their individual advantages
and disadvantages. The preferable properties as listed in the introduction
section may be complemented by a high frequency of patches with dihedral
symmetry. A promising approach we do not discuss in this article may
be the choice of inflation multipliers which are PV numbers. 

The results in this paper focus strictly on the Euclidean plane so
that Equation~\eqref{eq:lambda-a-square} applies. However, the methods
described herein might be adapted for other cases as well.

Finally several conjectures have been made, which require further
research.

\begin{table}
\caption{\label{tab:CASTs}Inflation multipliers and individual symmetry of
some CASTs}

\begin{center}
\resizebox{\textwidth}{!}{%

\begin{tabular}{|>{\centering}m{0.05\textwidth}|>{\raggedright}m{0.35\textwidth}|>{\centering}m{0.3\textwidth}|>{\centering}m{0.12\textwidth}|>{\raggedright}m{0.4\textwidth}|}
\hline 
n & Name & Inflation\linebreak multiplier & Patches with\linebreak individual \linebreak  symmetry & Reference\tabularnewline
\hline 
5 & Penrose & $\mu_{5,2}$ & $D_{5}$ & \citep{Penrose1974,Gardner1977,Penrose1979,HFonl}\linebreak \citep[Ch. 10.3]{Grunbaum:1986:TP:19304}\linebreak \citep[Ch. 6.2]{oro38933}\tabularnewline
\hline 
7 & Danzer's 7-fold variant & $\mu_{7,2}$ & $D_{7}$ & \citep[Fig. 1, Sec. 3, 2nd matrix]{Nischke1996}\linebreak  Herein
Fig.~\ref{fig:Nischke-Danzer-Mu2}\tabularnewline
\hline 
7 & Danzer's 7-fold variant & $\mu_{7,3}$ & $D_{1}$ & \citep[credited to L. Danzer]{HFonl} \tabularnewline
\cline{4-5} 
 & (two variants) &  & $D_{7}$ & \citep[Fig. 11]{Nischke1996}\tabularnewline
\hline 
7 & Math Pages 7-fold & $\mu_{7,3}$ & - & \citep{MathPages1}\tabularnewline
\hline 
9 & Math Pages 9-fold & $\mu_{9,4}$ & $D_{9}$ & \citep{MathPages1}\tabularnewline
\hline 
5 & Lançon-Billard / Binary & $\sqrt{\mu_{5,2}+2}$ & - & \citep{LanconF.1988,C.Godreche1992,HFonl}\linebreak \citep[Ch. 6.5.1]{oro38933}\tabularnewline
\hline 
6 & Shield & $\sqrt{\mu_{6,2}+2}$ & - & \citep{Gaehler1988,HFonl,Niizeki1987}\linebreak \citep[Ch. 6.3.2]{oro38933}\tabularnewline
\hline 
4 & Ammann-Beenker & $\mu_{4,2}+1$ & $D_{8}$ & \citep{Beenker,Socolar1989,journals/dcg/AmmannGS92,HFonl}\linebreak \citep[Ch. 10.4]{Grunbaum:1986:TP:19304}\citep[Ch. 6.1]{oro38933}\tabularnewline
\hline 
5 & Tie and Navette / Bowtie-Hexagon 1 & $\mu_{5,2}+1$ & - & \citep{LUCK1990832,Lueck1994139}\linebreak \citep[Sec. 8.2, Fig. 8.3]{scheffer1998}\linebreak \citep{HFonl}\linebreak \citep[Fig. 25]{Cromwell2008}\tabularnewline
\hline 
5 & Bowtie-Hexagon-Decagon 1 & $\mu_{5,2}+1$ & $C_{5}$ & \citep[credited to L. Andritz]{HFonl}\tabularnewline
\hline 
7 & Danzer's 7-fold\linebreak (two variants) & $\mu_{7,2}+1$ & $D_{1}$ & \citep[Fig. 12]{Nischke1996}\citep{HFonl}\linebreak \citep[Ch. 6.5.2]{oro38933}\tabularnewline
\hline 
7 & Franco-Ferreira-da-Silva 7-fold & $\mu_{7,2}+1$  & $D_{7}$ & \citep{Franco1994}\tabularnewline
\hline 
7 & Maloney's 7-fold & $\mu_{7,2}+1$  & $D_{7}$ & \citep{HFonl}\linebreak \citep[Fig. 9]{2014arXiv1404.5193G}\tabularnewline
\hline 
7 & \multirow{3}{0.35\textwidth}{Cyclotomic Trapezoids} &  &  & \multirow{3}{0.4\textwidth}{\citep{frettloeh1998,HFonl}}\tabularnewline
\cline{1-1} 
9 &  & $\mu_{n,2}+1$ & $D_{1}$ & \tabularnewline
\cline{1-1} 
11 &  &  &  & \tabularnewline
\hline 
4 & Watanabe-Ito-Soma 8-fold & $\mu_{4,2}+2$ & $D_{8}$ & \citep{Watanabe1986,Watanabe1987,Watanabe1995,HFonl}\tabularnewline
\hline 
4 & \multirow{3}{0.35\textwidth}{Generalized Goodman-Strauss rhomb} &  & $D_{1}$ & \multirow{3}{0.4\textwidth}{\citep{journals/dcg/Harriss05}\linebreak \citep[credited to C. Goodman-Strauss and E. O. Harris]{HFonl}}\tabularnewline
\cline{1-1} \cline{4-4} 
5 &  & $\mu_{n,2}+2$ & $C_{5}$,$D_{1}$ & \tabularnewline
\cline{1-1} \cline{4-4} 
$\geq6$ &  &  & $D_{1}$ & \tabularnewline
\hline 
6 & Watanabe-Soma-Ito 12-fold\linebreak (variants) & $\mu_{6,2}+2$ & $D_{12}$ & \citep{Watanabe1995,HFonl}\tabularnewline
\hline 
6 & Socolar & $\mu_{6,2}+2$ & $D_{2}$ & \citep{Socolar1989,HFonl,Niizeki1988}\tabularnewline
\hline 
6 & Stampfli-Gähler / Ship & $\mu_{6,2}+2$ & $D_{12}$ & \citep{STAMPFLI1986,Gaehler1988,BenAbraham2001}\tabularnewline
\hline 
6 & Square Triangle & $\mu_{6,2}+2$ & $D_{6}$ & \citep{Hermisson97aguide,1999math.ph...1014B,Baake2002,Frettloeh2011}\linebreak \citep[credited to M. Schlottmann]{HFonl}\linebreak \citep[Ch. 6.3.1]{oro38933}\tabularnewline
\hline 
5 & Cromwell & $\mu_{5,2}+3$ & $D_{10}$ & \citep[Fig. 12, 13]{Cromwell2008}\tabularnewline
\hline 
5 & Topkapi Scroll & $2\mu_{5,2}+2$ & $D_{10}$ & Herein Fig.~\ref{fig:Girih5-Topkapi}, derived from \citep[Fig. 14, 15]{Cromwell2008}
and patterns shown in the Topkapi Scroll, in detail \citep[Panels 28, 31, 32, 34]{necipoglu1995topkapi} \tabularnewline
\hline 
5 & Bowtie-Hexagon-Decagon 2 & $2\mu_{5,2}+3$ & $C_{5}$ & \citep[credited to L. Andritz]{HFonl}\tabularnewline
\hline 
5 & Bowtie-Hexagon-Decagon 3 & $3\mu_{5,2}+2$ & $C{}_{5}$ & \citep[credited to L. Andritz]{HFonl}\tabularnewline
\hline 
5 & Darb-i Imam Shrine & $4\mu_{5,2}+2$  & $D_{10}$ & \citep{Lu2007,Lu2007a,bridges2008:297}\linebreak \citep[Fig. 21]{Cromwell2008}\tabularnewline
\hline 
7 & Franco's 7-fold & $\mu_{7,3}+\mu_{7,2}+1$ & $D_{7}$ & \citep{Franco1993}\tabularnewline
\hline 
7 & Gähler-Kwan-Maloney 7-fold & $\mu_{7,3}+\mu_{7,2}+1$ & $D_{7}$ & \citep[Fig. 10]{2014arXiv1404.5193G}\tabularnewline
\hline 
7 & Socolar's 7-fold & $\mu_{7,3}+2\mu_{7,2}+1$ & $D_{7}$ & \citep[credited to J. Socolar]{HFonl}\tabularnewline
\hline 
9 & Franco-da-Silva-Inácio 9-fold & $\mu_{9,4}+\mu_{9,3}+\mu{}_{9,2}+1$ & $D_{9}$ & \citep{Franco1996}\tabularnewline
\hline 
11 & Maloney's 11-fold & $\mu_{11,6}+2\mu_{11,5}+2\mu_{11,4}+2\mu_{11,3}+2\mu_{11,2}+1$ & $D_{11}$ & \citep{Maloney2014,DBLP:journals/dmtcs/Maloney15}\tabularnewline
\hline 
\end{tabular}

}
\end{center}
\end{table}

\section*{Acknowledgment}

The author dedicates this paper to his daughter Lili and his parents
Marita and Herbert. He would like to thank M. Baake, D. Frettlöh,
U. Grimm, R. Lück and C. Mayr for their support and encouragement.

The author is aware that this article might not meet everyone's standards
and expectations regarding a mathematical scientific paper. He kindly
asks for the readers indulgence and hopes that the content and the
sketched ideas herein are helpful for further research despite possible
formal issues. The recent publications of G. Maloney \citep{Maloney2014,DBLP:journals/dmtcs/Maloney15},
J. Kari and M. Rissanen \citep{2015arXiv151201402K,Kari2016} and
T. Hibma \citep{2015arXiv150902053H,THonl}, who found similar results,
demanded a response on short notice.

\clearpage{}

\bibliographystyle{halpha}
\bibliography{girih7-reference__}

\end{document}